\documentclass[12pt]{amsart}
\usepackage{graphicx}
\usepackage{amssymb,amstext,amsfonts,amscd}

\newtheorem{theorem}{Theorem}
\newtheorem{proposition}[theorem]{Proposition}
\newtheorem{lemma}[theorem]{Lemma}

\theoremstyle{definition} 
\newtheorem{definition}[theorem]{Definition}
\newtheorem{remark}[theorem]{Remark}
\newtheorem{problem}{Problem}

\numberwithin{equation}{section}
\numberwithin{theorem}{section}

\newcommand{\bz}{\mathbb{Z}}
\newcommand{\br}{\mathbb{R}}
\newcommand{\bc}{\mathbb{C}}
\newcommand{\bp}{\mathbb{P}}

\newcommand{\bff}{\mathbb{F}}   
\newcommand{\bl}{\mathbb{L}}    
\newcommand{\bu}{\mathbb{U}}    
\newcommand{\M}{\mathcal{M}}
\newcommand{\E}{\mathcal{E}}    

\newcommand{\Fa}{\mathcal{F}}

\newcommand{\La}{\mathcal{L}}

\newcommand{\da}{\mathcal{D}}

\newcommand{\bv}{\mathbf{v}}

\newcommand{\rk}{\mathop\mathrm{rank}}
\newcommand{\Pic}{\mathop\mathrm{Pic}}
\newcommand{\BL}{\mathrm{bl}}
\newcommand{\mmod}{\ \mathrm{mod}\ }
\newcommand{\id}{\mathrm{id}}
\newcommand{\wvarphi}{{\widetilde{\varphi}}}

\title{On real anti-bicanonical curves with one double point on the $4$-th real Hirzebruch surface}

\author{Sachiko Saito}

\address{Department of Mathematics Education,\ Asahikawa Campus,\\
Hokkaido University of Education,\ Asahikawa, JAPAN}
\email{saito.sachiko@a.hokkyodai.ac.jp}

\begin{document}
\pagestyle{plain}

\begin{abstract}
We list up all the candidates for 
the real isotopy types of real anti-bicanonical curves with one real nondegenerate double point on the $4$-th real Hirzebruch surface $\br \bff_4$ 
by enumerating the connected components of the moduli space of 
real $2$-elementary K3 surfaces of type $(S,\theta) \cong ((3,1,1),- \id)$. 
We also list up all the candidates for 
the non-increasing simplest degenerations of real nonsingular anti-bicanonical curves on $\br \bff_4$. 
We find an interesting correspondence between 
the real isotopy types of real anti-bicanonical curves with one real nondegenerate double point on $\br \bff_4$ 
and the non-increasing simplest degenerations of real nonsingular anti-bicanonical curves on $\br \bff_4$. 
This correspondence is very similar to the one provided 
by the rigid isotopic classification of real sextic curves on $\br \bp^2$ with one real nondegenerate double point by I. Itenberg. 
\end{abstract}

\maketitle

\tableofcontents

\vspace{2cm}

-----------------------------------------------------------------------------------------------------------------------

{\footnotesize 
partially supported by JSPS Grant-in-Aid for Challenging Exploratory Research 25610001

2010 {\it AMS Mathematics Subject Classification}\ \ 14J28, 14P25, 14J10.
}

\newpage
\setcounter{section}{0}

\section{Introduction}

This paper is a continuation of \cite{NikulinSaito05} and \cite{NikulinSaito07}. 
In \cite{NikulinSaito05} the moduli spaces of ``($\da\br$)-nondegenerate" real K3 surfaces 
with non-symplectic holomorphic involutions (namely, real $2$-elementary K3 surfaces) are formulated 
and 
it is shown that the connected components of such a moduli space 
are in {\bf one to one correspondence} with 
the isometry classes of integral involutions of the K3 lattice of certain type 
(see Theorem \ref{theorem2005moduli} below and \cite{NikulinSaito05} for more precise statements). 
As its applications, we obtain 
the real isotopic classifications (more precisely, deformation classifications) of 
real nonsingular sextic curves on the real projective plane $\br \bp^2$, 
real nonsingular curves of bidegree $(4,4)$ on the real $\bp^1 \times \bp^1$ (hyperboloid, ellipsoid), and 
real nonsingular anti-bicanonical curves on the real Hirzebruch surfaces $\br \bff_m$ ($m=1,\,4$). 
Here the $m$-th Hirzebruch surface $\bff_m \ (m \geq 0)$ means the ruled surface over $\bp^1$ having an exceptional section $s$ with $s^2=-m$. 
$\bff_4$ has $2$ real structures (anti-holomorphic involutions). 
The real part of $\bff_4$ is homeomorphic to the $2$-dimensional torus or the empty set. 
On the other hand, 
$\bff_3$ has a unique real structure and its real part is homeomorphic to the Klein's bottle (see \cite{DIK2000}). 

We say a complex curve $A$ on a real Hirzebruch surface is {\bf real} 
if its anti-holomorphic involution can be restricted to the curve $A$. 
We say two real curves $\br A,\ \br A'$ on a real nonsingular surface $\br B$ are {\bf real isotopic} if 
there exists a continuous map $\Phi : \br B \times [0,1] \to \br B$ (a ``real isotopy from $\br A$ to $\br A'$") such that 
$\Phi_t := \Phi(\ ,t) : \br B \to \br B$ is a homeomorphism for any $t \in [0,1]$, $\Phi_0 = \id_{\br B}$, and $\Phi_1(\br A) = \br A'$. 
Moreover, two real curves $\br A,\ \br A'$ in a fixed class on a real nonsingular surface $\br B$ are {\bf rigidly isotopic} if 
there exists a real isotopy $\Phi : \br B \times [0,1] \to \br B$ from $\br A$ to $\br A'$ 
such that $\Phi_t(\br A)$ is contained in the same class for any $t \in [0,1]$. 

Using the same method as above, we obtain 
the real isotopic classifications of 
real nonsingular anti-bicanonical curves on the real Hirzebruch surfaces $\br \bff_2$ and $\br \bff_3$ 
(see also Theorem \ref{moduli311} and Remark \ref{F3} of this paper) in \cite{NikulinSaito07}. 
Especially, all the connected components of the moduli space of 
real $2$-elementary K3 surfaces of type $(S,\theta) \cong ((3,1,1),- \id)$, which are defined below, 
are enumerated in \cite{NikulinSaito07}. 
Here we should remark that 
any real $2$-elementary K3 surfaces of type $(S,\theta) \cong ((3,1,1),- \id)$ are 
($\da\br$)-{\bf nondegenerate} in the sense of \cite{NikulinSaito05}. 

Any $2$-elementary K3 surface $(X,\tau)$ of type $S \cong (3,1,1)$ has an elliptic fibration and 
a unique reducible fiber $E+F$, where the curves $E$ and $F$ are defined in \cite{NikulinSaito07} (see also Subsection \ref{RealK3-311} of this paper). 
There are two types of $F$, which are reducible and irreducible. 
Let $A$ be the fixed point set (curve) of $\tau$ on $X$, and $\pi: X \to Y:= X/\tau$ be the quotient map. 
$A$ intersects $\pi(E)$ at a single point. 
If we contract the curve $\pi(E)$ on $Y$ to a point, then we get the $3$-th Hirzebruch surface $\bff_3$, and 
the image of the curve $A$ is a real {\bf nonsingular} anti-bicanonical curve on $\bff_3$. 
This enables us to enumerate up (see \cite{NikulinSaito07}) 
all the real isotopy types of real nonsingular anti-bicanonical curves on $\br \bff_3$. 

However, on the other hand, 
if we contract the curve $\pi(F)$ on $Y$ to a point, then we get the $4$-th Hirzebruch surface $\bff_4$, and 
the image of the curve $A$ is real anti-bicanonical and has {\bf one real double point}. 
Even though a real $2$-elementary K3 surface of type $(S,\theta) \cong ((3,1,1),- \id)$ is always ($\da\br$)-{\bf nondegenerate}, 
the double point is possibly {\bf degenerate}, namely, a real cusp point. 
Thus, the main difficulty of the real isotopic classification of real anti-bicanonical curves on $\bff_4$ with one real double point 
is that 
the connected components (equivalently, the isometry classes of integral involutions of the K3 lattice) 
of the moduli space (\cite{NikulinSaito05}) of 
real $2$-elementary K3 surfaces of type $(S,\theta) \cong ((3,1,1),- \id)$ 
cannot distinguish degenerate and nondegenerate double points, 
equivalently, 
reducible $F$ and irreducible $F$. 
Hence, 
the connected components cannot distinguish 
the topological types (node, cusp, or isolated point) of the real double points (Remark \ref{realizability1}). 
This problem was left to the readers in \cite{NikulinSaito07} (see Remark 8). 

Thus, the aim of this paper is:\ 
{\it Classify the real isotopy types of 
real anti-bicanonical curves with one real nondegenerate double point on $\br \bff_4$.} 

First we list up all the candidates for the real isotopy types of 
real anti-bicanonical curves with one real nondegenerate double point on $\br \bff_4$ (Theorem \ref{isotopy-F4-double}) 
using some well-known topological interpretations (\cite{NikulinSaito05}, \cite{NikulinSaito07}) 
of some arithmetic invariants of integral involutions of the K3 lattice. 

Unfortunately, the realizability of each real isotopy type listed in Theorem \ref{isotopy-F4-double} 
has not been resolved in this paper. 
We only know that at least one of real isotopy types with nondegenerate double points 
can be realized for {\bf each} isometry class (see Remark \ref{realizability1}). 

In order to distinguish the real isotopy types, 
we should remove real $2$-elementary K3 surfaces which yield anti-bicanonical curves with degenerate double points on $\br \bff_4$ 
from the moduli space (period domain) in the sense of \cite{NikulinSaito05}. 
We follow Itenberg's argument (\cite{Itenberg92},\cite{Itenberg94},\cite{Itenberg95}) 
for the rigid isotopic classification of real sextic curves on $\br \bp^2$ with one nondegenerate double point. 
Lemma \ref{criterion} provides a sufficient condition for the double point to be non-degenerate. 

Moreover, according to his papers \cite{Itenberg92} and \cite{Itenberg94}, 
real curves of degree $6$ on $\br \bp^2$ with one nondegenerate double point 
are obtained by 
``non-increasing simplest degenerations" of real nonsingular curves of degree $6$ on $\br \bp^2$. 
Hence, we next list up the candidates for the non-increasing simplest degenerations of 
real nonsingular anti-bicanonical curves on $\br \bff_4$ (Theorem \ref{degenerations-F4-re-arran}). 

Then, we find an obvious interesting correspondence between 
the real isotopy types of curves with one real nondegenerate double point on $\br \bff_4$ (Theorem \ref{isotopy-F4-double}) 
and the non-increasing simplest degenerations of nonsingular curves on $\br \bff_4$ (Theorem \ref{degenerations-F4-re-arran}). 
See Remark \ref{interesting-correspondence}. 
We also get some similar properties (Lemma \ref{degene-from-nonsing}) 
to the degenerations of real nonsingular curves of degree $6$ on $\br \bp^2$. 

In the final section \ref{period domain and problems}, we review and confirm the periods of 
marked real $2$-elementary K3 surfaces $((X,\tau,\varphi),\ \alpha)$ of type $(S,\theta)$ 
satisfying $\alpha \circ \varphi_* \circ \alpha^{-1} = \psi$ for a fixed integral involution $\psi$, 
and give some further problems which are inspired by the argument in \cite{Itenberg92} and \cite{Itenberg94}. 

\medskip

{\bf Acknowledgments.}\ The author would like to thank the referee for his careful reading and valuable advice, 
Professor Ilia Itenberg for his inspiring papers, and 
Professor Viacheslav Nikulin for his constant encouragement. 

\section{Real $2$-elementary K3 surfaces of type $(S,\theta) \cong ((3,1,1),- \id)$}

\subsection{Real $2$-elementary K3 surfaces} \label{real_2-elementary K3}

\begin{definition}[Real $2$-elementary K3 surface]
A triple $(X,\tau,\varphi)$ is called 
a {\bf real K3 surface with non-symplectic (holomorphic) involution} 
(or {\bf real $2$-elementary K3 surface}) 
if\\
\ \ {\rm (1)}\ $(X,\tau)$ is a K3 surface $X$ with a non-symplectic holomorphic involution $\tau$, i.e., 
``$2$-elementary K3 surface" (\cite{Nikulin81}).\\
\ \ {\rm (2)}\ $\varphi$ is an anti-holomorphic involution on $X$.\\
\ \ {\rm (3)}\ $\varphi \circ \tau = \tau \circ \varphi$
\end{definition}

Note that any K3 surface with a non-symplectic holomorphic involution is algebraic. 

For a real K3 surface with non-symplectic involution $(X,\tau,\varphi)$, 
we call $\wvarphi := \tau \circ \varphi = \varphi \circ \tau$ 
the {\bf related (anti-holomorphic) involution} of $\varphi$ (\cite{NikulinSaito05}). 
The triple $(X,\tau,\tau \circ \varphi)$ is also 
a real K3 surface with non-symplectic involution. 

For a $2$-elementary K3 surface $(X,\tau)$, we denote by 
$${H_2}_+(X, \bz)$$
the fixed part of $\tau_* : H_2(X, \bz) \to H_2(X, \bz)$. 
Note that ${H_2}_+(X, \bz) \subset N(X)$, where $N(X)$ is the Picard lattice of $X$. 

We fix an even unimodular lattice 
$$\bl_{K3}$$
of signature $(3,19)$. 
The isometry class of such lattices is unique (the K3 lattice). 

Let $\alpha : H_2(X, \bz) \to \bl_{K3}$ be an isometry (marking). 
If we temporary set $S := \alpha({H_2}_+(X, \bz))$, 
then $S$ is a primitive (,i.e., $\bl_{K3}/S$ is free,) 
hyperbolic (i.e., $S$ is of signature $(1,\rk S -1)$) 
$2$-elementary (i.e., $S^\ast /S \cong (\bz/2\bz)^a$ for some nonnegative integer $a$) 
sublattice of $\bl_{K3}$. 

\bigskip

Now let 
$$S \ \ \ (\subset \bl_{K3})$$
be a primitive hyperbolic $2$-elementary sublattice of the K3 lattice $\bl_{K3}$. 

\begin{definition}
We set $r(S) := \rk S$.\ \ \ The non-negative integer $a(S)$ is defined by the equality
$$S^\ast /S \cong (\bz/2\bz)^{a(S)}.$$
We define
$$
\delta (S) := \left\{
\begin{array}{cl}
0 &\ \ \ \mbox{if}\ z \cdot \sigma (z) \equiv 0 \mmod 2 \ \ (\forall z \in \bl_{K3})\\
1 &\ \ \ \mbox{otherwise,}
\end{array}
\right.
$$
where we define 
$$\sigma : \bl_{K3} \to \bl_{K3}$$ to be 
the unique integral involution whose fixed part is $S$. 
\end{definition}

It is known that the triplet 
$$(r(S),a(S),\delta(S))$$
determines the isometry class of the lattice $S$ (\cite{Nikulin81}). 

If $S$ and $S^\prime$ are primitive hyperbolic $2$-elementary sublattices of the K3 lattice $\bl_{K3}$, 
and $S$ is isometric to $S^\prime$; 
then there exists an automorphism $f$ of $\bl_{K3}$ such that $f(S^\prime) = S$ (\cite{AlexeevNikulin2006}, \cite{Nikulin79}). 
Hence, if $(X,\tau)$ and $(X^\prime,\tau^\prime)$ are two $2$-elementary K3 surfaces, 
$\alpha : H_2(X, \bz) \to \bl_{K3}$ is an isometry, and 
${H_2}_+(X, \bz)$ is isometric to ${H_2}_+(X^\prime, \bz)$; 
then there exists an isometry (marking) $\alpha^\prime : H_2(X^\prime, \bz) \to \bl_{K3}$ such that 
$\alpha^\prime({H_2}_+(X^\prime, \bz)) = \alpha({H_2}_+(X, \bz))$. 
Thus, only the isometry class of ${H_2}_+(X, \bz)$ is essential for $2$-elementary K3 surfaces $(X,\tau)$. 
Henceforth, we often fix an isometry class, equivalently, the invariants $(r(S),a(S),\delta(S))$ 
instead of fixing a particular sublattice. 

\medskip

We quote the following formulation from \cite{NikulinSaito05}. 
We additionally fix a half-cone 
$V^+(S)$ of the cone $V(S):= \{x\in S \otimes \br\ |\ x^2>0\}$. 
We also fix a fundamental chamber $\M \subset V^+(S)$ for the group $W^{(-2)}(S)$ generated by reflections in all elements with square $(-2)$ in $S$. 
This is equivalent to fixing a fundamental subdivision 
$$\Delta(S)=\Delta(S)_+\cup -\Delta(S)_+$$
of all elements with square $-2$ in $S$. 
$\M$ and $\Delta(S)_+$ define each other by the condition $\M \cdot \Delta(S)_+ \ge 0$. 

\medskip

Let $(X,\tau)$ be a $2$-elementary K3 surface, and $\alpha : H_2(X, \bz) \to \bl_{K3}$ be a marking 
such that $\alpha({H_2}_+(X, \bz)) = S$. 
We always assume that $\alpha_{\br}^{-1}(V^+(S))$ contains a hyperplane section of $X$ 
and the set $\alpha^{-1}(\Delta(S)_+)$ contains only classes of effective curves of $X$. 

\bigskip

Now let $\theta$ be an integral involution of $S$. 

\begin{definition}[the action of $\varphi$ on ${H_2}_+(X, \bz)$]
A real $2$-elementary K3 surface $(X,\tau,\varphi)$ is called that of {\bf type $(S,\theta)$} 
if there exists an isometry (marking) 
$$\alpha : H_2(X, \bz) \cong \bl_{K3}$$
such that 
$\alpha({H_2}_+(X, \bz)) = S$ and the following diagram commutes:
$$
\begin{CD}
{H_2}_+(X, \bz) @> {\alpha}>>S\\
@V{\varphi_*}VV @VV{\theta}V\\
{H_2}_+(X, \bz) @> {\alpha}>>S .
\end{CD}
$$
\end{definition}

\begin{definition}[marked real $2$-elementary K3 surfaces] \label{marked_real_K3}
A pair $((X,\tau,\varphi),\ \alpha)$ of a real $2$-elementary K3 surface $(X,\tau,\varphi)$ of type $(S,\theta)$ and an isometry (marking) 
$\alpha : H_2(X, \bz) \cong \bl_{K3}$ 
such that 
\begin{itemize}
\item $\alpha({H_2}_+(X, \bz)) = S$,\\
\item $\alpha \circ \varphi_* = \theta \circ \alpha \ \text{on} \ {H_2}_+(X, \bz)$,\\
\item $\alpha_{\br}^{-1}(V^+(S))$ contains a hyperplane section of $X$ and\\
\item the set $\alpha^{-1}(\Delta(S)_+)$ contains only classes of effective curves of $X$ 
\end{itemize}
is called 
a {\bf marked real $2$-elementary K3 surface of type $(S,\theta)$}.
\end{definition}

\medskip

Note that we have 
$\theta(V^+(S)) = -V^+(S)$
and
$\theta(\Delta(S)_+) = -\Delta(S)_+$
for a marked real $2$-elementary K3 surface of type $(S,\theta)$. 

\bigskip

\begin{definition}[Integral involution $\psi$ of $\bl_{K3}$ of type $(S,\theta)$]
Let $S$ be a hyperbolic $2$-elementary sublattice of $\bl_{K3}$, 
$\theta : S \to S$ be an integral involution of the lattice $S$, and 
$\psi : \bl_{K3} \to \bl_{K3}$ be an integral involution of the lattice $\bl_{K3}$ such that the following diagram commutes:
$$
\begin{array}{rcl}
         S           & \subset &  \bl_{K3}          \\
\theta \ \downarrow  &         &  \downarrow \ \psi \\
         S           & \subset &  \bl_{K3} .
\end{array}
$$
We say such a pair $(\bl_{K3},\psi)$ (or $\psi$) 
an {\bf integral involution of $\bl_{K3}$ of type $(S,\theta)$}. 
\end{definition}

Let $((X,\tau,\varphi),\ \alpha)$ be a marked real $2$-elementary K3 surface of type $(S,\theta)$. 
If we set 
$$\psi := \alpha \circ \varphi_* \circ \alpha^{-1} : \bl_{K3} \to \bl_{K3},$$
then we have $\psi(S) = S$, and $\psi (x) = \theta (x)$ for every $x \in S$ 
because $\alpha^{-1} (x) \in {H_2}_+(X, \bz)$. 
Hence, $(\bl_{K3},\psi)$ is an integral involution of $\bl_{K3}$ of type $(S,\theta)$.

\begin{definition}[the associated integral involution]
We call the integral involution $\psi$ of $\bl_{K3}$ of type $(S,\theta)$ 
{\bf the associated integral involution} of $\bl_{K3}$ with a marked real $2$-elementary K3 surface $((X,\tau,\varphi),\ \alpha)$ of type $(S,\theta)$ 
if the following diagram commutes:
$$
\begin{CD}
H_2(X, \bz) @>{\alpha}>>\bl_{K3} \\
@V{\varphi_*}VV @VV{\psi}V\\
H_2(X, \bz) @>{\alpha}>>\bl_{K3} .
\end{CD}
$$
\end{definition}

Note that the fixed part $\bl_{K3}^\psi$ of $\psi$ is hyperbolic for any associated integral involution of $\bl_{K3}$. 

\medskip

\begin{definition}[($\da\br$)-nondegenerate, \cite{NikulinSaito05}]\label{da-degenerate}
$x \in N(X)\otimes \br$ is {\bf nef} if $x\not=0$ and $x\cdot C\ge 0$ for any effective curve on $X$. 
We say that a $2$-elementary K3 surface $(X,\tau)$ of type $S$ is {\bf ($\da$)-degenerate} if 
there exists $h \in \M$ such that $h$ is not nef. 
This is equivalent to the existence of an exceptional curve (i.e., irreducible and having negative self-intersection) 
with square $-2$ 
on the quotient surface $Y:=X/\tau$. 
This is also equivalent to have an element $\delta \in N(X)$ with $\delta^2=-2$ such that 
$\delta=(\delta_1+\delta_2)/2$ where $\delta_1\in S$, $\delta_2\in S^\perp_{N(X)}$ and $\delta_1^2=\delta_2^2=-4$. 

A real $2$-elementary K3 surface $(X,\tau,$ $\varphi)$ of type $(S,\theta)$ is {\bf ($\da\br$)-degenerate} 
if there exists a real element $h\in S_-\cap \M$ which is not nef, where we set 
$S_{\pm} := \{ x\in S \,|\, \theta(x) = \pm x \}$. 
This is equivalent to 
have an element $\delta \in N(X)$ with $\delta^2=-2$ such that 
$\delta=(\delta_1+\delta_2)/2$ where $\delta_1\in S$, $\delta_2\in S^\perp_{N(X)}$ 
and $\delta_1^2=\delta_2^2=-4$, and 
$\delta_1$ must be orthogonal to an element $h\in S_-\cap \text{int}(\M)$ with $h^2>0$. 
Here $\text{int}(\M)$ denote the interior part of $\M$, 
i. e., the polyhedron $\M$ without its faces. 
\end{definition}


Let $\Delta(S,L)^{(-4)}$ be the set of all elements $\delta_1$ in $S$ such that $\delta_1^2=-4$ 
and there exists $\delta_2\in S^\perp_L$ such that $(\delta_2)^2=-4$ and 
$\delta=(\delta_1+\delta_2)/2\in L$. 
Let $W^{(-4)}(S,L)\subset O(S)$ 
be the group
generated by reflections in all roots from $\Delta(S,L)^{(-4)}$,
and $W^{(-4)}(S,L)_\M$ be the stabilizer subgroup of $\M$ 
in $W^{(-4)}(S,L)$. 
Let 
$$G$$
be the subgroup generated by reflections $s_{\delta_1}$ in 
all elements $\delta_1\in \Delta(S,L)^{(-4)}$ 
which are contained either in $S_+$ or in $S_-$ and 
satisfy $s_{\delta_1}(\M)=\M$. 
(Then $G$ is a subgroup of $W^{(-4)}(S,L)_\M$. )


\begin{definition}[Isometries with respect to the group $G$]
Let $(\bl_{K3},\psi_1)$ and $(\bl_{K3},\psi_2)$ be 
two integral involutions of $\bl_{K3}$ of type $(S,\theta)$. 
An {\bf isometry with respect to the group} $G$ 
from $(\bl_{K3},\psi_1)$ to $(\bl_{K3},\psi_2)$ 
means 
an isometry $f: \bl_{K3} \to \bl_{K3}$ such that $f(S)=S$, ${f|}_S \in G$, and 
the following diagram commutes:
$$
\begin{CD}
\bl_{K3} @>{f}>>\bl_{K3} \\
@VV\psi_1 V@V\psi_2 VV\\
\bl_{K3} @>{f}>>\bl_{K3} .
\end{CD}
$$
\end{definition}

In the above definition, 
remark that $\theta \circ {f|}_S = {f|}_S \circ \theta$ on $S$, hence, 
${f|}_S$ is at least an automorphism of $(S,\theta)$. 
However, we require the condition ``${f|}_S \in G$".

\begin{definition}\label{auto-psi}
We say 
two integral involutions $(\bl_{K3},\psi_1)$ and $(\bl_{K3},\psi_2)$ of type $(S,\theta)$ 
are 
{\bf isometric with respect to the group $G$} 
if there exists an isometry  with respect to the group $G$ 
from $(\bl_{K3},\psi_1)$ to $(\bl_{K3},\psi_2)$.

By an {\bf automorphism} of an integral involution $(\bl_{K3},\psi)$ of type $(S,\theta)$ {\bf with respect to the group $G$} 
we mean 
an isometry with respect to the group $G$ from $(\bl_{K3},\psi)$ to itself.
Namely, 
an isometry $f: \bl_{K3} \to \bl_{K3}$ which satisfies that 
$$\psi \circ f = f \circ \psi,\ f(S)=S \ \text{and} \  {f|}_S \in G.$$
\end{definition}

\begin{definition}[analytically isomorphic with respect to $G$] \label{analytic-iso}
Two marked real $2$-elementary K3 surfaces 
$((X,\tau,\varphi),\ \alpha)$ and 
$((X^\prime,\tau^\prime,\varphi^\prime),\alpha^\prime)$ 
of type $(S,\theta)$ 
are {\bf analytically isomorphic with respect to the group} $G$ 
if 
there exists an analytic isomorphism 
$f : X \to X^\prime$ such that 
$f \circ \tau = \tau^\prime \circ f$, $f \circ \varphi = \varphi^\prime \circ f$ and 
$\alpha^\prime \circ f_* \circ \alpha^{-1}|S \in G $.    
\end{definition}

\medskip

By considering their associated integral involutions, 
we get 
a natural map from the moduli space of marked real $2$-elementary K3 surfaces of type $(S,\theta)$ 
to the set of 
isometry classes with respect to $G$ 
of integral involutions of $\bl_{K3}$ of type $(S,\theta)$ such that the fixed part $\bl_{K3}^\psi$ of $\psi$ is hyperbolic. 

\begin{theorem}[\cite{NikulinSaito05}]\label{theorem2005moduli}
The natural map above gives a bijective correspondence between 
the connected components of the moduli space of ($\da\br$)-nondegenerate 
marked real $2$-elementary K3 surfaces of type $(S,\theta)$ 
and 
the set of 
isometry classes with respect to $G$ 
of integral involutions of $\bl_{K3}$ of type $(S,\theta)$ such that the fixed part $\bl_{K3}^\psi$ of $\psi$ is hyperbolic. 
\end{theorem}

\subsection{Elliptic fibrations of $2$-elementary K3 surfaces of type $S \cong (3,1,1)$} \label{RealK3-311}
We quote some basic facts from \cite{AlexeevNikulin2006}. 
Let $(X,\tau)$ be a $2$-elementary K3 surface of type $S$ with invariants 
$$(r(S),a(S),\delta(S)) = (3,1,1),$$
and 
$A := X^{\tau}$ be  the fixed point set (curve) of $\tau$. 
Then we have 
$$A = A_0 \cup A_1\ \ \mbox{(disjoint union)},$$
where $A_0$ is a nonsingular rational curve ($\cong \bp^1$) with $A_0^2=-2$ 
and \ $A_1$ is a nonsingular curve of genus $9$.
$(X,\tau)$ has a structure of an elliptic pencil $|E+F|$ with its section $A_0$ and 
the unique reducible fiber $E+F$ having the following properties below:
\begin{description}
\item[(i)] $E$ is an irreducible nonsingular rational curve with $E^2=-2$ and $E\cdot A_0=1$.
\item[(ii)] $E\cdot F=2$, \ $F^2=-2$, \ $F\cdot A_0=0$, \ and \ 
$F$ is either an irreducible nonsingular rational curve (type {\bf IIa}), or 
the union of two irreducible nonsingular rational curves $F^\prime$ and $F^{\prime\prime}$ 
which are conjugate by $\tau$ and $F^\prime \cdot F^{\prime\prime} = 1$ (type {\bf IIb}) 
(that is, the reducible fiber $E+F$ corresponds to the extended root system $\widetilde{\mathbb{A}_1}$ or $\widetilde{\mathbb{A}_2}$). 
We have $(F^\prime)^2 = (F^{\prime\prime})^2 = -2$. See also \S 2.4 of \cite{AlexeevNikulin2006}. 
\item[(iii)] The classes $[A_0]$,\ $[E]$ and $[F]$ generate 
the lattice ${H_2}_+(X, \bz) \ (\cong S)$.
Moreover, 
$A_1 \cdot E = 1,\ \ A_1 \cdot F = 2$.
The Gram matrix of the lattice ${H_2}_+(X, \bz)$ 
with respect to the basis $[E],\ [F]$ and $[A_0]$ is as follows:
$$
\begin{array}{cccc}
        & {[E]}   & {[F]} & {[A_0]} \\
{[E]}   &   -2    &       &         \\
{[F]}   &    2    &  -2   &         \\
{[A_0]} &    1    &   0   &   -2
\end{array}
$$
\end{description}

\bigskip

We next consider the quotient complex surface 
$$Y := X/\tau$$
and let $\pi : X \to Y$
be the quotient map. 
Then, $A$, considered in $Y$, is contained in $|-2K_Y|$ (an anti-bicanonical curve on $Y$). 

We define the curves:
$$e:=\pi(E) \ \ \ \text{and} \ \ \ f:= \pi(F) .$$
If $F$ is the union of two nonsingular rational curves $F^\prime$ and $F^{\prime\prime}$ 
which are conjugate by $\tau$ and $F^\prime \cdot F^{\prime\prime} = 1$, then 
$$f = \pi(F) = \pi(F^\prime \cup F^{\prime\prime}) = \pi(F^\prime) = \pi(F^{\prime\prime}).$$

\medskip

We use the same symbols $A_0$ and $A_1$ for their images in $Y$ by $\pi$. 
Then, 
the Picard group $\Pic(Y)$ of $Y$ is generated by the classes $[e]$, $[f]$ and $[A_0]$. 
The Gram matrix of $\Pic(Y)$ with respect to the basis $[e],\ [f]$ and $[A_0]$ is as follows:
$$
\begin{array}{cccc}
        &  {[e]}  &  {[f]}  & {[A_0]} \\
{[e]}   &    -1   &         &         \\
{[f]}   &     1   &   -1    &         \\
{[A_0]} &     1   &    0    &   -4
\end{array}
$$

For $A_1$, we have 
$$A_1\cdot e =  1 \ \ \ \mbox{and} \ \ \ A_1\cdot f = 2.$$

\medskip

We have:

{\bf (1)}\ $F$ is a nonsingular rational curve if and only if 
$A_1$ intersects $f$ in {\bf two distinct points}. 

{\bf (2)}\ $F$ is a union of two nonsingular rational curves if and only if 
$A_1$ {\bf touches} $f$.

\begin{remark}[\cite{NikulinSaito07}] \label{G=1}
For any {\bf real} $2$-elementary K3 surface $(X,\tau,\varphi)$ of type $(S,\theta)$ with $S \cong (3,1,1)$, 
we have 
$$\theta  =  - \id$$
and 
$$G = \{ \id \} ,$$
where $\id$ stands for the identity map on $S$. 
\end{remark}

\subsection{Enumeration of the connected components of the moduli space} \label{enumeration}

For any $2$-elementary K3 surface $(X,\tau)$ of type $S \cong (3,1,1)$, 
all exceptional curves on $Y$ are exactly the curves \ $e,\ f$ and $A_0$. 
Hence, $(X,\tau)$ is ($\da$)-nondegenerate and 
any real $2$-elementary K3 surface $(X,\tau,\varphi)$ of type $(S,\theta) \cong ((3,1,1),- \id)$ 
is ($\da\br$)-nondegenerate (see Definition \ref{da-degenerate}). 
By Theorem \ref{theorem2005moduli}, we have: 
\begin{theorem}[\cite{NikulinSaito07}CTheorem 1] \label{moduli311}
The connected components of the moduli space (in the sense of \cite{NikulinSaito05}) of 
marked real $2$-elementary K3 surfaces $((X,\tau,\varphi),\ \alpha)$ of type $(S,\theta) \cong ((3,1,1),- \id)$ 
are in bijective correspondence with 
the isometry classes of 
integral involutions of $\bl_{K3}$ of type $(S,\theta) \cong ((3,1,1),- \id)$ with respect to $G = \{ \id \}$. 
\end{theorem}

\medskip

The complete isometry invariants of integral involutions $\psi$ of $\bl_{K3}$ of type $(S,\theta) \cong ((3,1,1),- \id)$ 
are the data
\begin{equation}
(r(\psi),\ a(\psi),\ \delta_{\psi S},\ H(\psi)). 
\label{geninv-real-311H}
\end{equation}

See \cite{NikulinSaito05}, Subsection 2.3 (after the equation (2.20)) 
for the definition of the invariant 
$\delta_{\psi S}$ and $H(\psi)$. 
Remark that 
$H(\psi)$ is a subgroup of the discriminant group $S_-/2S_- = S/2S = \bz/2\bz(\alpha([F]))$. 
Hence, $H(\psi) = 0$ or isomorphic to $\bz/2\bz$. 

All realizable data (\ref{geninv-real-311H}) are enumerated in \cite{NikulinSaito07}. 
We have 
$12$ data of ``Type 0" ($\Rightarrow H(\psi) = 0$),
$12$ data of ``Type Ia" , 
$39$ data of ``Type Ib" with $H(\psi) = 0$, and 
$39$ data of ``Type Ib" with $H(\psi) \cong \bz/2\bz$.
See also 
{\sc Table} \ref{H=0} for $H(\psi) = 0$ case and 
{\sc Table} \ref{H=F} for $H(\psi) \cong \bz/2\bz$ case, 
where $0$ or $1$ in each cell stands for the value of $\delta_{\psi S}$. 

Thus we have exactly {\bf 102} connected components of the moduli of 
real $2$-elementary K3 surfaces of type $(S,\theta) \cong ((3,1,1),- \id)$. 
By \cite{NikulinSaito05}, we find 
how invariants of an involution $\psi$ and 
those of its ``related integral involution" $\sigma \circ \psi$ are calculated from each other. 
Here $\sigma : \bl_{K3} \to \bl_{K3}$ is defined to be the integral involution of $\bl_{K3}$ whose fixed part is $S$. 
Identifying related pairs of anti-holomorphic involutions on each K3 surface, 
we have exactly {\bf 51} connected components. 

\begin{table}[!h]
\begin{center}
\begin{tabular}{|r||c|c|c|c|c|c|c|c|c|c|c|c|c|c|c|c|c|c|} \hline
9 &  &  &  &  &  &  &  &  & $1$ &  &  &  &  &  &  &  &  &  \\\hline
8 &  &  &  &  &  &  &  & $1$ &  & $0,1$ &  &  &  &  &  &  &  &  \\\hline
7 &  &  &  &  &  &  & $1$ &  & $1$ &  & $1$ &  &  &  &  &  &  &  \\\hline
6 &  &  &  &  &  & $1$ &  & $1$ &  & $0,1$ &  & $1$ &  &  &  &  &  &  \\\hline
5 &  &  &  &  & $1$ &  & $1$ &  & $1$ &  & $1$ &  & $1$ &  &  &  &  &  \\\hline
4 &  &  &  & $1$ &  & $0,1$ &  & $1$ &  & $0,1$ &  & $1$ &  & $0,1$ &  &  &  &  \\\hline
3 &  &  & $1$ &  & $1$ &  & $1$ &  & $1$ &  & $1$ &  & $1$ &  & $1$ &  &  &  \\\hline
2 &  & $0,1$ &  & $1$ &  & $0$ &  & $1$ &  & $0,1$ &  & $1$ &  & $0$ &  & $1$ &  &  \\\hline
1 & $1$ &  & $1$ &  &  &  &  &  & $1$ &  & $1$ &  &  &  &  &  & $1$ &  \\\hline
0 &  & $0$ &  &  &  &  &  &  &  & $0$ &  &  &  &  &  &  &  & $0$ \\\hline \hline
$a(\psi)$ / $r(\psi)$ & 1 & 2 & 3 & 4 & 5 & 6 & 7 & 8 & 9 & 10 & 11 & 12 & 13 & 14 & 15 & 16 & 17 & 18 \\\hline
\end{tabular}
\end{center}
\caption{All the data $(r(\psi),\ a(\psi),\ \delta_{\psi S})$ for $H(\psi) = 0$.}
\label{H=0}
\end{table}

\begin{table}[!h]
\begin{center}
\begin{tabular}{|r||c|c|c|c|c|c|c|c|c|c|c|c|c|c|c|c|c|c|} \hline
10 &  &  &  &  &  &  &  &  &  & $1$ &  &  &  &  &  &  &  &  \\\hline
9 &  &  &  &  &  &  &  &  & $0,1$ &  & $1$ &  &  &  &  &  &  &  \\\hline
8 &  &  &  &  &  &  &  & $1$ &  & $1$ &  & $1$ &  &  &  &  &  &  \\\hline
7 &  &  &  &  &  &  & $1$ &  & $0,1$ &  & $1$ &  & $1$ &  &  &  &  &  \\\hline
6 &  &  &  &  &  & $1$ &  & $1$ &  & $1$ &  & $1$ &  & $1$ &  &  &  &  \\\hline
5 &  &  &  &  & $0,1$ &  & $1$ &  & $0,1$ &  & $1$ &  & $0,1$ &  & $1$ &  &  &  \\\hline
4 &  &  &  & $1$ &  & $1$ &  & $1$ &  & $1$ &  & $1$ &  & $1$ &  & $1$ &  &  \\\hline
3 &  &  & $1$ &  & $0$ &  & $1$ &  & $0,1$ &  & $1$ &  & $0$ &  & $1$ &  & $0,1$ &  \\\hline
2 &  & $1$ &  &  &  &  &  & $1$ &  & $1$ &  &  &  &  &  & $1$ &  & $1$ \\\hline
1 & $0$ &  &  &  &  &  &  &  & $0$ &  &  &  &  &  &  &  & $0$ &  \\\hline \hline
$a(\psi)$ / $r(\psi)$ & 1 & 2 & 3 & 4 & 5 & 6 & 7 & 8 & 9 & 10 & 11 & 12 & 13 & 14 & 15 & 16 & 17 & 18 \\\hline
\end{tabular}
\end{center}
\caption{All the data $(r(\psi),\ a(\psi),\ \delta_{\psi S})$ for $H(\psi) \cong \bz/2\bz$.}
\label{H=F}
\end{table}

\subsection{The regions $A_+$ and $A_-$ in the real part $Y(\br)$ of $Y$}

Let $(X,\tau,\varphi)$ be 
a real $2$-elementary K3 surface of type $(S,\theta) \cong ((3,1,1),- \id)$. 
We use the notation in Subsection \ref{RealK3-311}. 
Let 
$$X_\varphi(\br)$$
denote the real part of $(X,\varphi)$, 
i.e., 
the fixed point set of $\varphi$, and let 
$$Y(\br)$$
be the real part of the quotient surface $Y$ with the real structure $\varphi_{\mathrm{mod}\ \tau}$. 

The real part 
$$\br A = \br A_0 \cup \br A_1$$
of the branch curve $A$ 
divides $Y(\br)$ into two regions
\footnote
{These two regions are called ``positive curves" in \cite{NikulinSaito05}.}
$$A_+ \ \ \text{and} \ \ A_-.$$

Either $A_+$ or $A_-$ is doubly covered by the real part $X_\varphi(\br)$, 
and the other by the real part $X_{\wvarphi}(\br)$, 
where we set $\wvarphi := \tau \circ \varphi$. 
Since $X_\varphi(\br)$ is always non-empty (\cite{NikulinSaito07}, p.27), 
$Y(\br)$, $A_+$, and $A_-$ are also non-empty. 
Regions $A_\pm$ could have several connected components. 

\medskip

We distinguish two regions $A_+$ and $A_-$ as follows.
\begin{definition}[The regions $A_+$ and $A_-$]\label{A_+A_-}
For a real $2$-elementary K3 surface $(X,\tau,\varphi)$ of type $(S,\theta) \cong ((3,1,1),- \id)$, 
we define the regions 
$A_+$ and $A_-$ ($\subset Y(\br)$) as follows:
$$
\left\{
\begin{array}{cll}
A_+ &:= \pi(X_\varphi(\br)) & \ \text{if}\ \ H(\psi) \cong \bz/2\bz,\\
A_- &:= \pi(X_\varphi(\br)) & \ \text{if}\ \ H(\psi) = 0.
\end{array}
\right.
$$
\end{definition}
Note that if $H(\psi) \cong \bz/2\bz$, then $H(\sigma \circ \psi)=0$.  

An important topological characterization of the invariant $H(\psi)$ is as follows. 
\begin{lemma}[\cite{NikulinSaito05}]
Suppose that the curve $F$ is irreducible. Then we have 
$H(\psi) \cong \bz/2\bz$ if and only if $[F_\varphi(\br)] = 0 \ \ \mbox{in} \ H_1(X_\varphi(\br);\bz/2\bz)$, 
where we set $F_\varphi(\br) := X_\varphi(\br) \cap F$.
\end{lemma}

\bigskip

\subsection{Real anti-bicanonical curves with one real double point on $\br \bff_4$}

We now contract the exceptional curve 
$$f = \pi(F)$$
to a point. We get a blow up
$$\BL : Y \to \bff_4,$$
where $\bff_4$ is the $4$-th Hirzebruch surface. 

We set
$$s:=\BL(A_0),\ \ \ \ A^\prime_1 := \BL(A_1),\ \ \ \ c:=\BL(e),$$
and we have
$$\BL(A) = \BL(A_0) + \, \BL(A_1) =  s +  A^\prime_1 \ \ \in |-2K_{\bff_4}|,$$
i.e., $\BL(A)$ is an anti-bicanonical curve, 
where 
$s$ is the exceptional section of $\bff_4$ with $s^2 = -4$ and 
$c$ is a fiber of the fibration $\bff_4\to s$ with $c^2=0$. 
Since $s \cdot A^\prime_1 = 0$, 
{\bf $A^\prime_1$ does not intersect the section $s$}. 

We have 
$$-2K_{\bff_4} \sim 12c+4s.$$
It follows that
$$A^\prime_1 \ \ \in |12c+3s|.$$

\bigskip

Let 
$$
\br \BL(A) = \br s \cup  \br A^\prime_1
$$
be the real part of the curve $\BL(A)$, where 
$\br s$ and $\br A^\prime_1$ are the real parts of $s$ and $A^\prime_1$ respectively. 

\medskip

We set 
$$P_0 := \BL(f)\ \ (\in \br c) .$$

Since $A_1\cdot f =2$ in $Y$, 
$A^\prime_1$ has {\bf one real double point} $P_0$. 
If $A_1$ intersects with $f$ at two distinct points in $Y$, then they are real points or non-real conjugate points. 
In the former case $P_0$ is a {\bf real node}, and in the latter case it is a {\bf real isolated point}. 
Anyway it is a {\bf nondegenerate} double point. 
If $A_1$ touches to $f$ in $Y$, then $P_0$ is a {\bf real cusp}. (a {\bf degenerate} double point)

Thus, there are three topological types of the curve $\br A^\prime_1$ near the double point $P_0$. 

\noindent
{\bf Node case:}\ \ 
$P_0$ is a real node of $\br A^\prime_1$.

\noindent
{\bf Cusp case:}\ \ 
$P_0$ is a real cusp of $\br A^\prime_1$. (degenerate double point) 

\noindent
{\bf Isolated point case:}\ \ 
$P_0$ is a real isolated point of $\br A^\prime_1$.

See Figure \ref{a1c2} below.

\begin{figure}[!h]
\begin{center}
\includegraphics[width=4cm]{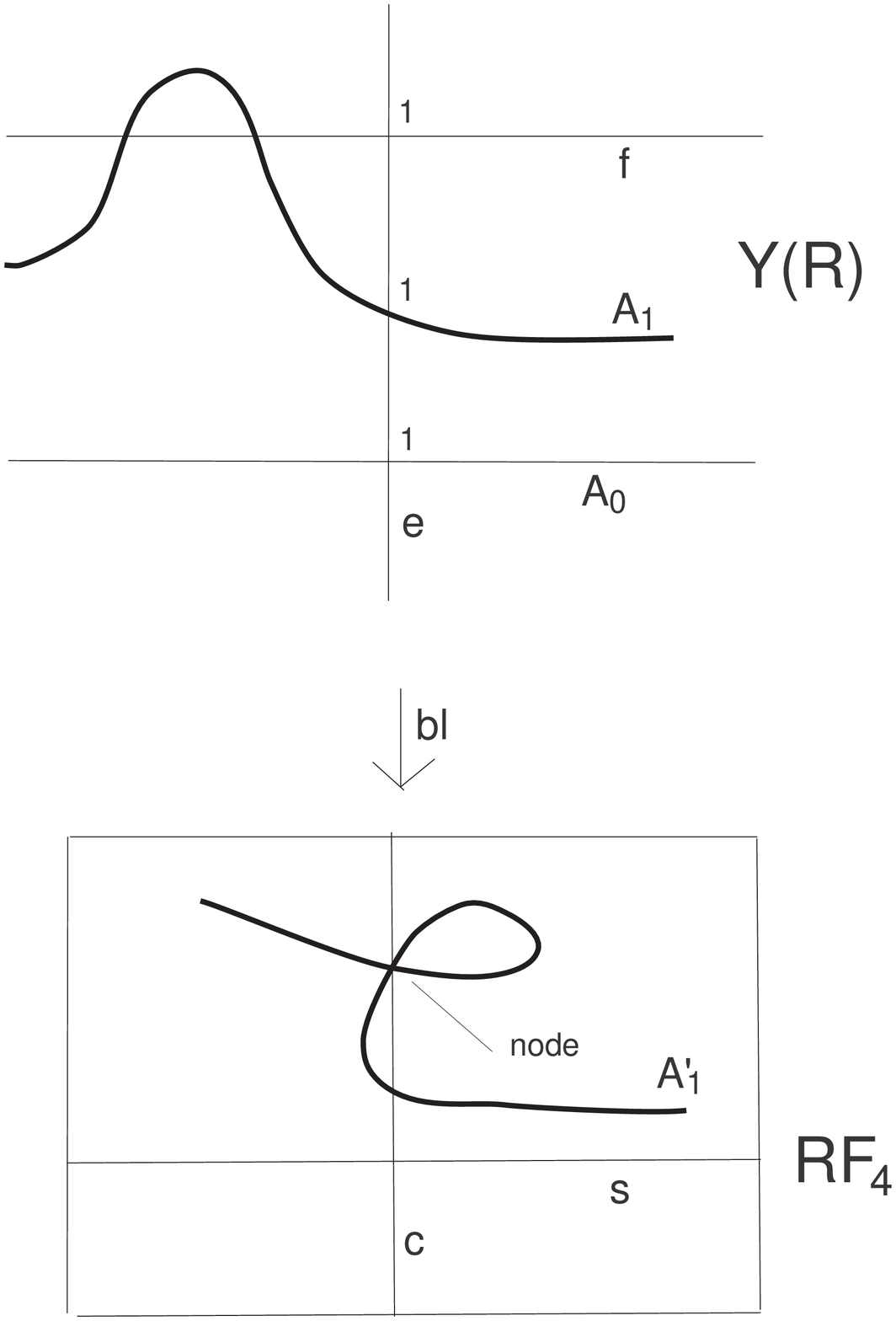}
\hspace{7mm}
\includegraphics[width=4cm]{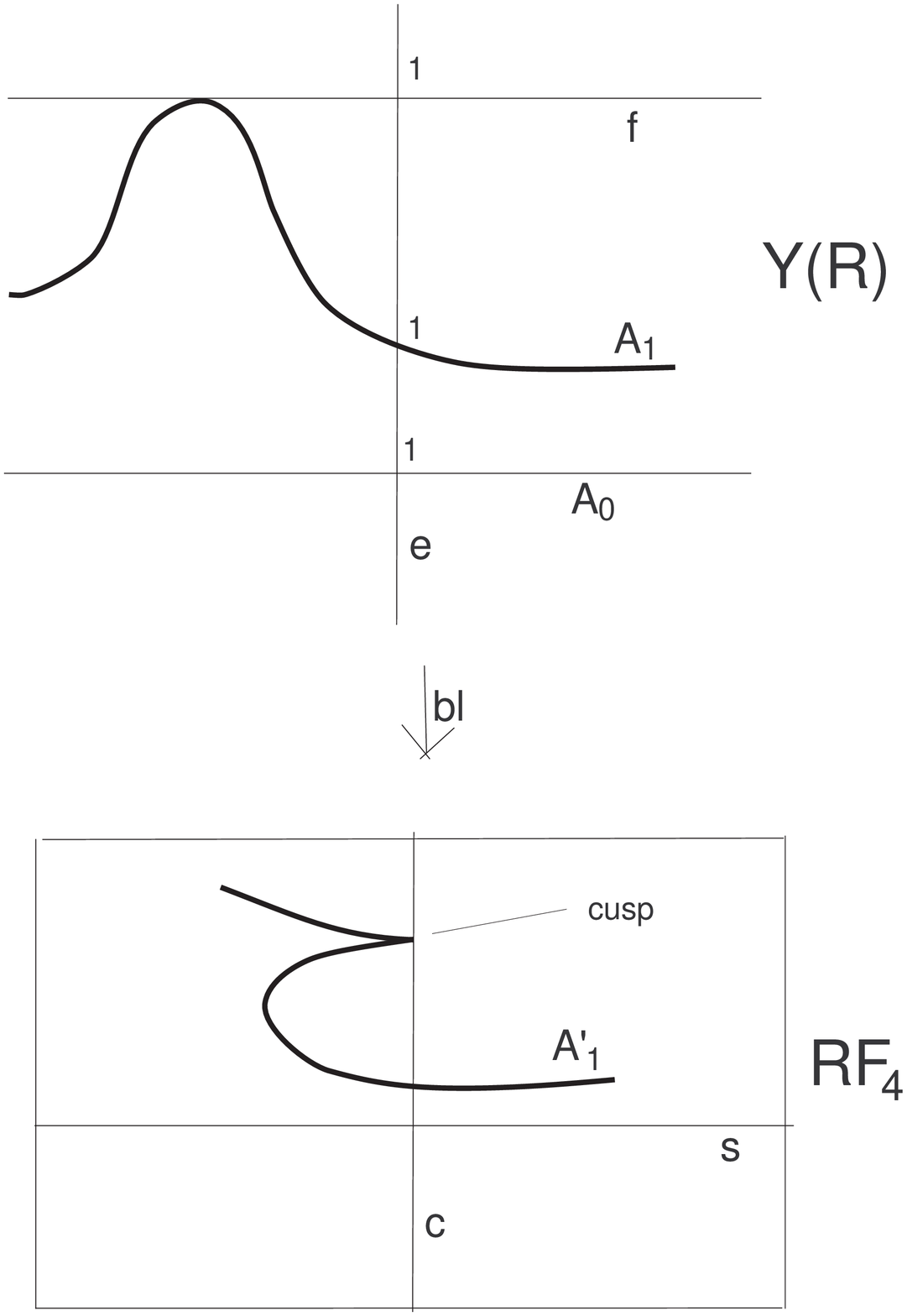}
\hspace{7mm}
\includegraphics[width=4cm]{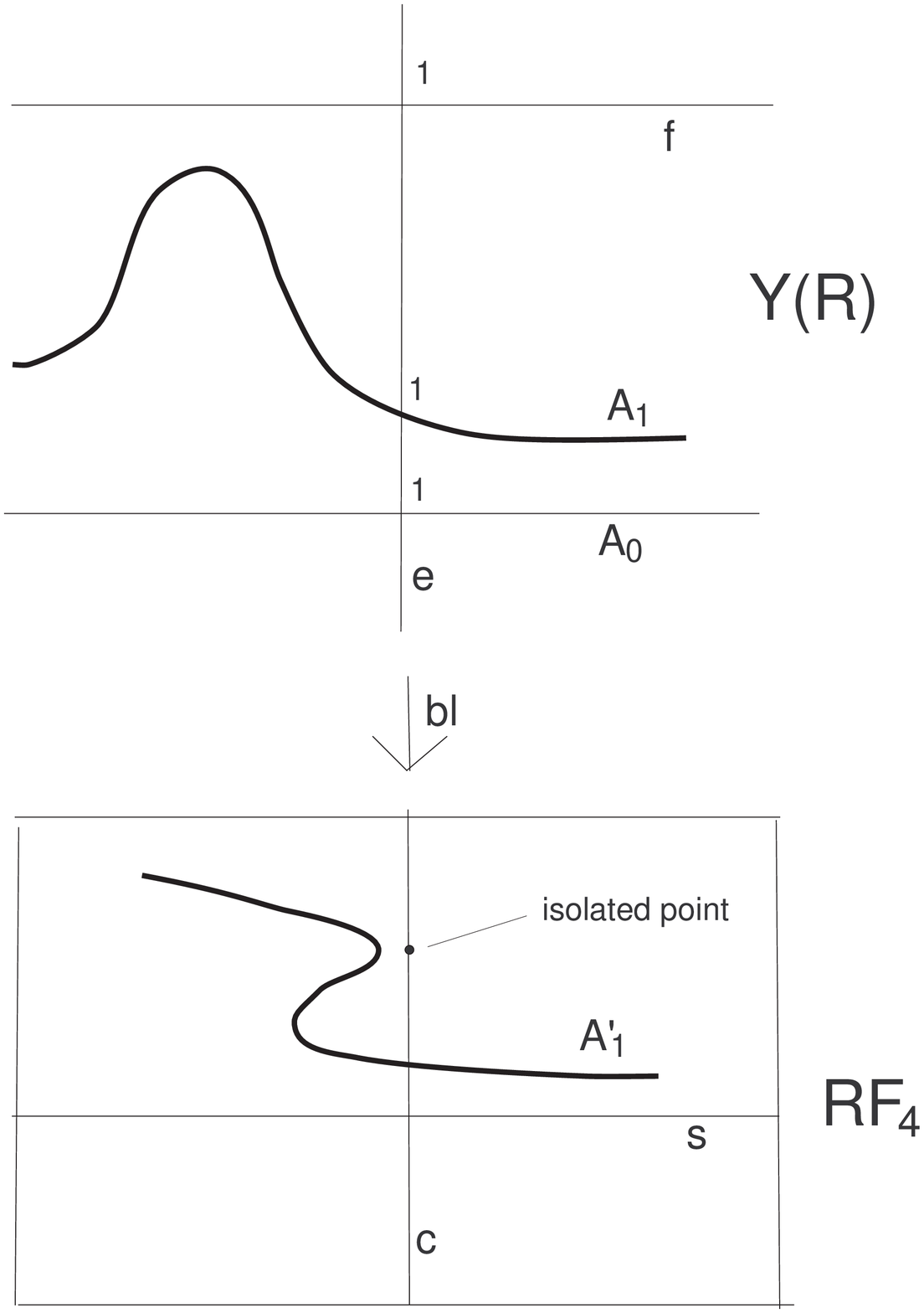}
\end{center}
\caption{The real double point $P_0$ of the curve $A^\prime_1$}
\label{a1c2}
\end{figure}

We have $A^\prime_1 \cdot s = 0$ and $A^\prime_1 \cdot c = 3$ ($A^\prime_1$ is a trigonal curve). 
Since $f \cdot A_0 = 0$ in $Y$, the section $s$ does not meet the double point $P_0$. 
We may assume that $e$ does not pass through any intersection point of $A_1$ and $f$ in $Y$. 
(See Figure \ref{a1c2}.) 
Since $A_1 \cdot e = 1$ in $Y$, 
the intersection point of $A_1$ with $e$ is {\bf real} and does not meet $f$. 
Via the map $\BL$, 
this point goes to a real intersection point of $A^\prime_1$ with $c$ with multiplicity $1$. 
Thus we set 
$$
P_1:=\BL(A_1 \cap e).\ \ \ \mbox{(the intersection point)}
$$

Since $A^\prime_1 \cdot c = 3$, 
$A^\prime_1$ intersects with $c$ at $P_0$ with multiplicity $2$ 
                       and at $P_1$ with multiplicity $1$.

\medskip

Since $A^\prime_1 \cap s = \emptyset$, 
any {\bf non-contractible (possibly real singular)} components of $\br A^\prime_1$ are ``parallel" to $\br s$. 
See Figure\ref{noncon-n-1}, \ref{noncon-n-2} and \ref{noncon}. 
But two types of non-contractible singular components in Figure\ref{noncon-n-1} are real isotopic, and 
three types of non-contractible components in Figure\ref{noncon} are real isotopic. 
\begin{figure}[!h]
\begin{center}
\includegraphics[width=3cm]{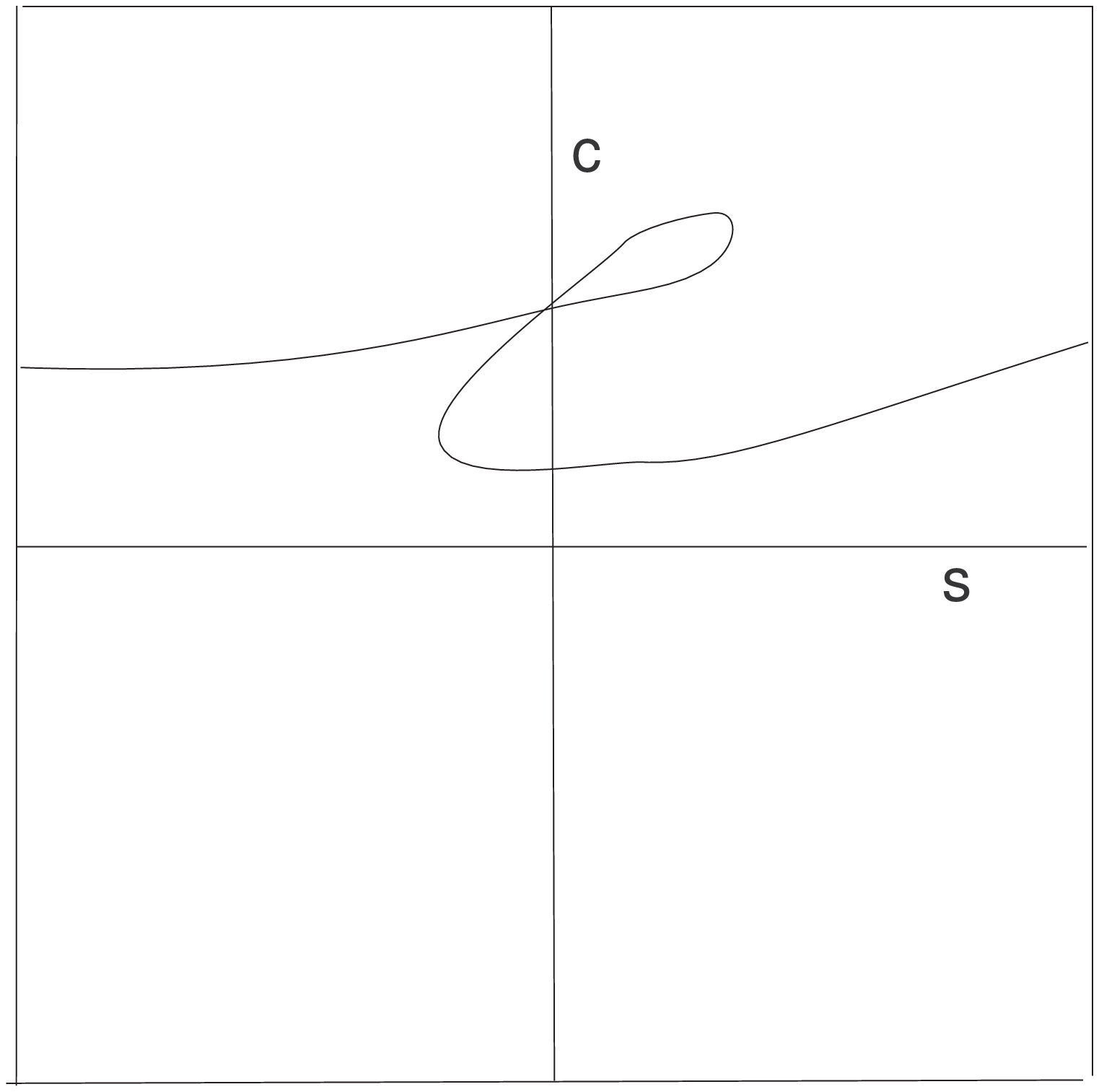}
\hspace{1.5cm}
\includegraphics[width=3cm]{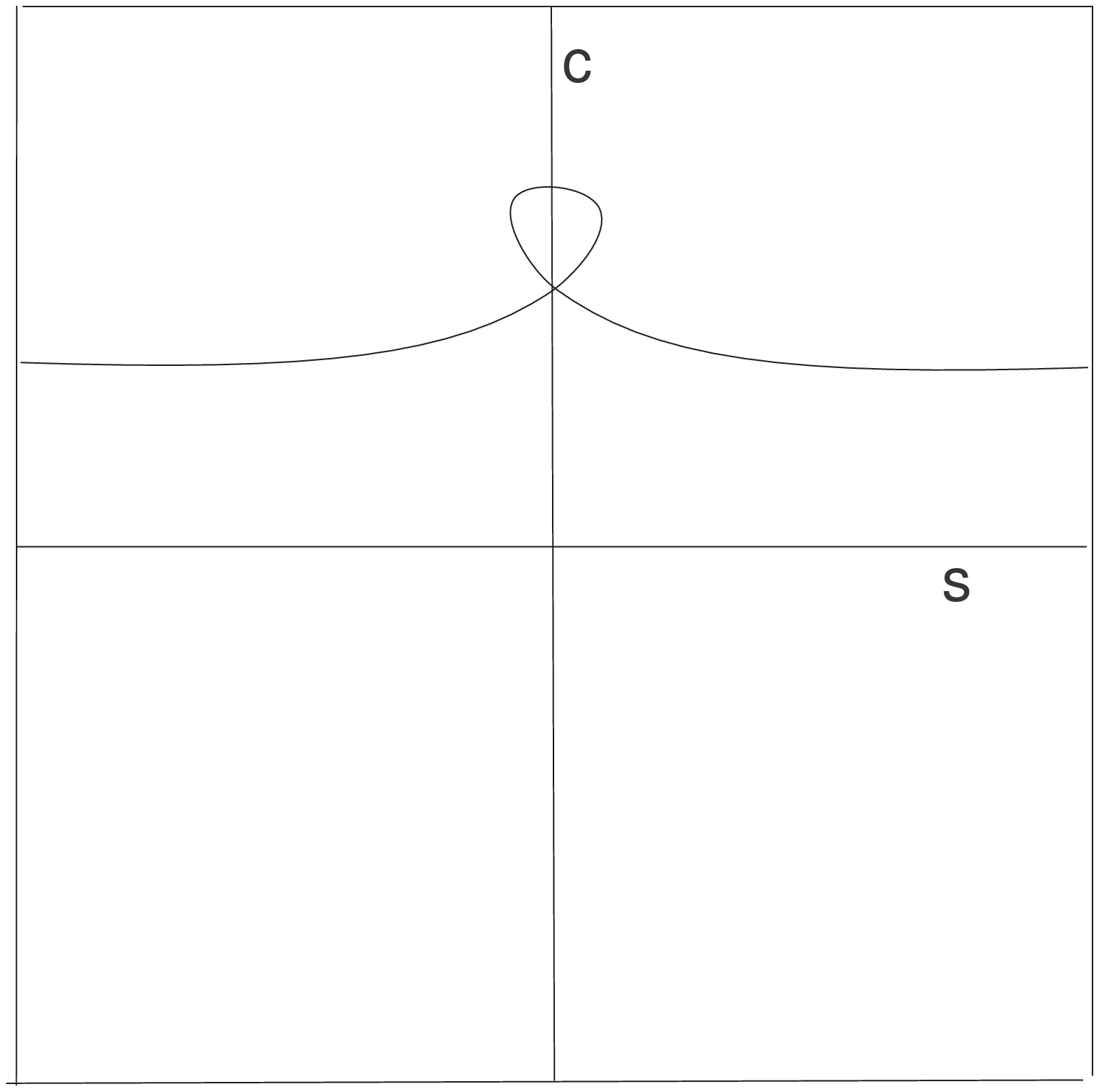}
\end{center}
\caption{A non-contractible component with a node of type {\bf Node (1)}}
\label{noncon-n-1}
\end{figure}
\begin{figure}[!h]
\begin{center}
\includegraphics[width=3cm]{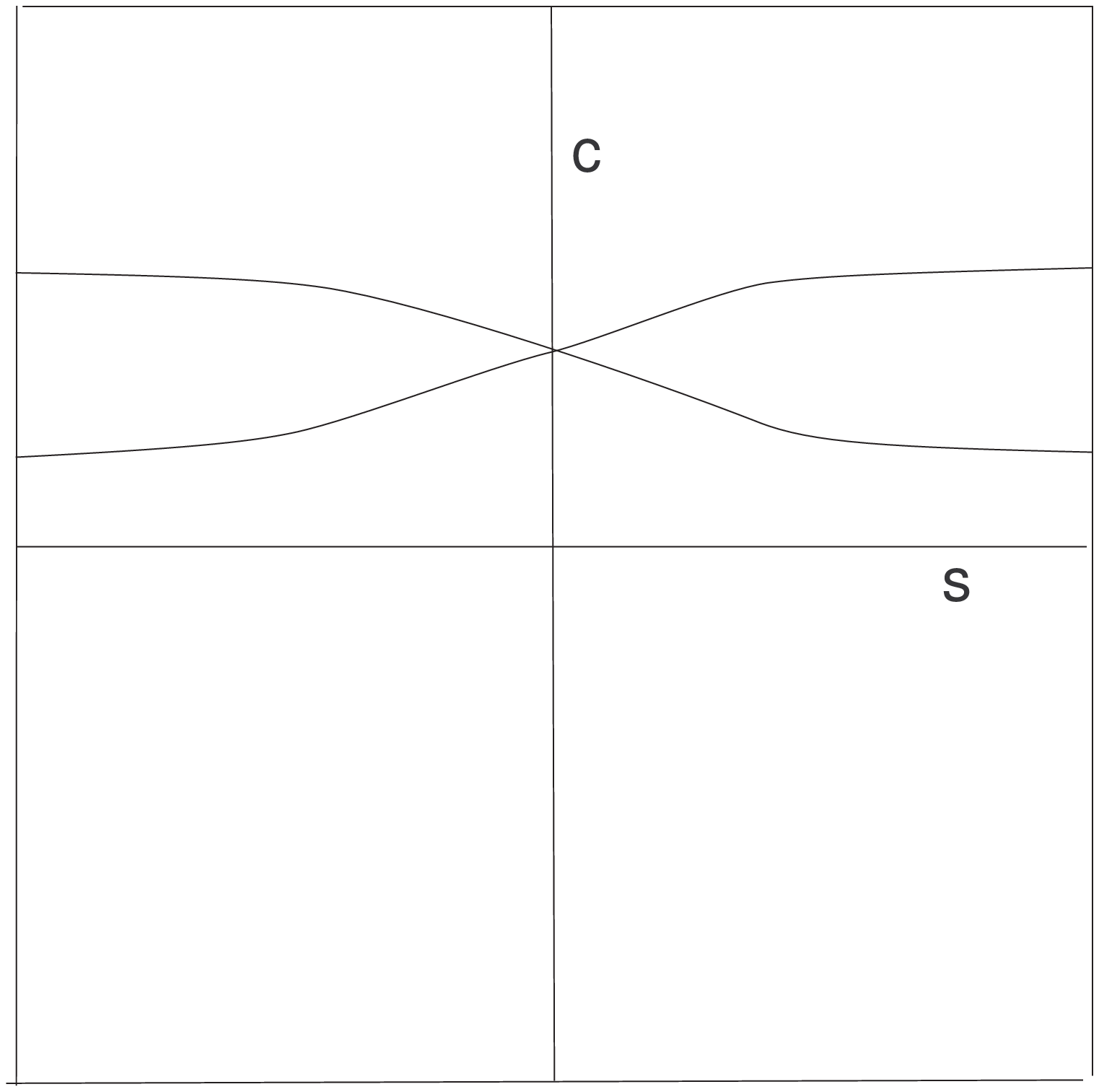}
\end{center}
\caption{A non-contractible component with a node of type {\bf Node (*)}}
\label{noncon-n-2}
\end{figure}
\begin{figure}[!h]
\begin{center}
\includegraphics[width=3cm]{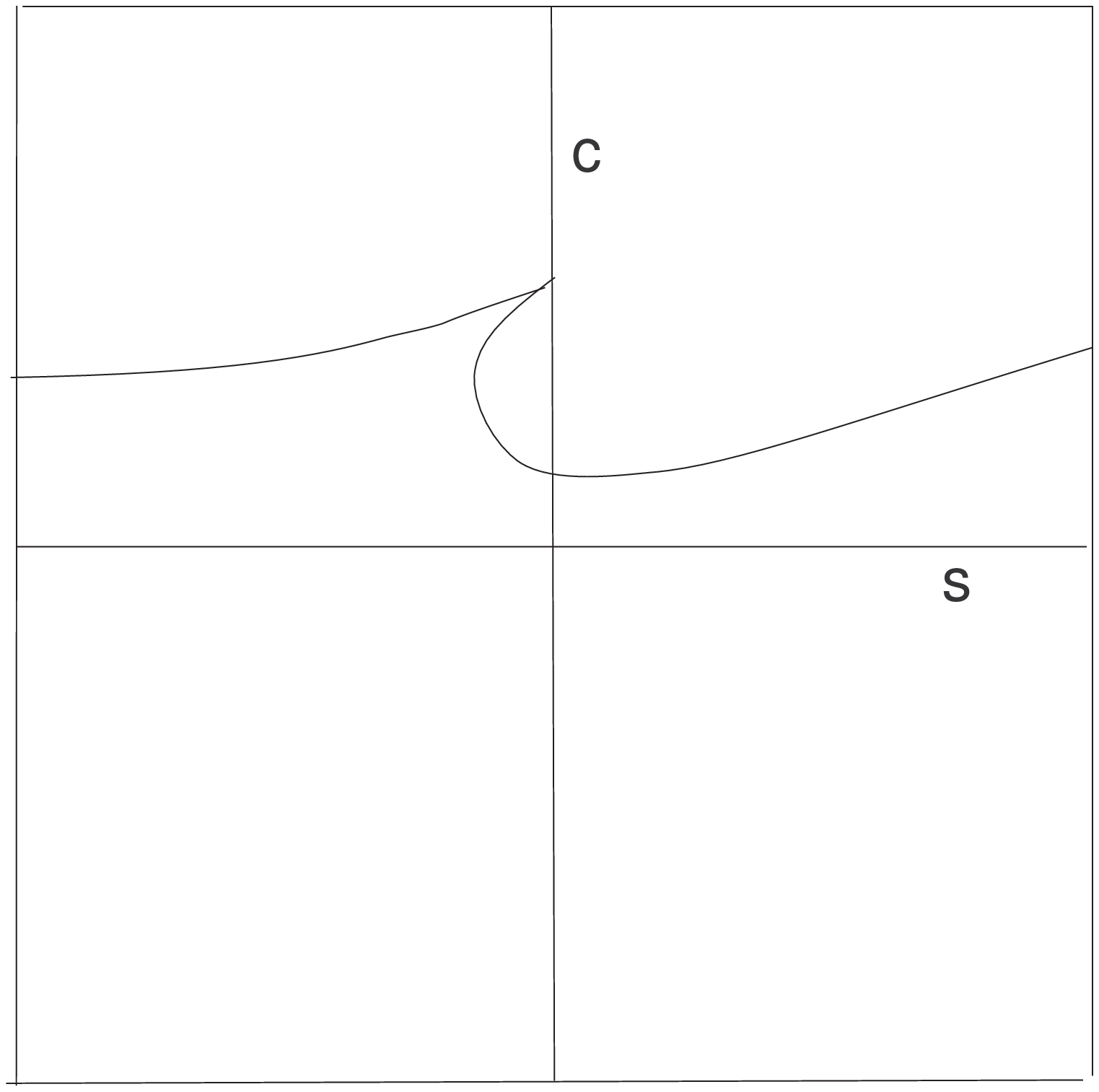}
\hspace{1cm}
\includegraphics[width=3cm]{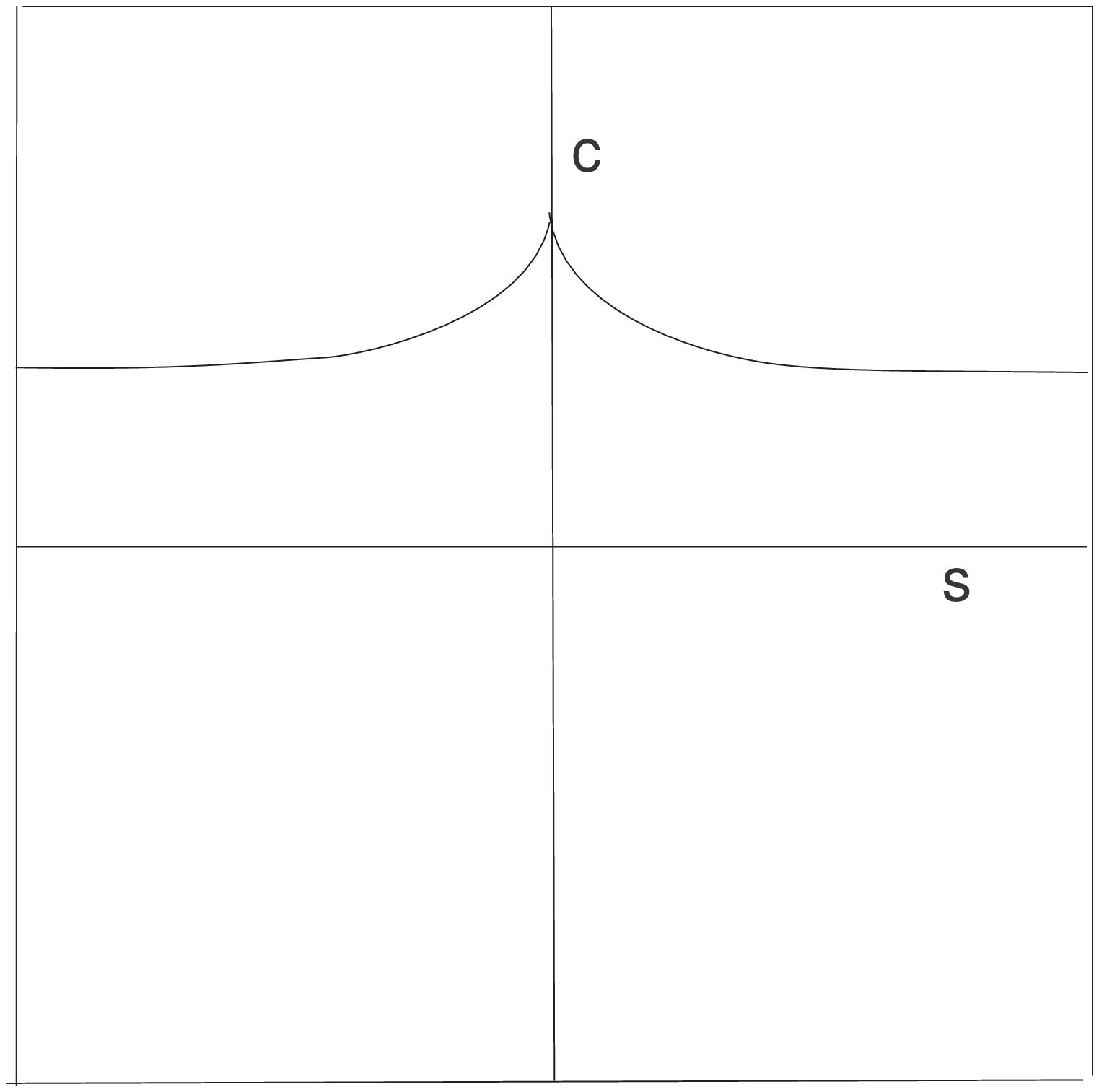}
\hspace{1cm}
\includegraphics[width=3cm]{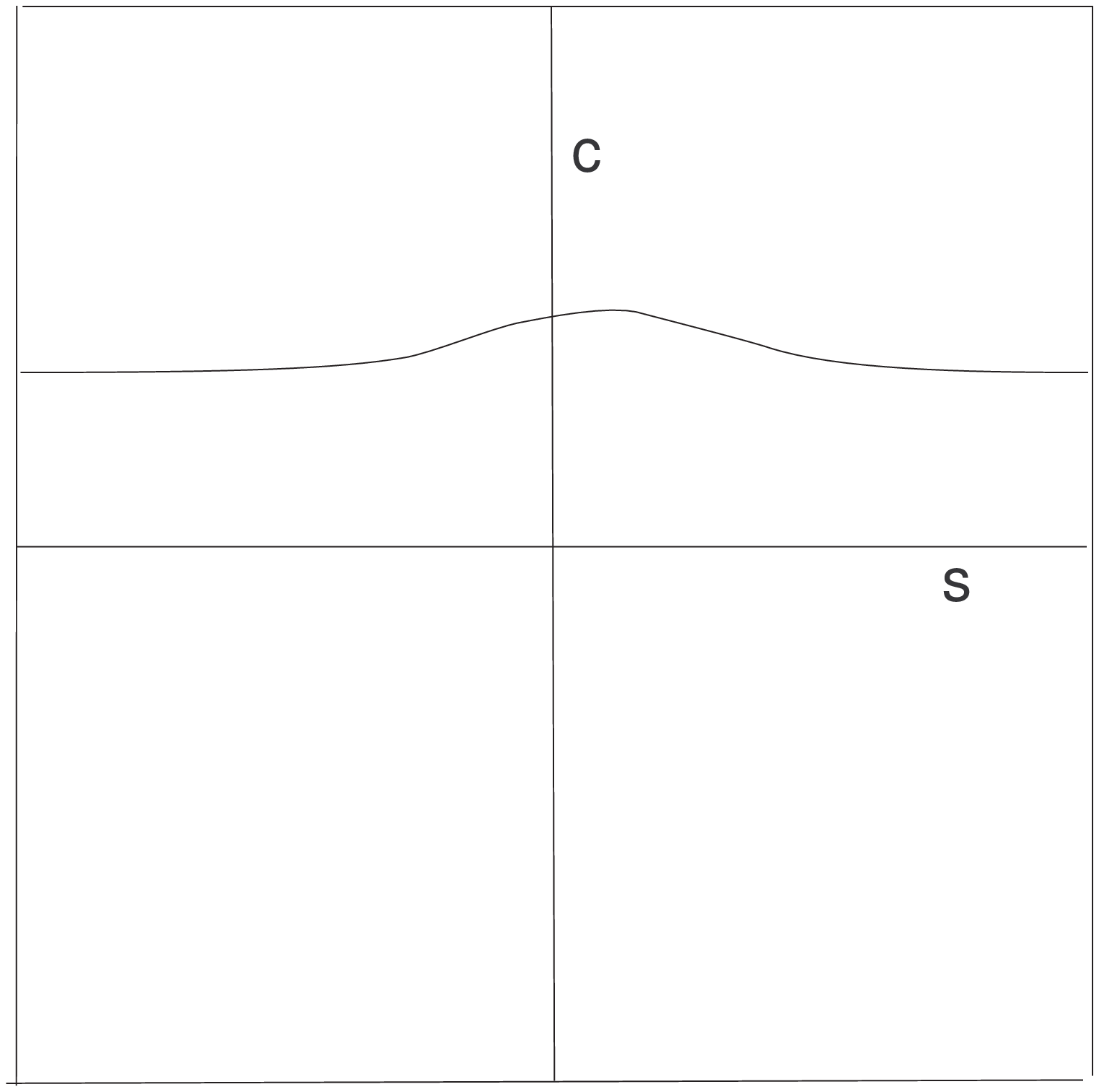}
\end{center}
\caption{A non-contractible component with a {\bf real cusp} or without singular points}
\label{noncon}
\end{figure}   


\begin{definition}
We call a connected component of $\br A^\prime_1$ an {\bf oval} if 
it has no real singular points 
and contractible in $\br \bff_4$ (a torus), namely, 
realizes $0$ in $H_1(\br \bff_4 ; \bz)$. 
\end{definition}

\begin{remark}\label{trigonal}
Since $A^\prime_1 \cdot c = 3$, 
$A^\prime_1$ is a {\bf trigonal} curve on $\bff_4$. See \cite{DIZ2011}, \cite{Orevkov2003} for related results. 
We see that 
the interior of each oval of $\br A^\prime_1$ 
does not contain any other ovals. 
Moreover, 
the ovals are canonically ordered according to their projections to the base $\br s$. 
\end{remark}

\bigskip

We now get the following possibilities. 

\noindent
{\bf Node case:}

{\bf Node (1) case:}\ \ \ 
We consider the case when {\bf both} the node $P_0$ and the intersection point $P_1$ 
are contained in {\bf the same} connected singular component of $\br A^\prime_1$. 
Then $\br A^\prime_1$ meets $\br c$ at only $P_0$ and $P_1$, 
and the singular component is {\bf not contractible}. 
See Figure \ref{noncon-n-1}. 
Note that $\br A^\prime_1$ might have some ovals. 
Since $A^\prime_1 \cdot c =3$, the interior of any oval of $\br A_1$ does not contain other ovals.
The interior of the node also does not contain ovals. 

{\bf The non-contractible singular component containing $P_0$ and $P_1$} and {\bf the section $\br s$} 
divide $\br \bff_4$ (a torus) into three parts, which are 
{\bf the interior of the node} and {\bf two non-contractible regions}. 

\begin{definition}[The regions $R_1$ and $R_2$ in Node (1) case]
Let $R_1$ denote 
{\bf the non-contractible region which is connected with the interior of the node in the blow up of $\br \bff_4$}, 
and 
let $R_2$ denote the other non-contractible region. 
We define the integers $\alpha$ and $\beta$ as follows:
\begin{equation}
\begin{array}{lcl}
\alpha & := & \# \{ \text{ovals\ contained\ in}\ R_1 \},\\
\beta  & := & \# \{ \text{ovals\ contained\ in}\ R_2 \}.
\end{array}
\label{alpha-beta}
\end{equation}
See {\bf the left} figure of Figure \ref{Case1-1}. 
\end{definition}
\begin{figure}[!h]
\begin{center}
\includegraphics[width=5cm]{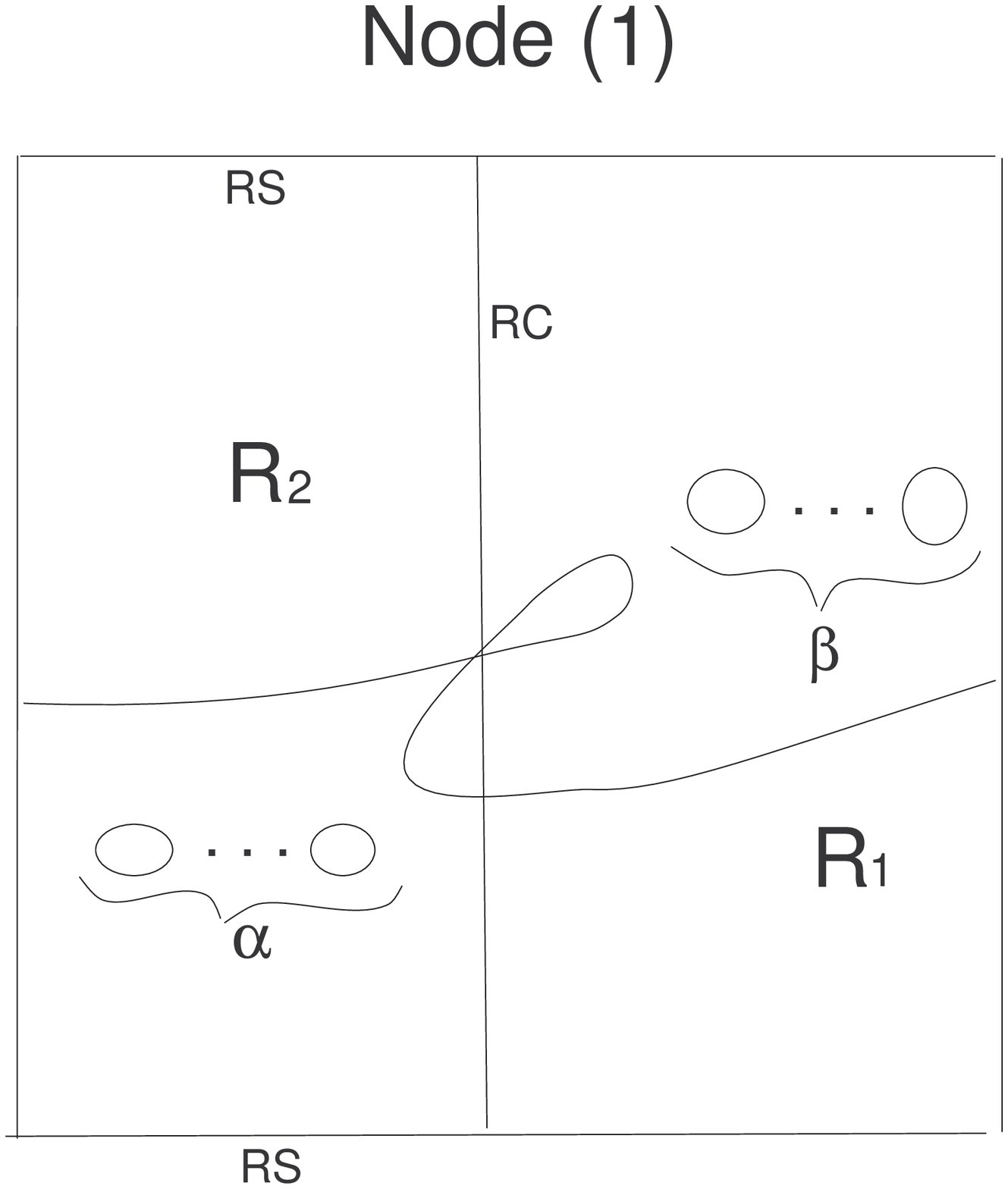}
\hspace{2cm}
\includegraphics[width=5cm]{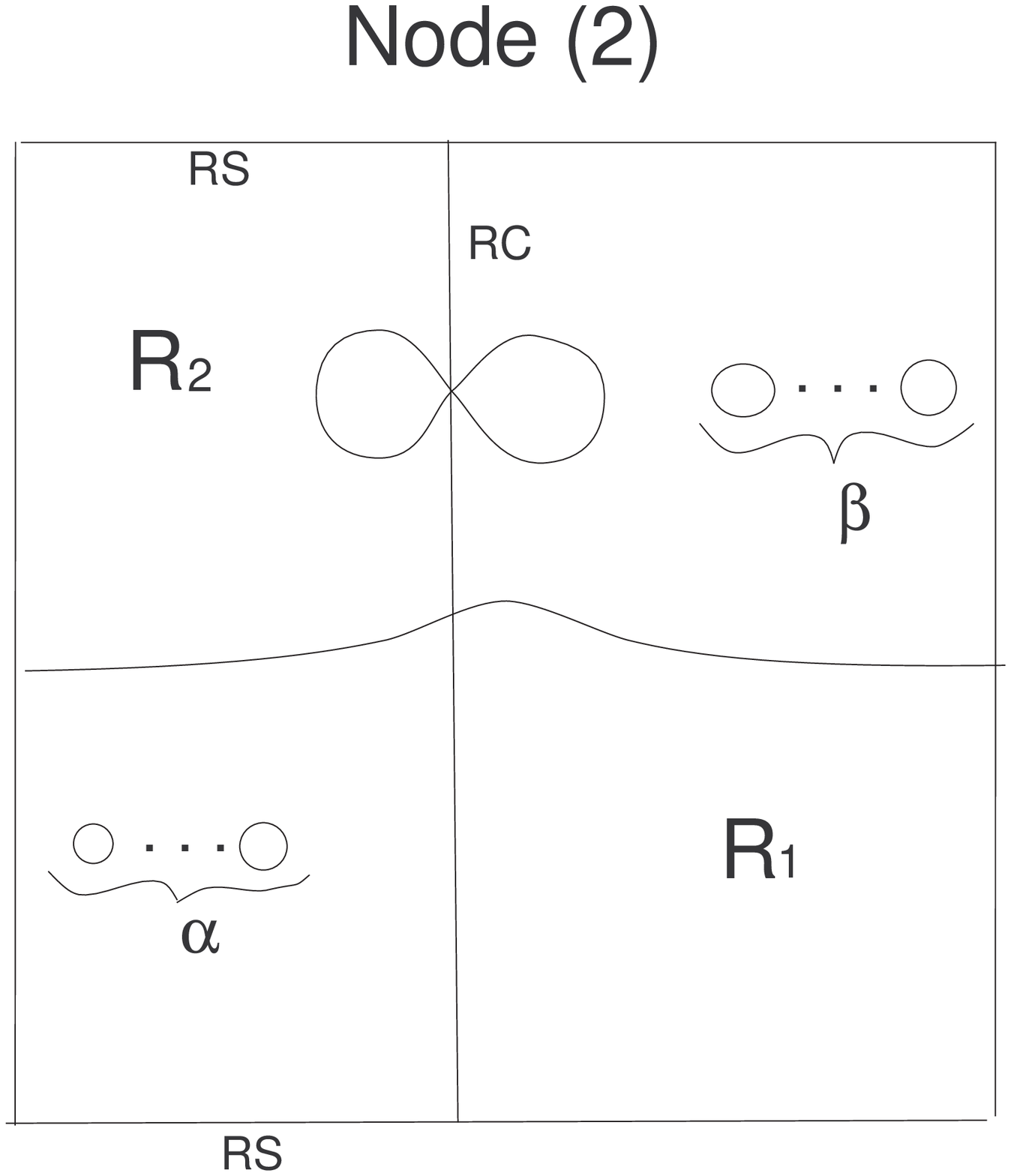}
\end{center}
\caption{The regions $R_1$ and $R_2$ in {\bf Node (1)} or {\bf Node (2)} case}
\label{Case1-1}
\end{figure}

\medskip

{\bf Node (2) case and Node (*) case:}\ \ \ 
When the node $P_0$ and the intersection point $P_1$ respectively are contained in {\bf different} connected components of $\br A^\prime_1$, 
the component containing $P_1$ is nonsingular and not contractible like the rightmost figure of Figure \ref{noncon} above. 
The component containing $P_0$ can be either 
{\bf contractible} (the left figure of Figure \ref{con})\ or \ {\bf non-contractible} (Figure \ref{noncon-n-2}). 
\begin{figure}[!h]
\begin{center}
\includegraphics[width=3cm]{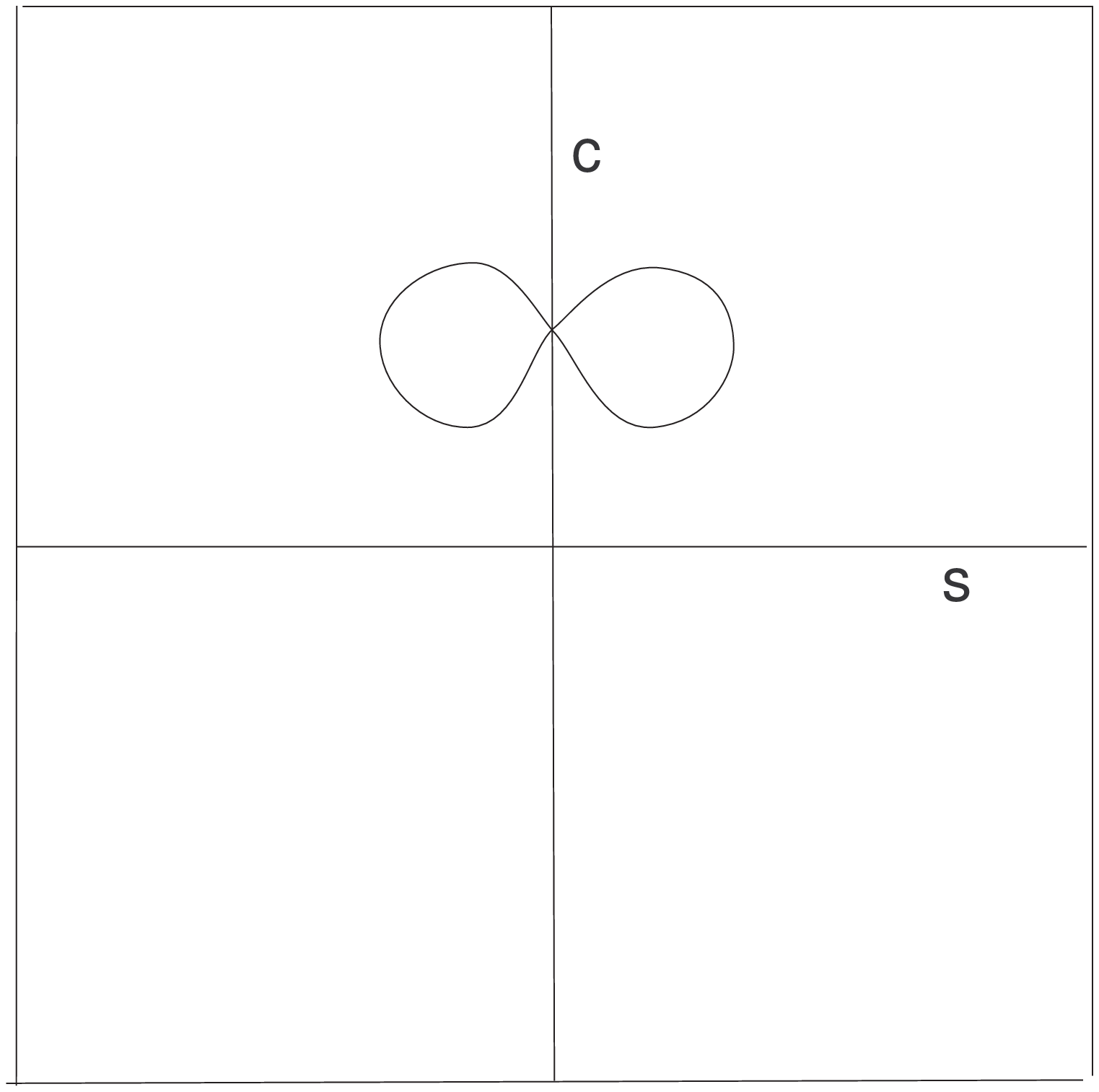}
\hspace{1cm}
\includegraphics[width=3cm]{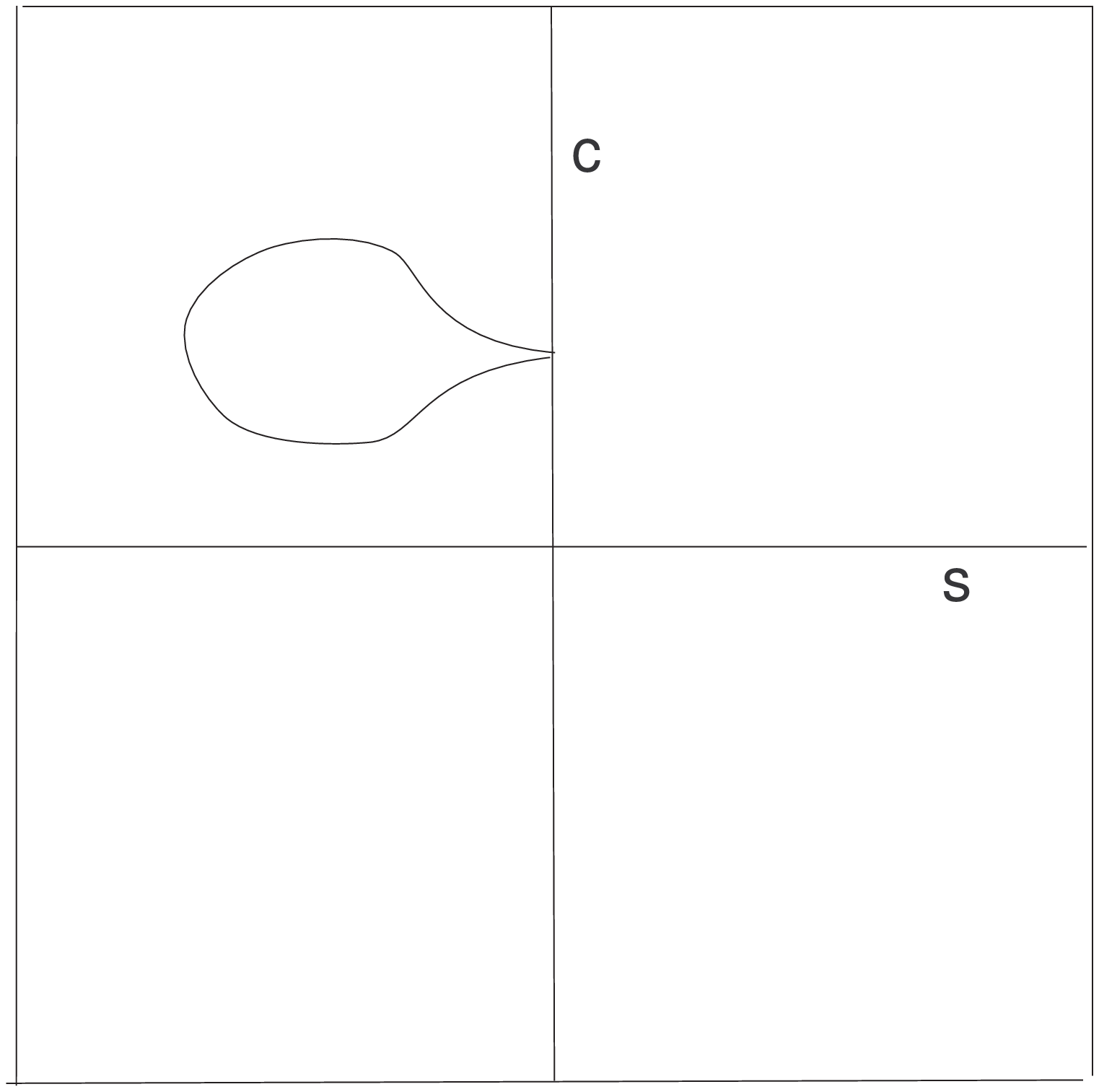}
\end{center}
\caption{Contractible components with a real node or a real cusp}
\label{con}
\end{figure}

{\bf Node (2) case:}\ 
When the component containing the node $P_0$ is {\bf contractible} (the left figure of Figure \ref{con}), 
$\br A^\prime_1$ might have some ovals. 
{\bf The component containing $P_1$} and {\bf the section $\br s$} 
divide $\br \bff_4$ into two regions. 
\begin{definition}[The regions $R_1$ and $R_2$ in Node (2) case]
Let $R_1$ denote 
{\bf the region which does not contain the contractible component containing the node $P_0$}, 
and 
let $R_2$ denote the other region. 
(Since $A^\prime_1 \cdot c =3$, 
the interior of the contractible component containing $P_0$ and 
the interior of any oval of $\br A_1$ cannot contain any other ovals.) 

We define the integers $\alpha$ and $\beta$ by (\ref{alpha-beta}). 
See {\bf the right} figure of Figure \ref{Case1-1}.
\end{definition}

\medskip

{\bf Node (*) case:}\ 
If the component containing the node $P_0$ is {\bf non-contractible} (see Figure \ref{noncon-n-2}), 
then 
$\br A^\prime_1$ has {\bf no ovals} (see Figure \ref{Case1-2}). 
\begin{figure}[!h]
\begin{center}
\includegraphics[width=3.5cm]{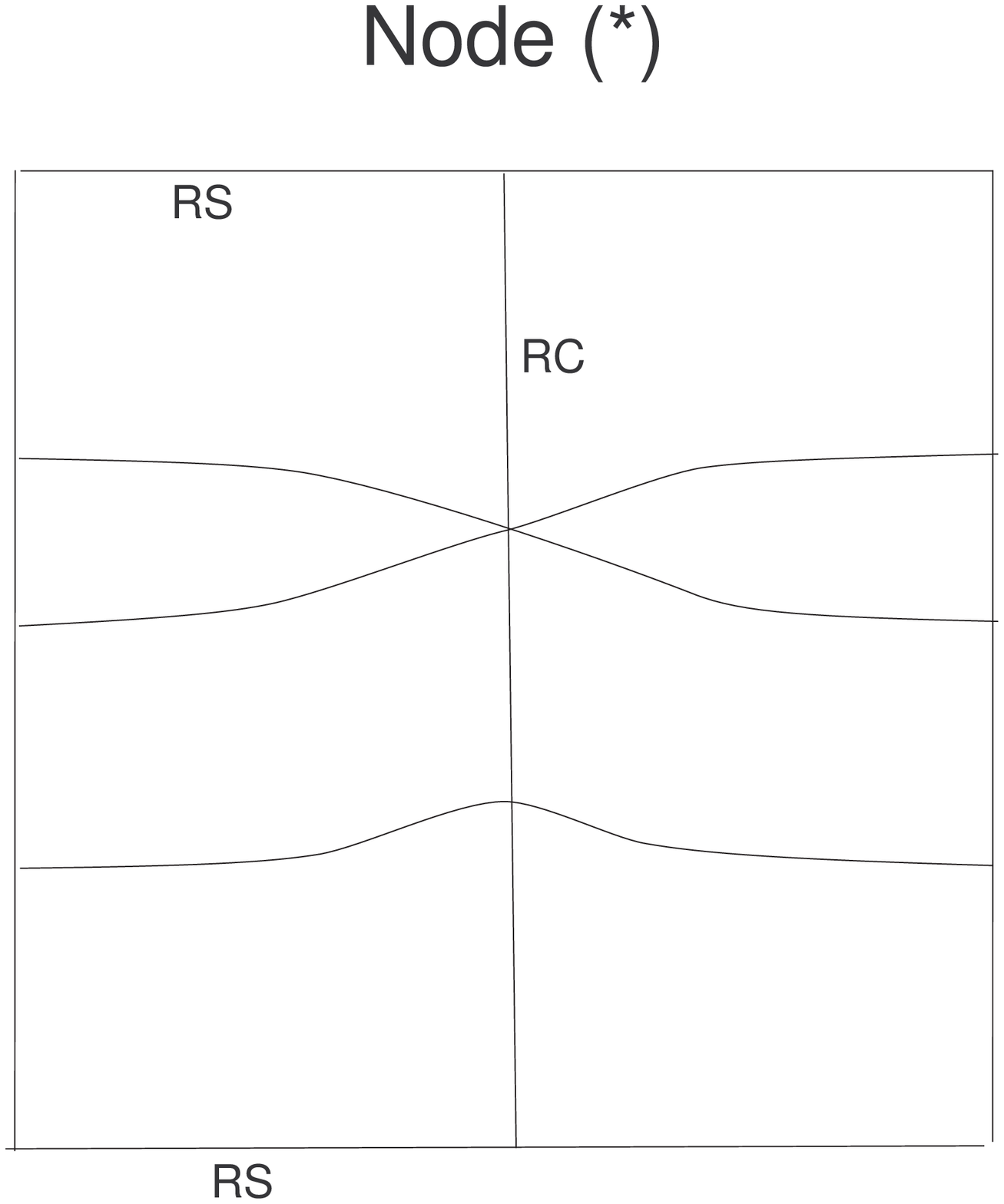}
\end{center}
\caption{{\bf Node (*)}}
\label{Case1-2}
\end{figure}


\bigskip

\noindent
{\bf Cusp case:}

{\bf Cusp (1):}\ \ \ 
When both the cusp $P_0$ and the point $P_1$ are contained in {\bf the same} connected component of $\br A^\prime_1$, 
$\br A^\prime_1$ meets $\br c$ at only $P_0$ and $P_1$, 
and the component containing $P_0$ and $P_1$ is not contractible (the leftmost figure of Figure \ref{noncon}). 
$\br A^\prime_1$ might have some ovals. 
{\bf The non-contractible component containing $P_0$ and $P_1$} and {\bf the section $\br s$} 
divide $\br \bff_4$ into two regions (the left figure of Figure \ref{Case2}). 
{\bf One of these regions goes to a non-orientable region via the blow up of $\br \bff_4$.} 

\begin{definition}[The regions $R_1$ and $R_2$ in Cusp (1) case]
Let $R_2$ denote the region which goes to a {\bf non-orientable region} via the blow up of $\br \bff_4$.
and let $R_1$ denote the other region. 
(Since $A^\prime_1 \cdot c =3$, the interior of any oval of $\br A_1$ does not contain any other ovals.)

We define the integers $\alpha$ and $\beta$ by (\ref{alpha-beta}). 
See the left figure of Figure \ref{Case2}.
\end{definition}

\begin{figure}[!h]
\begin{center}
\includegraphics[width=4.5cm]{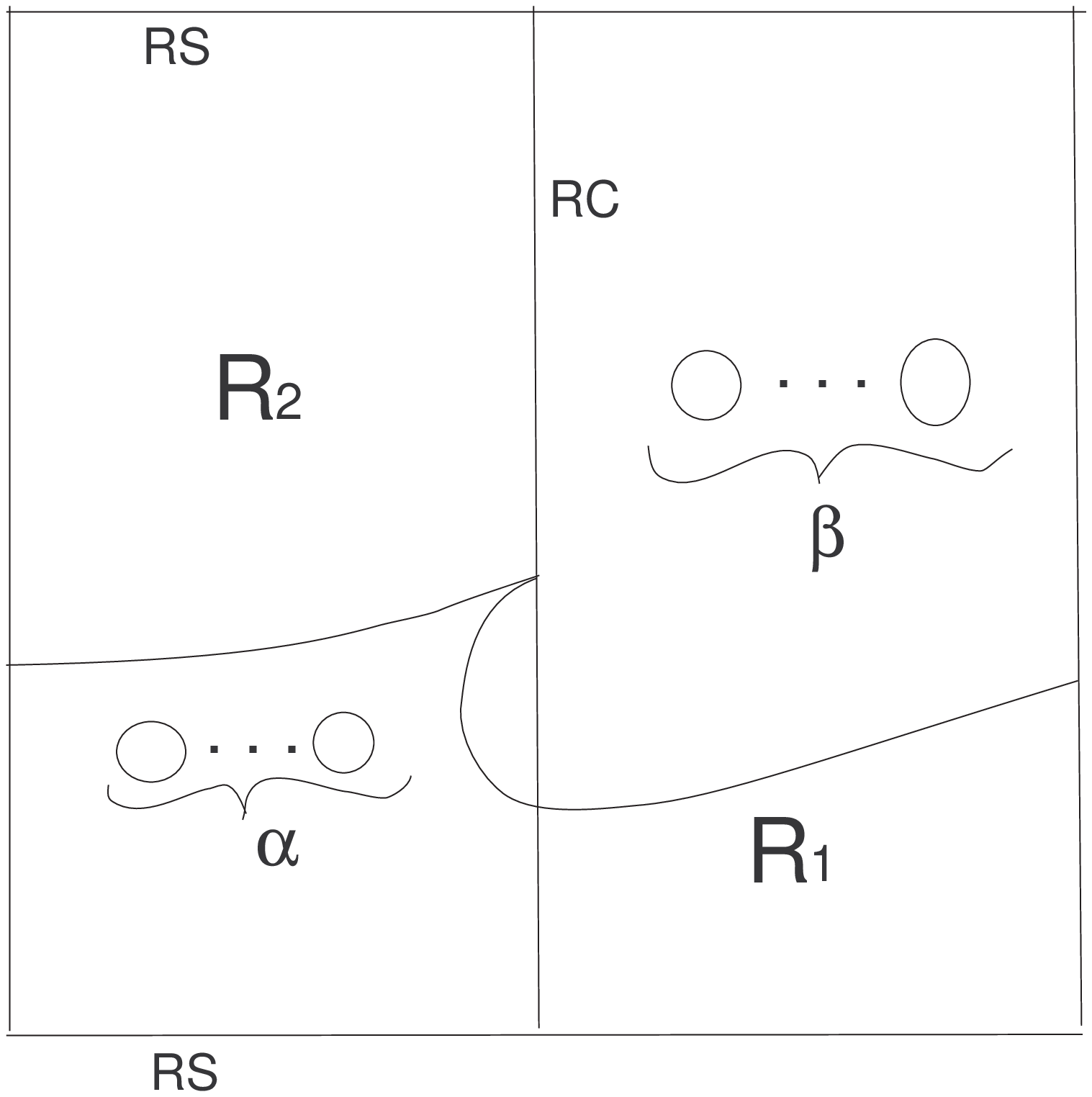}
\hspace{1.5cm}
\includegraphics[width=4.5cm]{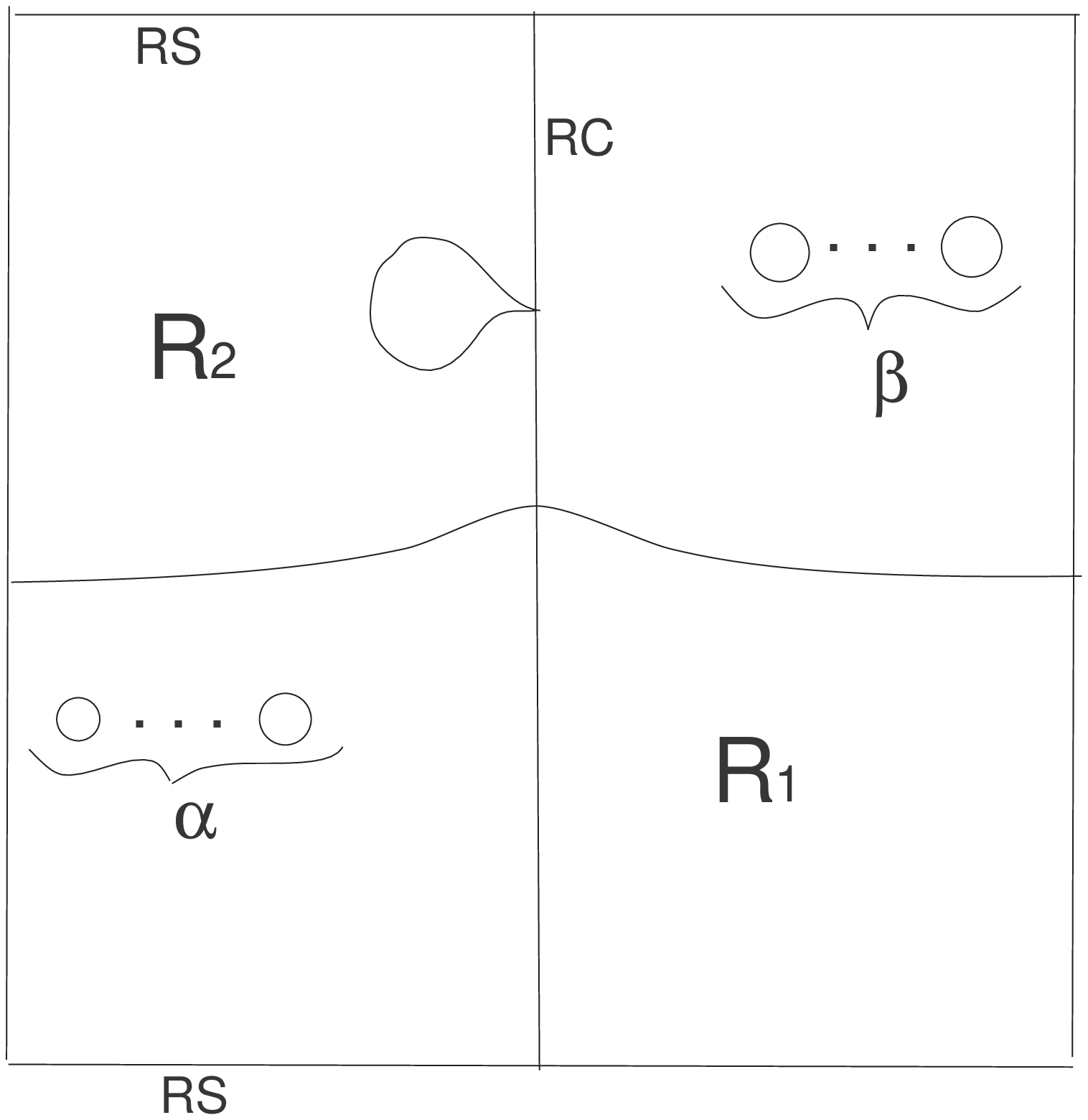}
\end{center}
\caption{The regions $R_1$ and $R_2$ in {\bf Cusp (1)} or {\bf Cusp (2)} case}
\label{Case2}
\end{figure} 

\medskip

{\bf Cusp (2)}\ \ \ 
When the cusp $P_0$ and $P_1$ respectively are contained in {\bf different} connected components of $\br A^\prime_1$, 
the component containing $P_1$ is {\bf not contractible} (the rightmost figure of Figure \ref{noncon}). 
The component containing the cusp $P_0$ should be {\bf contractible} (the right figure of Figure \ref{con}). 
(If not, then this component would be like the middle figure of Figure \ref{noncon} 
and the number of the intersection points of $\br A^\prime_1$ with $\br c$ would be even. 
This contradicts with $A^\prime_1 \cdot c = 3$.) 
$\br A^\prime_1$ might have some ovals. 
{\bf The component containing $P_1$} and {\bf the section $\br s$} 
divide $\br \bff_4$ into two regions. 
\begin{definition}[The regions $R_1$ and $R_2$ in Cusp (2) case]
Let $R_1$ denote the region which does not contain the contractible component which contains the cusp $P_0$, 
and let $R_2$ denote the other region. (the right figure of Figure \ref{Case2}) 
(Since $A^\prime_1 \cdot c =3$, the interior of any oval of $\br A_1$ does not contain any ovals.)

We define the integers $\alpha$ and $\beta$ by (\ref{alpha-beta}). 
See the right figure of Figure \ref{Case2} above.
\end{definition}

\bigskip

\noindent
{\bf Isolated point case:}

In this case the connected component containing $P_1$ is nonsingular and non-contractible like the right figure of Figure \ref{noncon}. 
$\br A^\prime_1$ might have some ovals. 
{\bf The component containing $P_1$} and {\bf the section $\br s$} 
divide $\br \bff_4$ into two regions. 
\begin{definition}[The regions $R_1$ and $R_2$ in Isolated point case]
Let $R_1$ denote the region which {\bf does not contain the isolated point}, 
and let $R_2$ denote the other region. 
(Since $A^\prime_1 \cdot c =3$, the interior of any oval of $\br A_1$ does not contain any other ovals.)

We define the integers $\alpha$ and $\beta$ by (\ref{alpha-beta}). 
See Figure \ref{Case3}. 
\end{definition}

\begin{figure}[!h]
\begin{center}
\includegraphics[width=4cm]{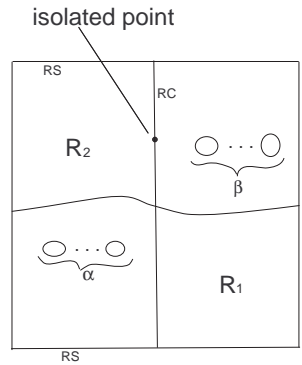}
\end{center}
\caption{The regions $R_1$ and $R_2$ in {\bf Isolated point} case}
\label{Case3}
\end{figure}

\clearpage

From the above argument, we have:
\begin{proposition}\label{allcases_F_4}
The real isotopy type of the singular connected component, 
which has one double point, 
of the curve $\br A^\prime_1$ on $\br \bff_4$ 
is one of the following $6$ types: 
\begin{itemize}
\item {\rm Node (1)}, {\rm Node (2)}, {\rm Node (*)},
\item {\rm Cusp (1)}, {\rm Cusp (2)}, or
\item {\rm Isolated point}.
\end{itemize}
See {\sc Table \ref{allcases_F_4_table}}. $\Box$

\begin{table}[!h]
\begin{center}
\begin{tabular}{|c|l|l|l|}
\hline
{\bf I}   & {\rm Node (1)} \includegraphics[width=2.5cm]{node1-dege.eps} & {\rm Cusp (1)} \includegraphics[width=2.5cm]{Case2-1.eps} & {\rm Isolated point} \includegraphics[width=2.5cm]{isolated1.eps} \\ \hline
\multicolumn{4}{c}{}    \\
\cline{1-3}
{\bf II}  & {\rm Node (2)} \includegraphics[width=2.5cm]{node2-dege.eps} & {\rm Cusp (2)} \includegraphics[width=2.5cm]{Case2-2.eps} & \multicolumn{1}{c}{} \\
\cline{1-3}
\multicolumn{4}{c}{}    \\
\cline{1-2}
{\bf III} & {\rm Node (*)} \includegraphics[width=2.5cm]{node-ast-dege.eps} & \multicolumn{2}{c}{} \\
\cline{1-2}
\end{tabular}
\end{center}
\caption{Real isotopy types of the singular component of the curve $\br A^\prime_1$.}
\label{allcases_F_4_table}
\end{table}
\end{proposition}

\medskip

\subsection{Topology of the real parts of K3 surfaces $X$ viewed via the blow up $Y \to \bff_4$}

We determine the topology of the real parts $X_\varphi(\br)$ and $X_{\wvarphi}(\br)$ of 
a real $2$-elementary K3 surface $(X,\tau,\varphi)$ of type $(S,\theta) \cong ((3,1,1),- \id)$, 
and the real part $Y(\br)$ of the quotient surface $Y$ with the real structure $\varphi_{\mathrm{mod}\ \tau}$. 
Recall Proposition \ref{allcases_F_4}. 

\medskip

\noindent
{\bf I.}\ {\bf Node (1)},\ {\bf Cusp (1)}\ and \ {\bf Isolated point} cases.

In these cases, by the definitions of the regions $A_\pm$ and $R_1,\ R_2$, 
we see that: 
\begin{itemize}
\item $A_+$ is homeomorphic to the disjoint union of (an annulus with $\alpha$ holes) and ($\beta$ disks), and
\item 
$A_-$ is homeomorphic to 
the disjoint union of (((an annulus $\setminus D^2)\cup$\ {\bf M\"obius band}) with $\beta$ holes) and ($\alpha$ disks). 
\end{itemize}

\noindent
For example, see Figure \ref{Case1-1_A+A-} for {\bf Node (1)} case.

\begin{figure}[!h]
\begin{center}
\includegraphics[width=6.5cm]{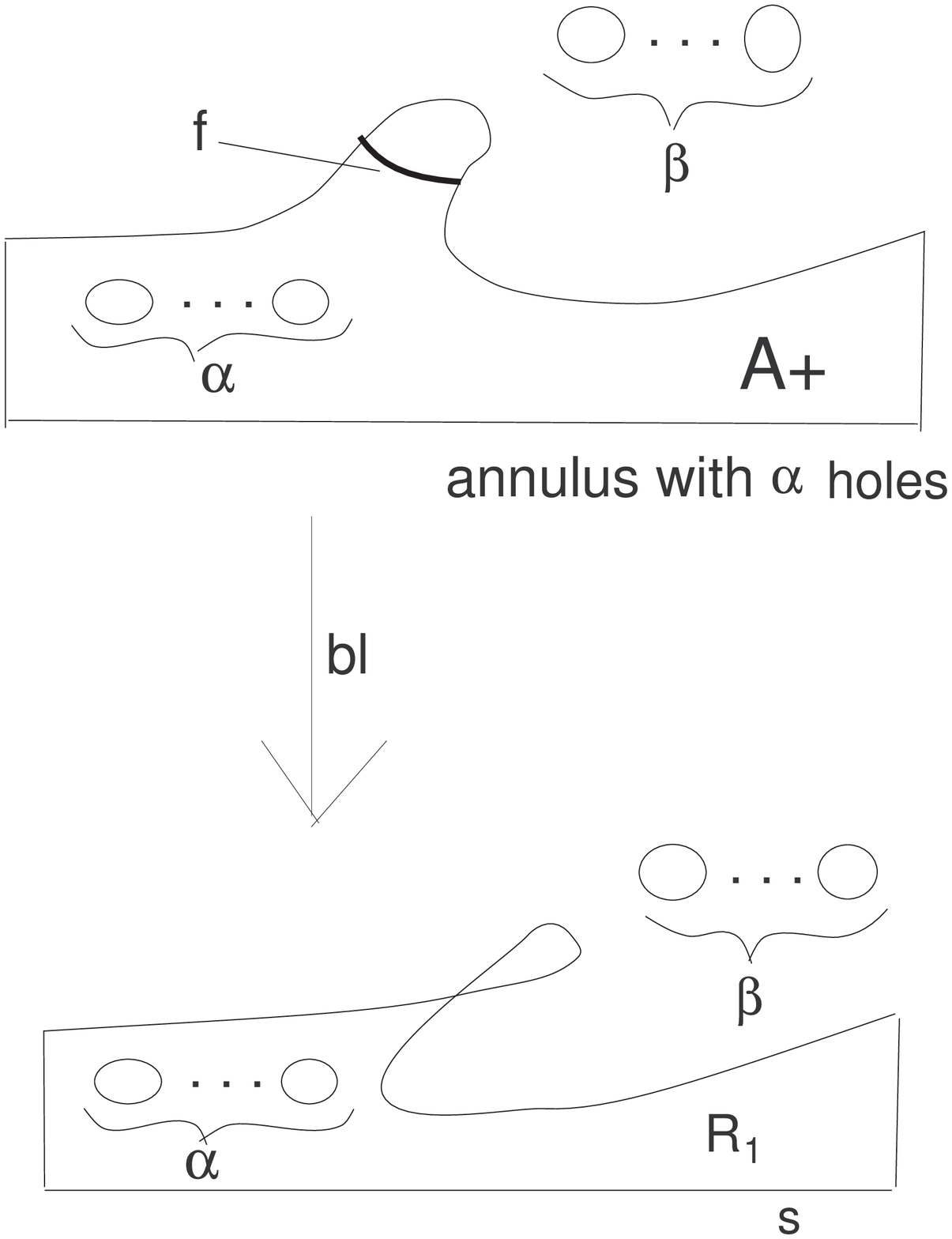}
\hspace{2cm}
\includegraphics[width=6.5cm]{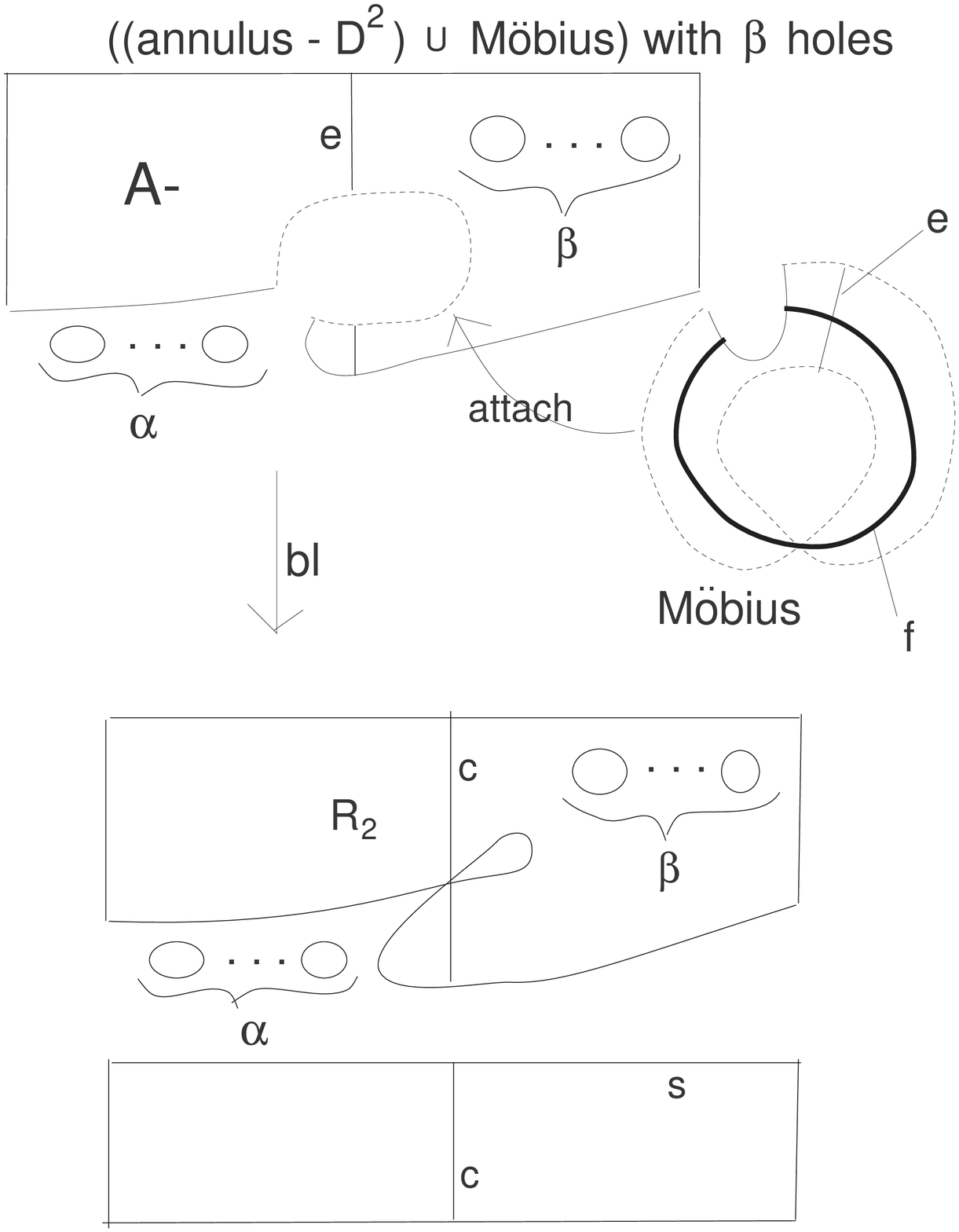}
\end{center}
\caption{The regions $A_+$ and $A_-$ in {\bf Node (1)} case}
\label{Case1-1_A+A-}
\end{figure}   

\medskip

Suppose that the invariant $H(\psi) = 0$ for the involution $\varphi$. 

If $F$ is {\bf irreducible}, 
then we have 
$[X_\varphi(\br) \cap F] \neq 0$ and $\pi(X_\varphi(\br))=A_-$. 
On the other hand, for $\wvarphi$, we have 
$[X_{\wvarphi}(\br) \cap F] = 0$ and $\pi(X_{\wvarphi}(\br))=A_+$. 

We can say that 
$$
\begin{array}{lcl}
\alpha & = & \# \{ \text{ovals\ whose\ {\bf interiors}\ are\ contained\ in}\ \BL(A_-) \},\ \text{and}\\
\beta  & = & \# \{ \text{ovals\ whose\ {\bf interiors}\ are\ contained\ in}\ \BL(A_+) \}.
\end{array}
$$

Thus we have
$$
X_\varphi(\br)    \sim    \Sigma_{2+ \beta} \cup \alpha S^2.
$$

Moreover, we have 
$(r(\psi),a(\psi),\delta_{\psi S}) \neq (10,10,0),\ (10,8,0)$,
$$
r(\psi)=9+\alpha -\beta,\ \ \ a(\psi)=9-\alpha -\beta .
$$

On the other hand, we have 
$$
X_{\wvarphi}(\br)    \sim    \Sigma_{1+ \alpha} \cup \beta S^2.
$$
Hence, we have $H(\sigma \circ \psi) \cong \bz/2\bz$    
and 
$$
r(\sigma \circ \psi)=10-\alpha +\beta,\ \ \ a(\sigma \circ \psi)=10-\alpha -\beta .$$

\medskip

We omit the cusp and isolated point cases. 

\medskip

\noindent
{\bf II.}\ {\bf Node (2)} and {\bf Cusp (2)} cases.

In these cases, by the definitions of the regions $A_\pm$ and $R_1,\ R_2$, we see that

\begin{itemize}
\item 
$A_+$ is homeomorphic to the disjoint union of (an annulus with $\alpha$ holes) and ($(\beta +1)$ disks), and 
\item 
$A_-$ is homeomorphic to 
the disjoint union of 
(((an annulus $\setminus D^2) \cup$\ {\bf M\"obius band}) with $(\beta +1)$ holes) and ($\alpha$ disks). 
\end{itemize}
For example, See Figure \ref{Case1-2-1_A+A-} for {\bf Node (2)} case. 
\begin{figure}[!h]
\begin{center}
\includegraphics[width=6.5cm]{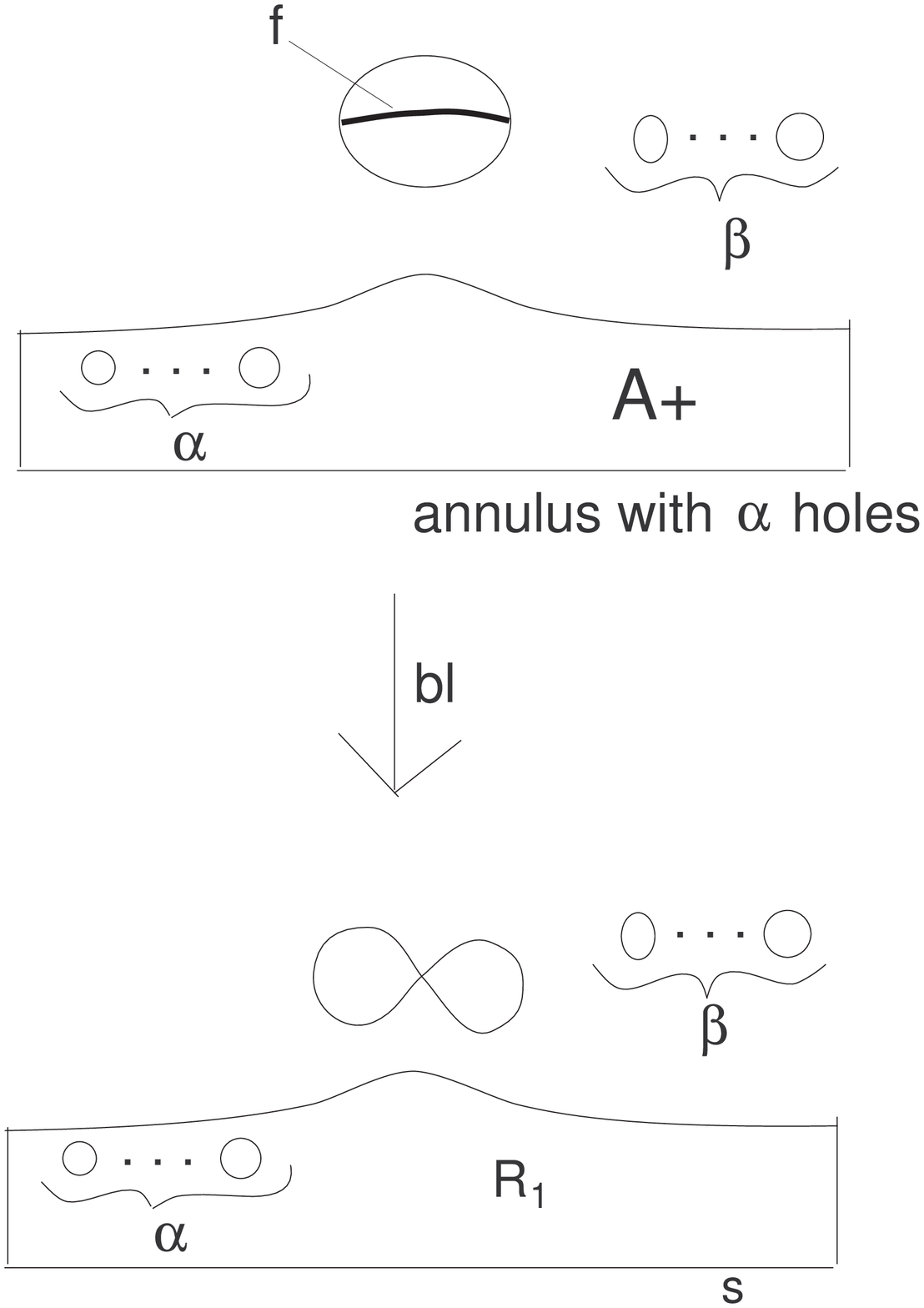}
\hspace{2cm}
\includegraphics[width=6.5cm]{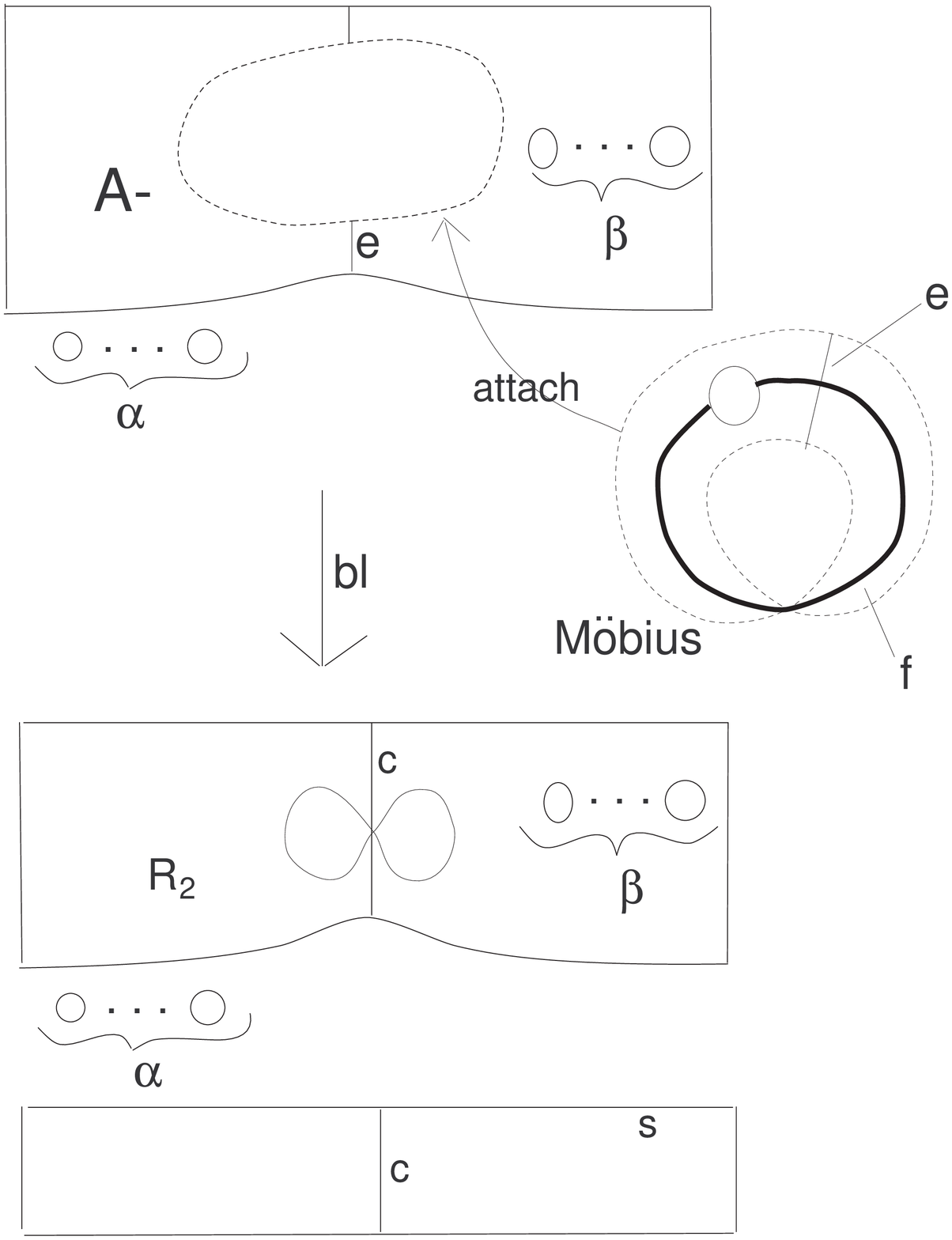}
\end{center}
\caption{The regions $A_+$ and $A_-$ in {\bf Node (2)} case}
\label{Case1-2-1_A+A-}
\end{figure}    

Suppose that $A_- = \pi(X_\varphi(\br))$, namely, 
the invariant $H(\psi) = 0$ for the involution $\varphi$. 
Then we have
$$
X_\varphi(\br)    \sim    \Sigma_{2+(\beta +1)} \cup \alpha S^2.
$$

Moreover, we have 
$(r(\psi),a(\psi),\delta_{\psi S}) \neq (10,10,0),\ (10,8,0)$, 
$$
r(\psi)=8 + \alpha - \beta,\ \ \ a(\psi)=8 - \alpha - \beta .
$$

On the other hand, we have 
$A_+ = \pi(X_{\wvarphi}(\br))$ and 
$$
X_{\wvarphi}(\br)    \sim    \Sigma_{1+ \alpha} \cup (\beta +1) S^2,
$$
Moreover, we have $H(\sigma \circ \psi) \cong \bz/2\bz$   
and 
$$
r(\sigma \circ \psi)=11 - \alpha + \beta,\ \ \ a(\sigma \circ \psi)=9 - \alpha - \beta .
$$

We omit the cusp cases.

\bigskip

\noindent
{\bf III.}\ {\bf Node (*)} case.

In this case, we see that
\begin{itemize}
\item $A_+$ is homeomorphic to $D^2 \setminus 2D^2$ and 
\item $A_-$ is the disjoint union of an {\bf M\"obius band} and an annulus. 
\end{itemize}
See Figure \ref{Case1-2-2_A+A-}. 
\begin{figure}[!h]
\begin{center}
\includegraphics[width=6.5cm]{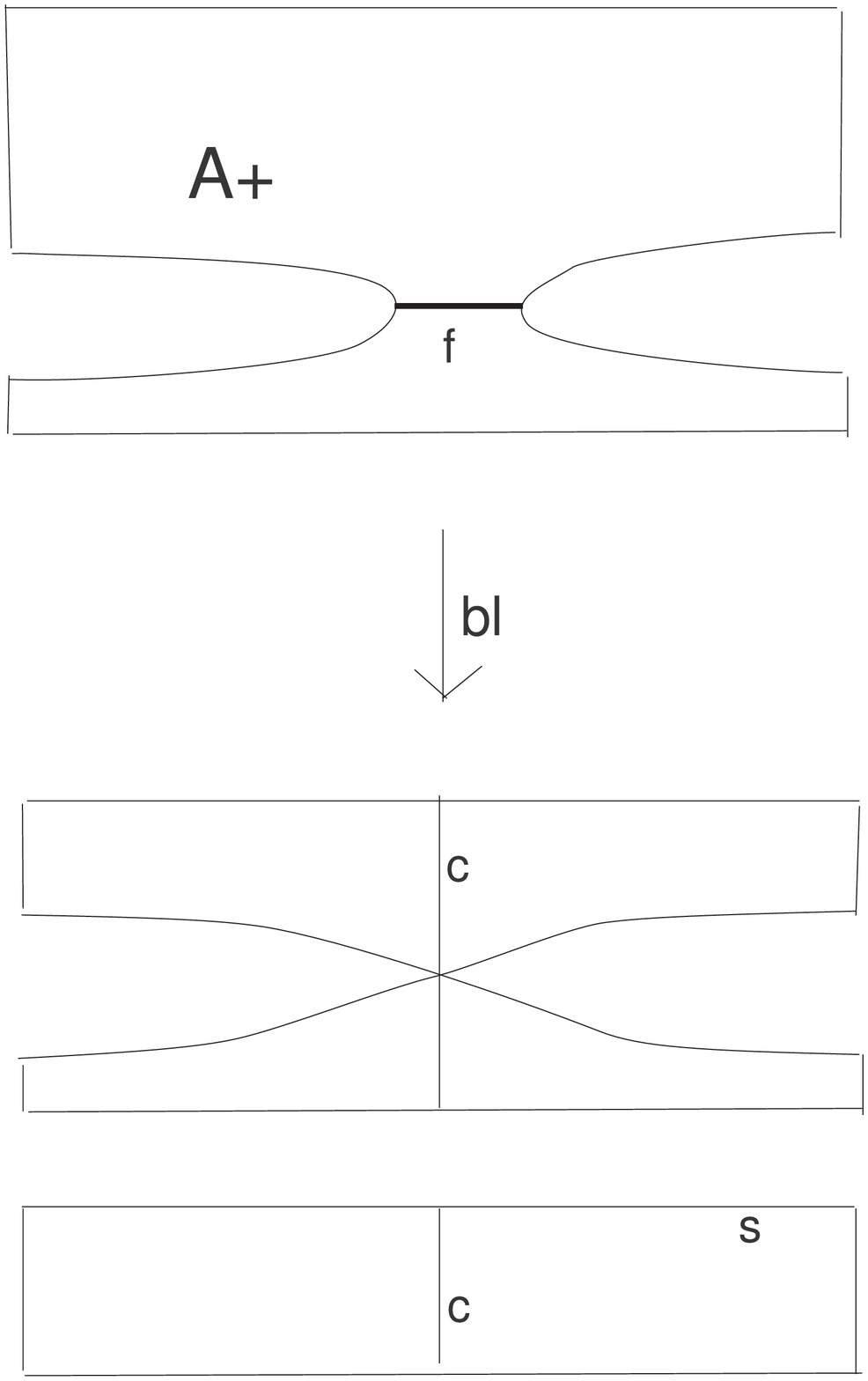}
\hspace{2cm}
\includegraphics[width=6.5cm]{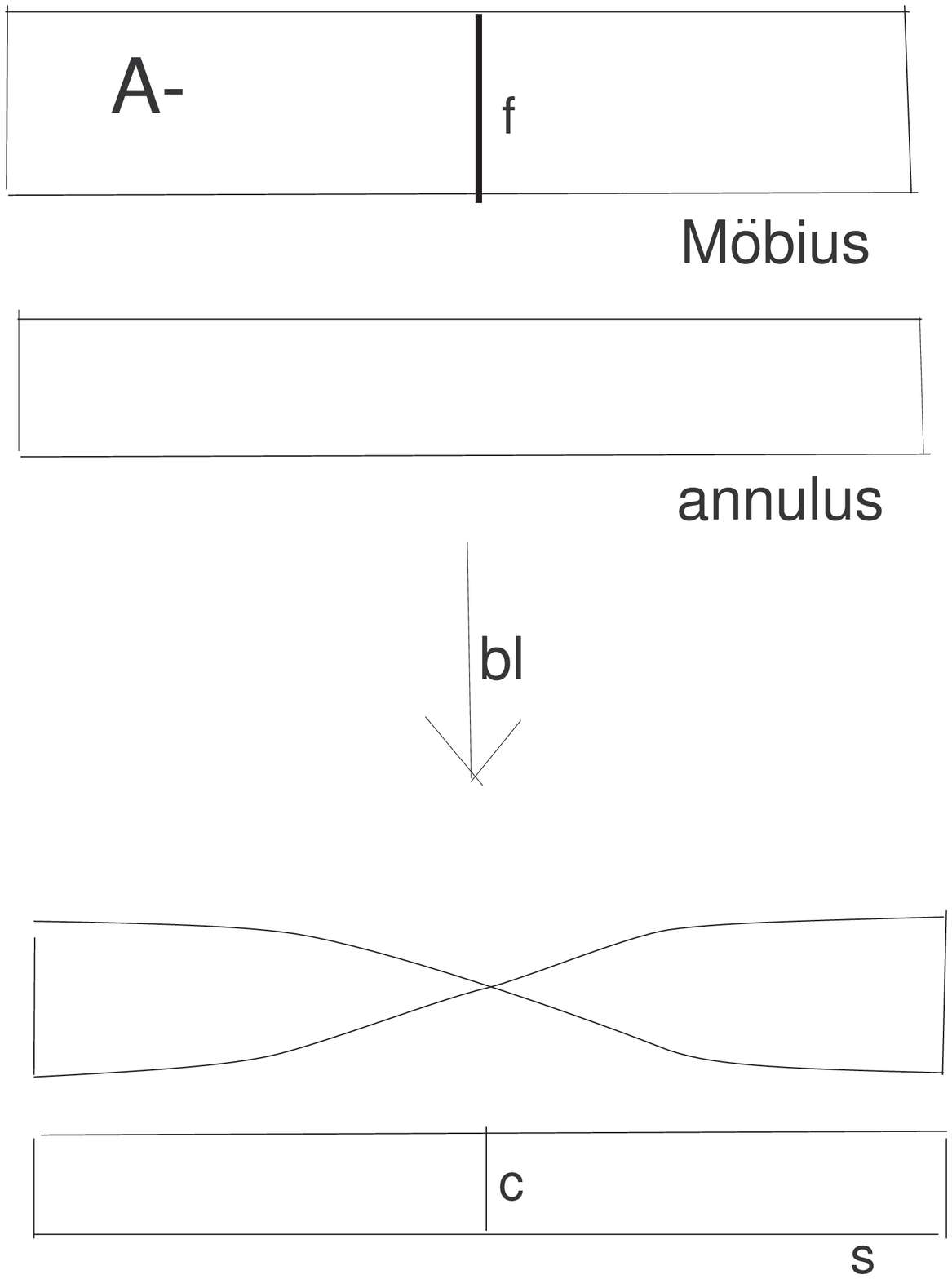}
\end{center}
\caption{The regions $A_+$ and $A_-$ in {\bf Node (*)} case}
\label{Case1-2-2_A+A-}
\end{figure}

Suppose that, for $\varphi$, $A_- = \pi(X_\varphi(\br))$. 
Then we see that
$$
X_\varphi(\br)    \sim    T^2 \cup T^2 ,
$$
and 
$$
A_+ = \pi(X_{\wvarphi}(\br)) , \ \ 
X_{\wvarphi}(\br) \sim \Sigma_2. 
$$
Moreover, we have 
$$
H(\psi) = 0 \ \ \mbox{and}\ \ 
(r(\psi),a(\psi),\delta_{\psi S}) = (10,8,0) .
$$

On the other hand, we have $H(\sigma \circ \psi) \cong \bz/2\bz$ and 
$$
(r(\sigma \circ \psi),a(\sigma \circ \psi),\delta_{\sigma \circ \psi S})=(9,9,0) .
$$

\subsection{Real isotopy types of real anti-bicanonical curves $\br \BL(A)$ 
with one real double point on $\br A^\prime_1$ on $\br \bff_4$}

In Node and Cusp cases, the number of connected components of $\br A^\prime_1$ 
equals to that of connected components of $\br A_1$. Hence, we have
$$
1 \leq \# \{\text{Connected\ components\ of}\ \br A^\prime_1 \} \leq 10.
$$

\noindent
In {\bf Node (1)} and {\bf Cusp (1)} cases, we have 
$$
0 \leq \alpha + \beta \leq 9.
$$

If for $\varphi$, $A_- = \pi(X_\varphi(\br))$, namely, the invariant $H(\psi) = 0$, then 
$
\alpha + \beta = 9 - a(\psi).
$

\bigskip

\noindent
In {\bf Node (2)} and {\bf Cusp (2)} cases, we have 
$$
0 \leq \alpha + \beta \leq 8.
$$

If for $\varphi$, $A_- = \pi(X_\varphi(\br))$, namely, the invariant $H(\psi) = 0$, then 
$
\alpha + \beta = 8 - a(\psi).
$

\bigskip

\noindent
In Isolated point case, the number of connected components of $\br A^\prime_1$ 
equals to that of connected components of $\br A_1$ plus $1$. 

Hence, we have 
$$
2 \leq \# \{\text{Connected\ components\ of}\ \br A^\prime_1 \} \leq 11.
$$

Hence, we have 
$$
0 \leq \alpha + \beta \leq 9.
$$

If for $\varphi$, $A_- = \pi(X_\varphi(\br))$, namely, the invariant $H(\psi) = 0$, then 
$
\alpha + \beta = 8 - a(\psi).
$

\bigskip

We already have all the isometry classes. Recall {\sc Table \ref{H=0} and Table \ref{H=F}} in Subsection \ref{enumeration}. 
\begin{theorem} \label{isotopy-F4-double}
We have the following. 
\begin{itemize}
\item 
For {\bf each} isometry class with $H(\psi) = 0$, 
the real isotopy type of a real anti-bicanonical curve 
$\br \BL(A) = \br s \cup  \br A^\prime_1$ on $\br \bff_4$ 
with {\bf one real nondegenerate double point} on $\br A^\prime_1$ on $\br \bff_4$ 
is one of the data listed up in Table \ref{3-1-1-delta_F_1}. 

\item 
For {\bf each} isometry class with $H(\psi) \cong \bz/2\bz$, 
the real isotopy type of a real anti-bicanonical curve 
$\br \BL(A) = \br s \cup  \br A^\prime_1$ on $\br \bff_4$ 
with {\bf one real nondegenerate double point} on $\br A^\prime_1$ on $\br \bff_4$ 
is one of the data listed up in Table \ref{3-1-1-delta_F_0}. 
\end{itemize}
Note that the isometry class No.$k$ and the isometry class No.$k^\prime$ 
are related integral involutions for each $k=1, \dots ,50$. 
The isometry class $(10,8,0,H(\psi) = 0)$ and $(9,9,0,H(\psi) \cong \bz/2\bz)$ are also related integral involutions. $\Box$
\end{theorem}

\setlength{\topmargin}{-0.6in}
\setlength{\textheight}{10in}

\begin{table}[!ht]
\begin{center}
{\footnotesize 
\begin{tabular}{|r||r|r|c||r|r||cc|cc|cc|c|}
\hline
\multicolumn{4}{|l||}{Isometry class}&\multicolumn{2}{l||}{} 
&\multicolumn{2}{l|}{Node (1)} &\multicolumn{2}{l|}{Isolated point} & \multicolumn{2}{l|}{Node (2)} & Node (*) \\
\multicolumn{4}{|l||}{of type $((3,1,1),- \id)$}&\multicolumn{2}{l||}{} &\multicolumn{2}{l|}{} &\multicolumn{2}{l|}{} & \multicolumn{2}{l|}{} &  \\
\hline
\hline
No.& $r(\psi)$ & $a(\psi)$ & $\delta_{\varphi S}$ & $g$ & $k$ & $\alpha$   & $\beta$ & $\alpha$ & $\beta$ & $\alpha$ & $\beta$ &     \\ \hline
1&1 & 1 & 1 & 10 & 0 & 0 & 8 & 0 & 8 &0 & 7 & \\ \hline
2&2 & 0 & 0 & 10 & 1 & 1 & 8 & 1 & 8 &1 & 7 & \\ \hline
3&2 & 2 & 0 & 9  & 0 & 0 & 7 & 0 & 7 &0 & 6 & \\ \hline
4&2 & 2 & 1 & 9  & 0 & 0 & 7 & 0 & 7 &0 & 6 & \\ \hline
5&3 & 1 & 1 & 9  & 1 & 1 & 7 & 1 & 7 &1 & 6 & \\ \hline
6&3 & 3 & 1 & 8  & 0 & 0 & 6 & 0 & 6 &0 & 5 & \\ \hline
7&4 & 2 & 1 & 8  & 1 & 1 & 6 & 1 & 6 &1 & 5 & \\ \hline
8&4 & 4 & 1 & 7  & 0 & 0 & 5 & 0 & 5 &0 & 4 & \\ \hline
9&5 & 3 & 1 & 7  & 1 & 1 & 5 & 1 & 5 &1 & 4 & \\ \hline
10&5 & 5 & 1 & 6  & 0 & 0 & 4 & 0 & 4 &0 & 3 & \\ \hline
11&6 & 2 & 0 & 7  & 2 & 2 & 5 & 2 & 5 &2 & 4 & \\ \hline
12&6 & 4 & 0 & 6  & 1 & 1 & 4 & 1 & 4 &1 & 3 & \\ \hline
13&6 & 4 & 1 & 6  & 1 & 1 & 4 & 1 & 4 &1 & 3 & \\ \hline
14&6 & 6 & 1 & 5  & 0 & 0 & 3 & 0 & 3 &0 & 2 & \\ \hline
15&7 & 3 & 1 & 6  & 2 & 2 & 4 & 2 & 4 &2 & 3 & \\ \hline
16&7 & 5 & 1 & 5  & 1 & 1 & 3 & 1 & 3 &1 & 2 & \\ \hline
17&7 & 7 & 1 & 4  & 0 & 0 & 2 & 0 & 2 &0 & 1 & \\ \hline
18&8 & 2 & 1 & 6  & 3 & 3 & 4 & 3 & 4 &3 & 3 & \\ \hline
19&8 & 4 & 1 & 5  & 2 & 2 & 3 & 2 & 3 &2 & 2 & \\ \hline
20&8 & 6 & 1 & 4  & 1 & 1 & 2 & 1 & 2 &1 & 1 & \\ \hline
21&8 & 8 & 1 & 3  & 0 & 0 & 1 & 0 & 1 &0 & 0 & \\ \hline
22&9 & 1 & 1 & 6  & 4 & 4 & 4 & 4 & 4 &4 & 3 & \\ \hline
23&9 & 3 & 1 & 5  & 3 & 3 & 3 & 3 & 3 &3 & 2 & \\ \hline
24&9 & 5 & 1 & 4  & 2 & 2 & 2 & 2 & 2 &2 & 1 & \\ \hline
25&9 & 7 & 1 & 3  & 1 & 1 & 1 & 1 & 1 &1 & 0 & \\ \hline
26&9 & 9 & 1 & 2  & 0 & 0 & 0 & 0 & 0 &  &   & \\ \hline
27&10 & 0 & 0 & 6 & 5 & 5 & 4 & 5 & 4 &5 & 3 & \\ \hline
28&10 & 2 & 0 & 5 & 4 & 4 & 3 & 4 & 3 &4 & 2 & \\ \hline
29&10 & 2 & 1 & 5 & 4 & 4 & 3 & 4 & 3 &4 & 2 & \\ \hline
30&10 & 4 & 0 & 4 & 3 & 3 & 2 & 3 & 2 &3 & 1 & \\ \hline
31&10 & 4 & 1 & 4 & 3 & 3 & 2 & 3 & 2 &3 & 1 & \\ \hline
32&10 & 6 & 0 & 3 & 2 & 2 & 1 & 2 & 1 &2 & 0 & \\ \hline
33&10 & 6 & 1 & 3 & 2 & 2 & 1 & 2 & 1 &2 & 0 & \\ \hline
  &10 & 8 & 0 & 2 & 1 &   & &  & &  &   & $T^2 \cup T^2$ \\ \hline
34&10 & 8 & 1 & 2 & 1 & 1 & 0 & 1 & 0 &  &   & \\ \hline
35&11 & 1 & 1 & 5 & 5 & 5 & 3 & 5 & 3 &5 & 2 & \\ \hline
36&11 & 3 & 1 & 4 & 4 & 4 & 2 & 4 & 2 &4 & 1 & \\ \hline
37&11 & 5 & 1 & 3 & 3 & 3 & 1 & 3 & 1 &3 & 0 & \\ \hline
38&11 & 7 & 1 & 2 & 2 & 2 & 0 &  2 & 0 &  &   & \\ \hline
39&12 & 2 & 1 & 4 & 5 & 5 & 2 & 5 & 2 &5 & 1 & \\ \hline
40&12 & 4 & 1 & 3 & 4 & 4 & 1 & 4 & 1 &4 & 0 & \\ \hline
41&12 & 6 & 1 & 2 & 3 & 3 & 0 & 3 & 0 &  &   & \\ \hline
42&13 & 3 & 1 & 3 & 5 & 5 & 1 & 5 & 1 &5 & 0 & \\ \hline
43&13 & 5 & 1 & 2 & 4 & 4 & 0 & 4 & 0 &  &   & \\ \hline
44&14 & 2 & 0 & 3 & 6 & 6 & 1 & 6 & 1 &6 & 0 & \\ \hline
45&14 & 4 & 0 & 2 & 5 & 5 & 0 & 5 & 0 &  &   & \\ \hline
46&14 & 4 & 1 & 2 & 5 & 5 & 0 & 5 & 0 &  &   & \\ \hline
47&15 & 3 & 1 & 2 & 6 & 6 & 0 & 6 & 0 &  &   & \\ \hline
48&16 & 2 & 1 & 2 & 7 & 7 & 0 & 7 & 0 &  &   & \\ \hline
49&17 & 1 & 1 & 2 & 8 & 8 & 0 & 8 & 0 &  &   & \\ \hline
50&18 & 0 & 0 & 2 & 9 & 9 & 0 & 9 & 0 &  &   & \\
\hline
\end{tabular}
}
\end{center}
\caption{Candidates for real isotopy types of 
real anti-bicanonical curves $\br \BL(A)$ 
with one real {\bf nondegenerate} double point on $\br A^\prime_1$ on $\br \bff_4$ 
for each isometry class of type $(S,\theta) \cong ((3,1,1),- \id)$ with $H(\psi) = 0$}
\label{3-1-1-delta_F_1}
\end{table}

\begin{table}[h]
\begin{center}
{\footnotesize 
\begin{tabular}{|r||r|r|c||r|r||cc|cc|cc|c|}
\hline
\multicolumn{4}{|l||}{Isometry class}&\multicolumn{2}{l||}{} 
&\multicolumn{2}{l|}{Node (1)} &\multicolumn{2}{l|}{Isolated point} & \multicolumn{2}{l|}{Node (2)} & Node (*) \\
\multicolumn{4}{|l||}{of type $((3,1,1),- \id)$}&\multicolumn{2}{l||}{} &\multicolumn{2}{l|}{} &\multicolumn{2}{l|}{} & \multicolumn{2}{l|}{} &  \\
\hline
\hline
No.& $r(\psi)$ & $a(\psi)$ & $ \delta_{\varphi S}$ & $g$  & $k$ & $\alpha$ & $\beta$ & $\alpha$ & $\beta$ & $\alpha$ & $\beta$ &     \\ \hline
1'&18 & 2 & 1 & 1 & 8 & 0 & 8& 0 & 8 &0 & 7& \\ \hline
2'&17 & 1 & 0 & 2 & 8 & 1 & 8& 1 & 8 &1 & 7& \\ \hline
3'&17 & 3 & 0 & 1 & 7 & 0 & 7& 0 & 7 &0 & 6& \\ \hline
4'&17 & 3 & 1 & 1 & 7 & 0 & 7& 0 & 7 &0 & 6& \\ \hline
5'&16 & 2 & 1 & 2 & 7 & 1 & 7& 1 & 7 &1 & 6& \\ \hline
6'&16 & 4 & 1 & 1 & 6 & 0 & 6& 0 & 6 &0 & 5& \\ \hline
7'&15 & 3 & 1 & 2 & 6 & 1 & 6& 1 & 6 &1 & 5& \\ \hline
8'&15 & 5 & 1 & 1 & 5 & 0 & 5& 0 & 5 &0 & 4& \\ \hline
9'&14 & 4 & 1 & 2 & 5 & 1 & 5& 1 & 5 &1 & 4& \\ \hline
10'&14 & 6 & 1 & 1 & 4 & 0 & 4& 0 & 4 &0 & 3& \\ \hline
11'&13 & 3 & 0 & 3 & 5 & 2 & 5& 2 & 5 &2 & 4& \\ \hline
12'&13 & 5 & 0 & 2 & 4 & 1 & 4& 1 & 4 &1 & 3& \\ \hline
13'&13 & 5 & 1 & 2 & 4 & 1 & 4& 1 & 4 &1 & 3& \\ \hline
14'&13 & 7 & 1 & 1 & 3 & 0 & 3& 0 & 3 &0 & 2& \\ \hline
15'&12 & 4 & 1 & 3 & 4 & 2 & 4& 2 & 4 &2 & 3& \\ \hline
16'&12 & 6 & 1 & 2 & 3 & 1 & 3& 1 & 3 &1 & 2& \\ \hline
17'&12 & 8 & 1 & 1 & 2 & 0 & 2& 0 & 2 &0 & 1& \\ \hline
18'&11 & 3 & 1 & 4 & 4 & 3 & 4& 3 & 4 &3 & 3& \\ \hline
19'&11 & 5 & 1 & 3 & 3 & 2 & 3& 2 & 3 &2 & 2& \\ \hline
20'&11 & 7 & 1 & 2 & 2 & 1 & 2& 1 & 2 &1 & 1& \\ \hline
21'&11 & 9 & 1 & 1 & 1 & 0 & 1& 0 & 1 &0 & 0& \\ \hline
22'&10 & 2 & 1 & 5 & 4 & 4 & 4& 4 & 4 &4 & 3& \\ \hline
23'&10 & 4 & 1 & 4 & 3 & 3 & 3& 3 & 3 &3 & 2& \\ \hline
24'&10 & 6 & 1 & 3 & 2 & 2 & 2& 2 & 2 &2 & 1& \\ \hline
25'&10 & 8 & 1 & 2 & 1 & 1 & 1& 1 & 1 &1 & 0& \\ \hline
26'&10 & 10 &1 & 1 & 0 & 0 & 0& 0 & 0 &  &  & \\ \hline
27'&9 & 1 & 0 &  6 & 4 & 5 & 4& 5 & 4 &5 & 3& \\ \hline
28'&9 & 3 & 0 &  5 & 3 & 4 & 3& 4 & 3 &4 & 2& \\ \hline
29'&9 & 3 & 1 &  5 & 3 & 4 & 3& 4 & 3 &4 & 2& \\ \hline
30'&9 & 5 & 0 &  4 & 2 & 3 & 2& 3 & 2 &3 & 1& \\ \hline
31'&9 & 5 & 1 &  4 & 2 & 3 & 2& 3 & 2 &3 & 1& \\ \hline
32'&9 & 7 & 0 &  3 & 1 & 2 & 1& 2 & 1 &2 & 0& \\ \hline
33'&9 & 7 & 1 &  3 & 1 & 2 & 1& 2 & 1 &2 & 0& \\ \hline
   &9 & 9 & 0 &  2 & 0 & 1 & 0& 1 & 0 &  &  & $\Sigma_2$ \\ \hline
34'&9 & 9 & 1 &  2 & 0 & 1 & 0& 1 & 0 &  &  & \\ \hline
35'&8 & 2 & 1 &  6 & 3 & 5 & 3& 5 & 3 &5 & 2& \\ \hline
36'&8 & 4 & 1 &  5 & 2 & 4 & 2& 4 & 2 &4 & 1& \\ \hline
37'&8 & 6 & 1 &  4 & 1 & 3 & 1& 3 & 1 &3 & 0& \\ \hline
38'&8 & 8 & 1 &  3 & 0 & 2 & 0& 2 & 0 & & & \\ \hline
39'&7 & 3 & 1 &  6 & 2 & 5 & 2& 5 & 2 &5 & 1& \\ \hline
40'&7 & 5 & 1 &  5 & 1 & 4 & 1& 4 & 1 &4 & 0& \\ \hline
41'&7 & 7 & 1 &  4 & 0 & 3 & 0& 3 & 0 & & & \\ \hline
42'&6 & 4 & 1 &  6 & 1 & 5 & 1& 5 & 1 &5 & 0& \\ \hline
43'&6 & 6 & 1 &  5 & 0 & 4 & 0& 4 & 0 & & & \\ \hline
44'&5 & 3 & 0 &  7 & 1 & 6 & 1& 6 & 1 &6 & 0& \\ \hline
45'&5 & 5 & 0 &  6 & 0 & 5 & 0& 5 & 0 &  &  & \\ \hline
46'&5 & 5 & 1 &  6 & 0 & 5 & 0& 5 & 0 &  &  & \\ \hline
47'&4 & 4 & 1 &  7 & 0 & 6 & 0& 6 & 0 &  &  & \\ \hline
48'&3 & 3 & 1 &  8 & 0 & 7 & 0& 7 & 0 &  &  & \\ \hline
49'&2 & 2 & 1 &  9 & 0 & 8 & 0& 8 & 0 &  &  & \\ \hline
50'&1 & 1 & 0 & 10 & 0 & 9 & 0& 9 & 0 &  &  & \\ \hline
\end{tabular}
} 
\end{center}
\caption{Candidates for real isotopy types of 
real anti-bicanonical curves $\br \BL(A)$ 
with one real {\bf nondegenerate} double point on $\br A^\prime_1$ on $\br \bff_4$ 
for each isometry class of type $(S,\theta) \cong ((3,1,1),- \id)$ with $H(\psi) \cong \bz/2\bz$}
\label{3-1-1-delta_F_0}
\end{table}

\clearpage

\setlength{\topmargin}{-0.4in}
\setlength{\textheight}{9.5in}

\begin{remark} \label{realizability1}
By Theorem \ref{isotopy-F4-double}, 
we find that 
each isometry class of integral involutions of the K3 lattice $\bl_{K3}$ of type $(S,\theta) \cong ((3,1,1),- \id)$ 
may contain several real isotopy types (for example, Node (1), Isolated point, and Node (2)) 
of real anti-bicanonical curves $\br \BL(A)$ with one real nondegenerate double point on $\br A^\prime_1$ on $\br \bff_4$. 
Hence, the realizability of all the real isotopy types listed in Table \ref{3-1-1-delta_F_1},\ref{3-1-1-delta_F_0} 
has not been solved yet. 
However, we can take a real $2$-elementary K3 surface for which $F$ is irreducible 
in the same connected component of the moduli (see \cite{NikulinSaito07}, notes after Theorem 2). 
Hence, for each isometry class, at least one of real isotopy types is realizable. 
Especially, 
Node (*) with the isometry class $(10,8,0,H(\psi) = 0)$ is realizable. 
It is conjectured that Node (1) with $(\alpha,\beta)=(1,0)$ and Isolated point with $(\alpha,\beta)=(1,0)$ do not exist 
for the isometry class $(9,9,0,H(\psi) \cong \bz/2\bz)$. 
See also Remark \ref{interesting-correspondence} below. 
\end{remark}

\begin{remark}[\cite{NikulinSaito07}]  \label{F3} 
On the other hand, 
if we contract the exceptional curve $e=\pi(E)$ in $Y$, then 
we get a map onto the $3$-th Hirzebruch surface $\bff_3$: $\BL_1 : Y \to \bff_3 .$ 
Then $s:=\BL_1(A_0)$ is the exceptional section of $\bff_3$ with $s^2=-3$ 
and $c:=\BL_1(f)$ is a fiber. 
We have $\BL_1(A) = s + \BL_1(A_1) \in |-2K_{\bff_3}|$. 
$A_1 := \BL_1(A_1)$ is a real {\bf nonsingular} curve of genus $9$. 
The real isotopic classification of $\br A_1$ on $\br \bff_3$ 
was already done in \cite{NikulinSaito07}, Theorem 1. 
\end{remark}

\bigskip

In order to {\bf distinguish} the real isotopy types of real anti-bicanonical curves 
with one real double point on $\br \bff_4$, 
we expect that Itenberg's argument of the rigid isotopic classification of 
real curves of degree $6$ on $\br \bp^2$ with one nondegenerate double point are helpful. 
See \cite{Itenberg92},\cite{Itenberg94} and \cite{Itenberg95}, and also the last section \ref{period domain and problems} of this paper. 
This classification corresponds to 
that of real {\bf nonsingular} curves $A$ in $|-2K_{\bff_1}|$ on the first real Hirzebruch surface $\br \bff_1$ 
when we blow up $\bp^2$ at the nondegenerate double point to $\bff_1$. 
The double coverings $X$ of $\bff_1$ ramified along the nonsingular curves $A$ are 
real $2$-elementary K3 surfaces of type $(S,\theta) \cong (\langle 2 \rangle \oplus \langle -2 \rangle, - \id)$. 
Moreover, the classification of real curves of degree $6$ on $\br \bp^2$ with one nondegenerate double point 
is related to that of 
``non-increasing simplest degenerations" (conjunctions and contractions) of 
real {\bf nonsingular} curves of degree $6$ on $\br \bp^2$. See \cite{Itenberg92} and \cite{Itenberg94}. 

\medskip

Thus we next pay attention to the degenerations of real nonsingular anti-bicanonical curves on $\br \bff_4$. 

\section{Degenerations of nonsingular real anti-bicanonical curves on $\br \bff_4$}

\subsection{Review of nonsingular real anti-bicanonical curves on $\br \bff_4$}

The contents of this subsection are quoted from the last section of \cite{NikulinSaito05}. 

Let $\bu$ be the even unimodular lattice of signature $(1,1)$ (the hyperbolic plane). 
Consider real $2$-elementary K3 surfaces $(X,\tau,\varphi)$ of type $(S,\theta) \cong (\bu,- \id)$. 
All these real $2$-elementary K3 surfaces are ($\da$)-nondegenerate. 

Let $A$ be the fixed point set of $\tau$. Then $A$ is a real nonsingular curve. 
We have $Y := X/\tau = \bff_4$. 
Let $\pi:X\to \bff_4$ be the quotient map. 
We use the same symbol $A$ for its image in $\bff_4$ by $\pi$. 
Then $A \in |-2K_{\bff_4}|$. 
Let $s$ be the exceptional section with $s^2=-4$ of $\bff_4$, 
and $c$ the fiber of the fibration $f: \bff_4\to s$, where $c^2=0$. 
We have $-2K_{\bff_4} \sim 12c+4s$. 
The nonsingular curve $A$ has {\bf two irreducible components} $s$ and $A_1$; 
$$A = s \cup A_1\ \ \mbox{(disjoint union)}.$$
Conversely, 
any nonsingular curve $A_1$ in $|12c+3s|$ 
gives a nonsingular curve $A=s+A_1$ in $|-2K_{\bff_4}|$.  

We set $C:=\pi^\ast (c)$ and $E:=\pi^*(s)/2$ in $H_2(X, \bz)$. 
Then $C^2=0$, $E^2=-2$ and $C\cdot E=1$. 
$C$ and $E$ generate the fixed part of $\tau_* : H_2(X, \bz) \to H_2(X, \bz)$. 
Hence, $S \cong \bz C + \bz E \cong \bu$. 

\medskip

An isometry class of an integral involution $(\bl_{K3},\psi)$ of type $(S,\theta) \cong (\bu,- \id)$ 
is determined (see \cite{NikulinSaito05}) by the data
\begin{equation}
(r(\psi),\ a(\psi),\ \delta_\psi = \delta_{\psi S}).
\label{invF_4}
\end{equation}

The complete list of the data (\ref{invF_4}) 
is given in Section 7 (Figure 30) of \cite{NikulinSaito05}. 

There are 
$14$ isometry classes with $\delta_\psi = 0$ and
$49$ isometry classes with $\delta_\psi = 1$. Thus we have $63$ classes.

If we identify related integral involutions, there are
$10$ isometry classes with $\delta_\psi = 0$ and
$27$ isometry classes with $\delta_\psi = 1$. Thus we have $37$ classes.

\medskip

If $(r(\psi),a(\psi),\delta_\psi)=(10,10,0)$, 
then $\br \bff_4 = \emptyset .$ Hence, $\br A = \emptyset .$

If $(r(\psi),a(\psi),\delta_\psi)\not=(10,10,0)$,
then $\br \bff_4$ is not empty and homeomorphic to a $2$-torus. 
We have $\br A \supset \br s \neq \emptyset .$
The real curve $\br A_1$ is contained in the open cylinder $\br \bff_4 \setminus \br s$. 
The region $A_- = \pi(X_\varphi(\br))$ with the invariants (\ref{invF_4}) 
has the real isotopy type given in Figure \ref{Ugraph-picture}.

When $(r(\psi),a(\psi),\delta_\psi)\not=(10,10,0)$ and $\not=(10,8,0)$, 
we set 
$$g:=(22- r(\psi) - a(\psi))/2\ \ \ \mbox{and}\ \ \ k:=(r(\psi) - a(\psi))/2.$$

\medskip

\begin{figure}[!h]
\begin{center}
\includegraphics[width=5cm]{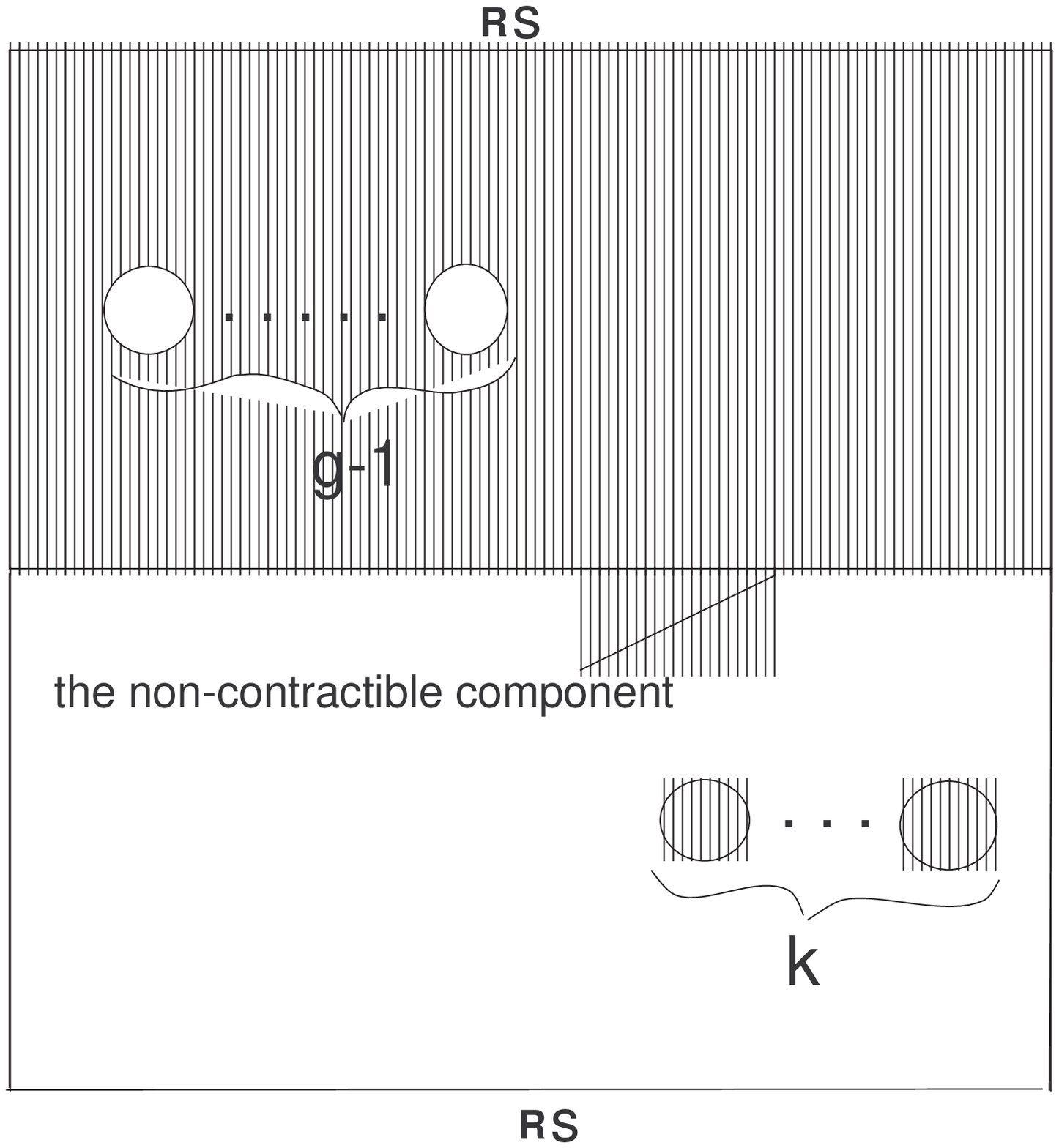}
\hspace{1.5cm} 
\includegraphics[width=5cm]{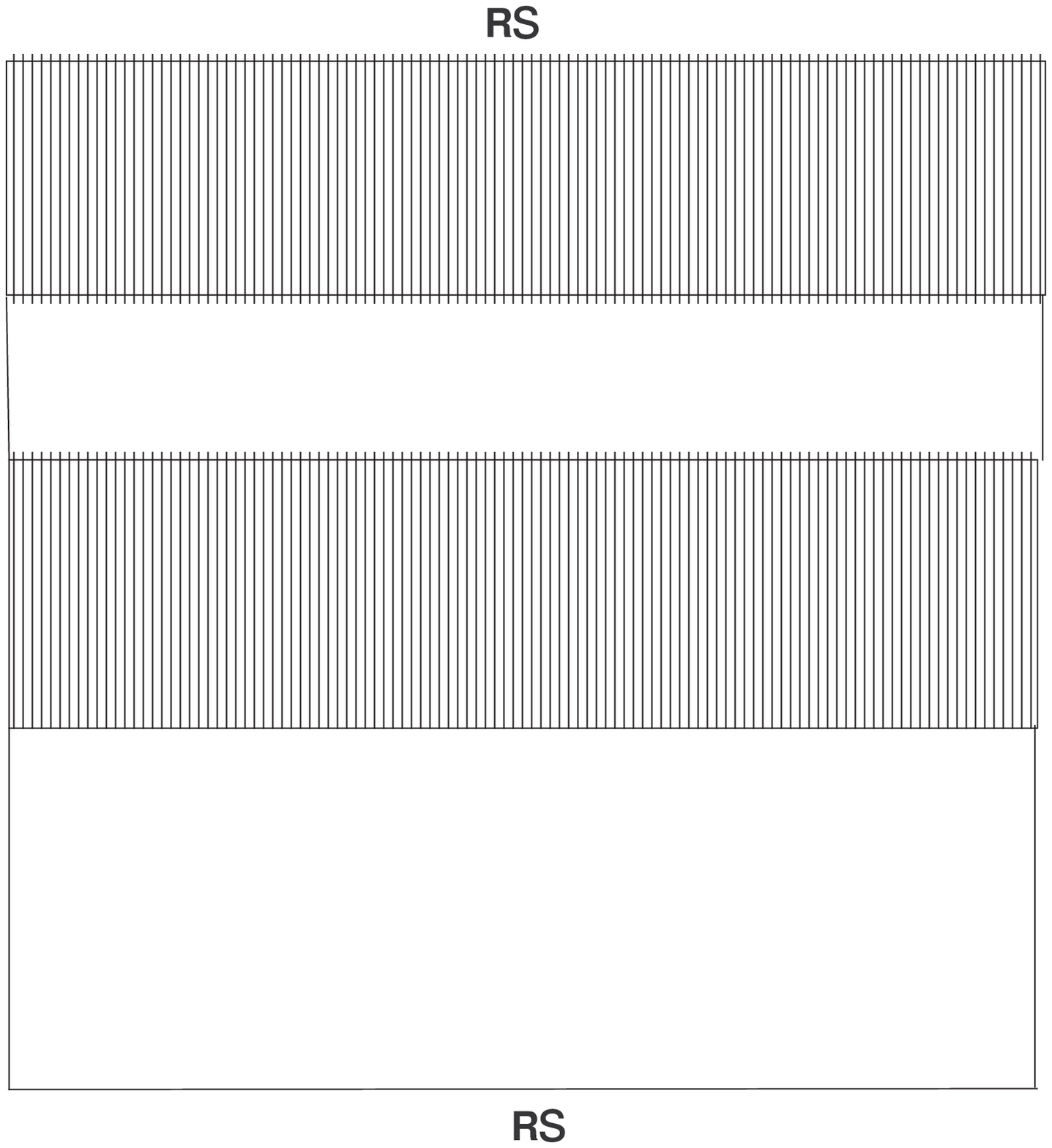} 
\end{center}
\caption{A region $A_-$ with $(r(\psi),a(\psi),\delta_\psi)\not=(10,8,0),(10,10,0)$ and 
         a region $A_-$ with $(r(\psi),a(\psi),\delta_\psi)=(10,8,0)$}
\label{Ugraph-picture}
\end{figure}

Since all real $2$-elementary K3 surfaces of type $(S,\theta) \cong (\bu,- \id)$ are ($\da$)-nondegenerate, 
by Theorem \ref{theorem2005moduli}, we have: 
\begin{theorem}[\cite{NikulinSaito05}, Theorem 27]
A connected component of the moduli of 
the regions $A_- := \pi(X_\varphi(\br))$, curves $A \in |-2K_{\bff_4}|$, 
up to the action of the automorphisms group of $\bff_4$ over $\br$ 
is determined by 
the data (\ref{invF_4}), equivalently, 
the real isotopy types of $A_-$ and the invariants $\delta_\psi = \delta_{\psi S}$. 
All the data are given in Section 7 (Figure 30) of \cite{NikulinSaito05}. 
See also Figure \ref{Ugraph-picture}.
\end{theorem}

\medskip

\subsection{Degenerations of nonsingular real anti-bicanonical curves on $\br \bff_4$}

We next introduce the notions of ``non-increasing simplest degenerations" (conjunctions, contractions) 
of {\bf nonsingular} real anti-bicanonical curves 
on $\br \bff_4$ 
as analogies of Section 3 (pp.284--285) of \cite{Itenberg92}. 

As stated in the previous subsection, 
a nonsingular curve in $|-2K_{\bff_4}|$ has two irreducible components 
$$s \ \mbox{and} \ A_1,$$
where $A_1$ is a nonsingular curve in 
$$|12c+3s|.$$        

\bigskip

Let $C_0$ be a real curve in $|12c+3s|$ on $\bff_4$ with one nondegenerate double point, 
and 
$$C_t \ \ (-\varepsilon < t < \varepsilon)$$
be a smoothing 
(, where every $C_t\ (t\neq 0)$ is a real nonsingular curve in $|12c+3s|$ on $\bff_4$, ) 
such that 
$$\sharp \{ \mbox{ovals of a nonsingular curve}\ C_{t_{-1}}\ \} \ \geq \ 
\sharp \{ \mbox{ovals of a nonsingular curve}\ C_{t_1}\ \}$$
for any $t_{-1} < 0$ and any $t_1 > 0$. 

We call such a family $C_t \ (t_{-1} \leq t \leq 0)$ 
a {\bf non-increasing simplest degeneration} of $C_{t_{-1}}$ to $C_0$. 

We do not know whether 
non-increasing simplest degenerations of $C_{t_{-1}}$ to $C_0$ are realizable 
for {\bf any pair} $C_{t_{-1}}$ and $C_0$. 

\medskip

We define $8$ kinds of non-increasing simplest degenerations. 

\begin{definition}[Conjunctions 1),\ 2),\ 1'),\ 2') and Contractions 3),\ 3')]\label{def-degenerations}
First we fix {\bf an isometry class} of integral involutions of type $(S,\theta) \cong (\bu,- \id)$ with 
$(r(\psi),a(\psi),\delta_\psi)\not=(10,8,0)$ and $\not=(10,10,0)$. 
Take a corresponding real $2$-elementary K3 surfaces 
$$(X,\tau,\varphi)$$
of type $(S,\theta) \cong (\bu,- \id)$. 
Let $\pi : X \to X/\tau = \bff_4$ be the quotient map. 
Then we get 
a real nonsingular curve $A=s+A_1$ in $|-2K_{\bff_4}|$ where 
$A_1$ is a real nonsingular curve in $|12c+3s|$ on $\br \bff_4$. 
Suppose that 
{\it 
the fixed point set $X_\varphi(\br)$ is homeomorphic to $\Sigma_g \cup kS^2$.
}
Then the region $\pi(X_\varphi(\br)) \ (\subset \br \bff_4)$ is 
the disjoint union of 
an annulus with $(g-1)$ holes  and   $k$ disks. 
The boundary of the annulus with $(g-1)$ holes 
consists of (see Figure \ref{Ugraph-picture}) 
{\it 
the non-contractible component of $\br A_1$, $\br s$, and the $(g-1)$ empty ovals.
}
\begin{itemize}
\item {\bf Conjunction 1)}\ 
The conjunction of the non-contractible component and one of the $(g-1)$ empty ovals. 
(Recall that the annulus with $(g-1)$ holes 
is covered by the fixed point set of 
the anti-holomorphic involution $\varphi$.)

\item {\bf Conjunction 1')}\ 
The conjunction of the non-contractible component and one of the $k$ empty ovals. 
(Remark that the region (annulus) surrounded by 
the non-contractible component, $\br s$ and the $k$ empty ovals 
is covered by the fixed point set of 
the related involution $\wvarphi$ of the anti-holomorphic involution $\varphi$.)

\item {\bf Conjunction 2)}\ 
The conjunction of two of the $(g-1)$ empty ovals. 

\item {\bf Conjunction 2')}\ 
The conjunction of two of the $k$ empty ovals. 

\item {\bf Contraction 3)}\ 
The contraction of one of the $(g-1)$ empty ovals. 

\item {\bf Contraction 3')}\ 
The contraction of one of the $k$ empty ovals. 
\end{itemize}
\end{definition}

\begin{definition}[Conjunctions 4),\ 4')]
Consider {\bf the isometry class} of integral involutions of type 
$(S,\theta) \cong (\bu,- \id)$ with $(r(\psi),a(\psi),\delta_\psi)$ $=$ $(9,9,1)$ or $(11,9,1)$. 
Remark that these two involutions are {\bf related}. 

If $(r(\psi),a(\psi),\delta_\psi)=(9,9,1)$, then we have 
$g=2,\ k=0$. 
Then the region $\pi(X_\varphi(\br)) \ (\subset \br \bff_4)$ is 
an annulus with one hole. 
The boundary of the annulus 
consists of 
{\it 
the non-contractible component of $\br A_1$, $\br s$, and the empty oval.  
}
\begin{itemize}
\item {\bf Conjunction 4)}\ 
The empty oval conjuncts with itself 
and becomes the union of two real lines (Node (*)) on $\br \bff_4$. 
\end{itemize}

If $(r(\psi),a(\psi),\delta_\psi)=(11,9,1)$, then we have 
$g=1,\ k=1$. 
Then the region $\pi(X_\varphi(\br)) \ (\subset \br \bff_4)$ is 
the disjoint union of 
an annulus  and   one disk. 
The boundary of the annulus 
consists of 
{\it 
the non-contractible component of $\br A_1$ and $\br s$.
}
\begin{itemize}
\item {\bf Conjunction 4')}\ 
The empty oval conjuncts with itself 
and becomes the union of two real lines (Node (*)) on $\br \bff_4$. 
\end{itemize}
\end{definition}

See Figure \ref{Degenerations}. 
Compare with Figure \ref{Ugraph-picture}.

\begin{figure}[!hb]
\begin{center}
\includegraphics[width=4cm]{node1-dege.eps}
\includegraphics[width=2.5cm]{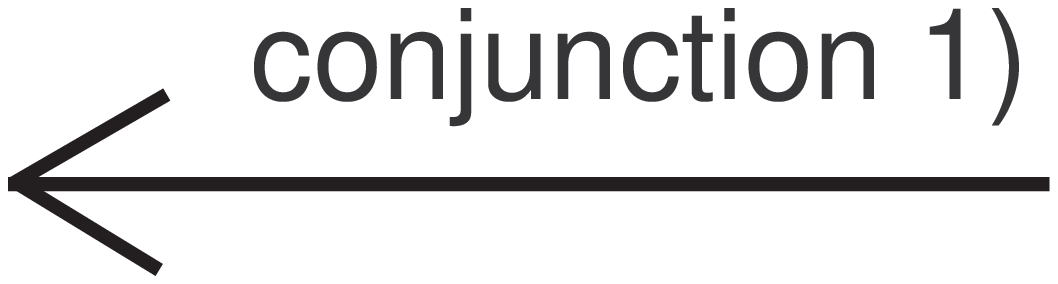}
\includegraphics[width=4cm]{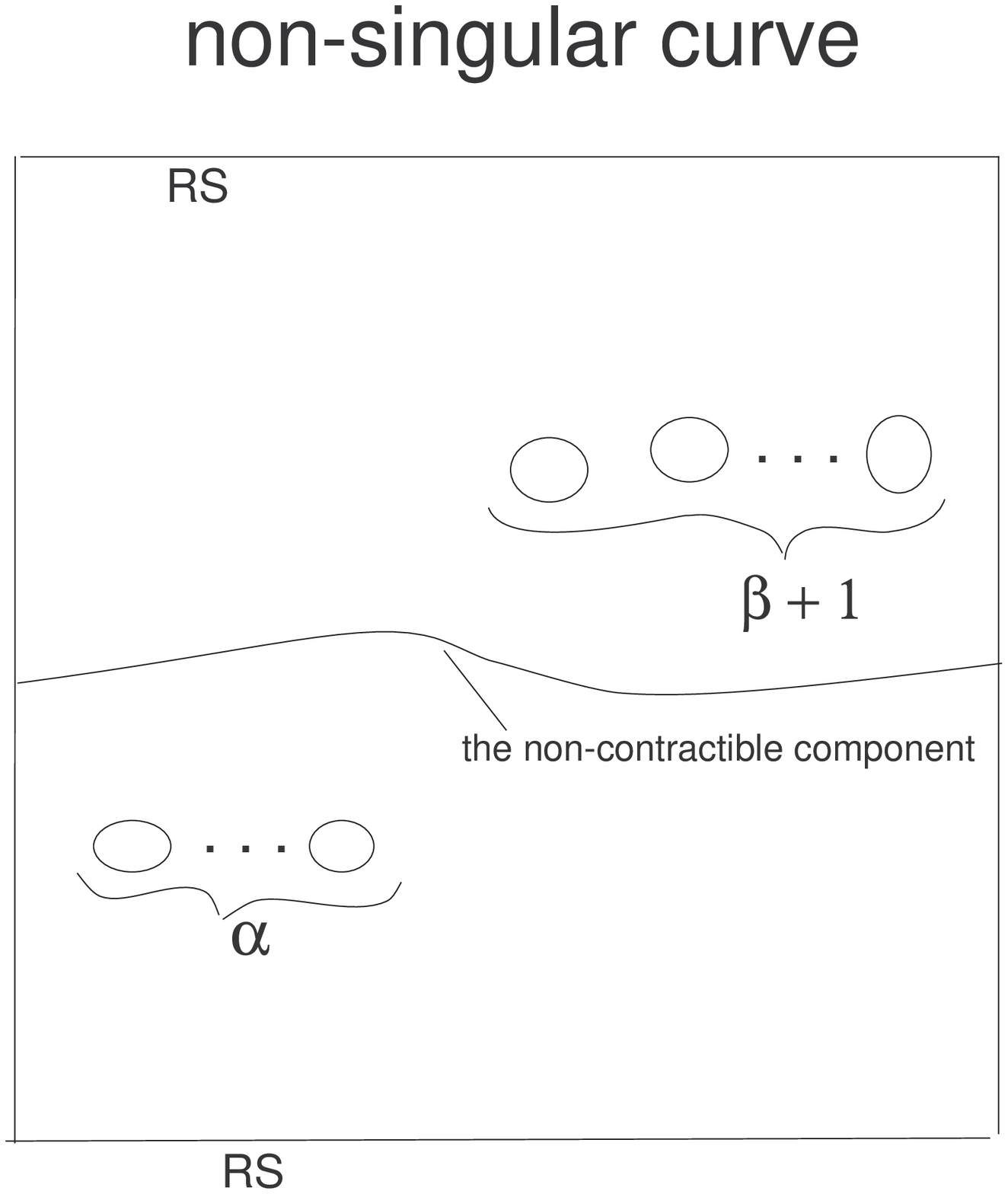}
\includegraphics[width=2.5cm]{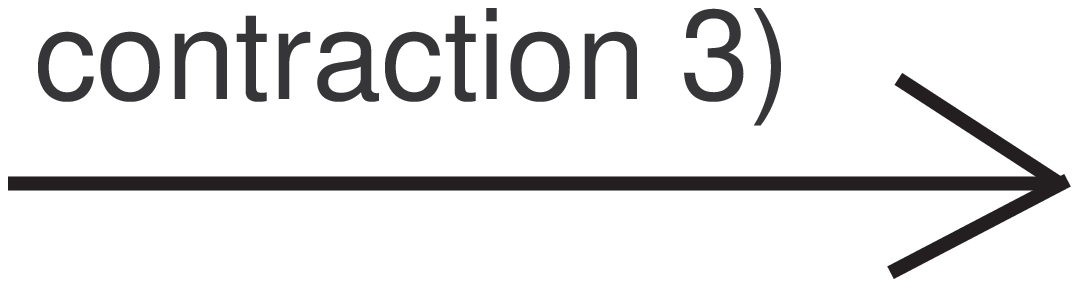}
\includegraphics[width=4cm]{isolated1.eps}

\vspace{1cm}

\includegraphics[width=4cm]{node2-dege.eps}
\includegraphics[width=2.5cm]{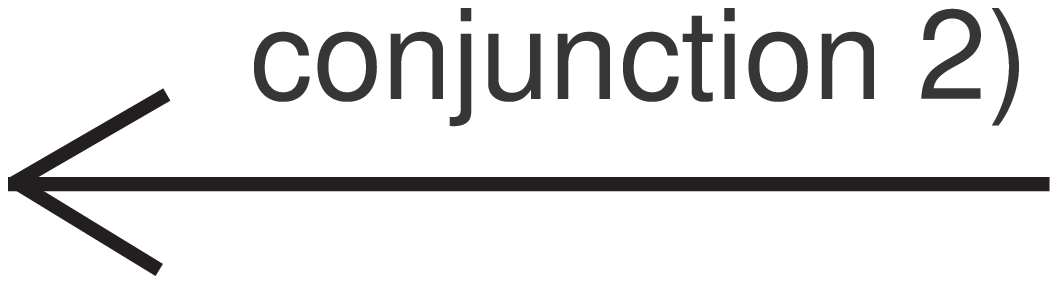}
\includegraphics[width=4cm]{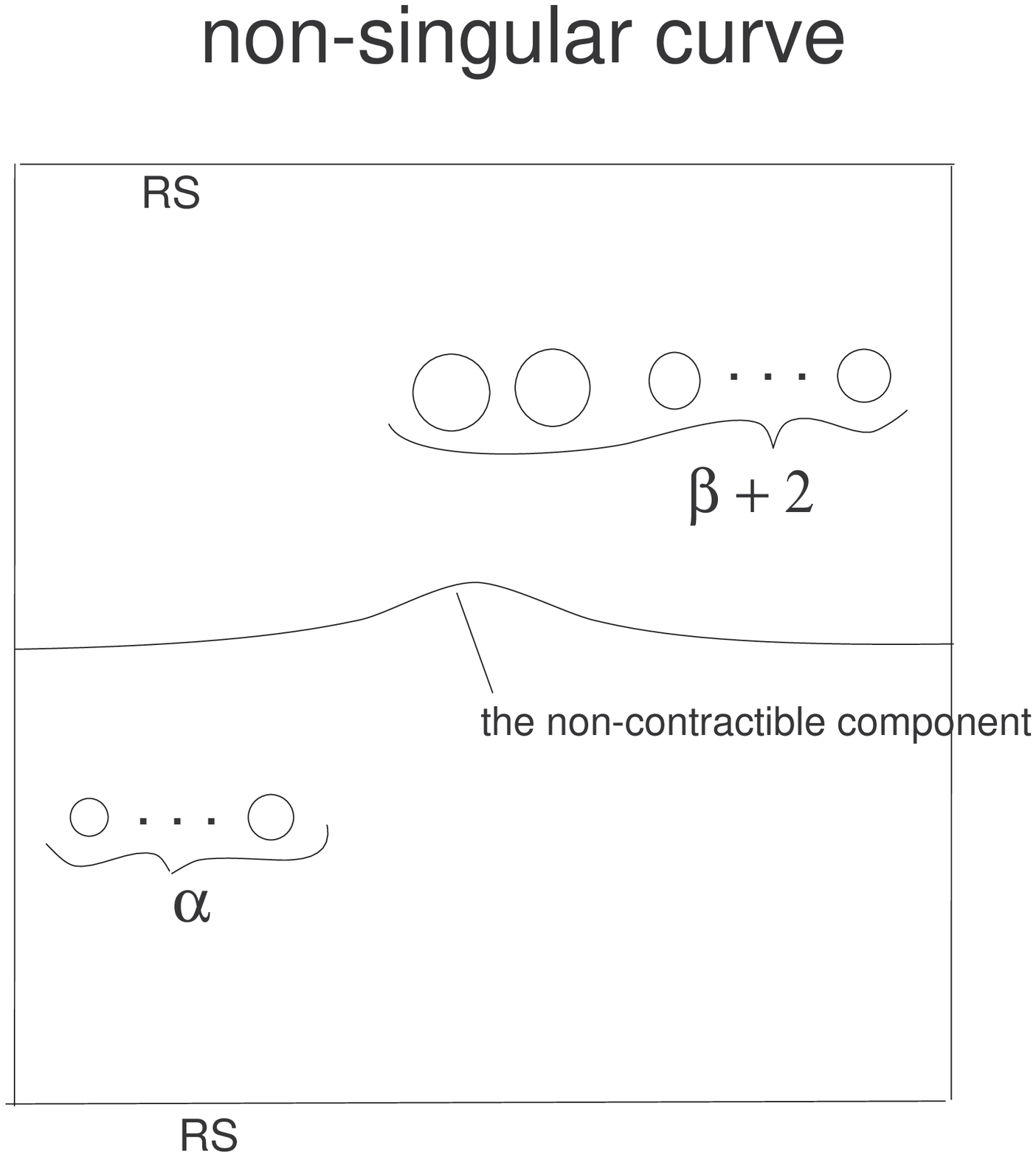}
\includegraphics[width=2.5cm]{contraction-right-3.eps}
\includegraphics[width=4cm]{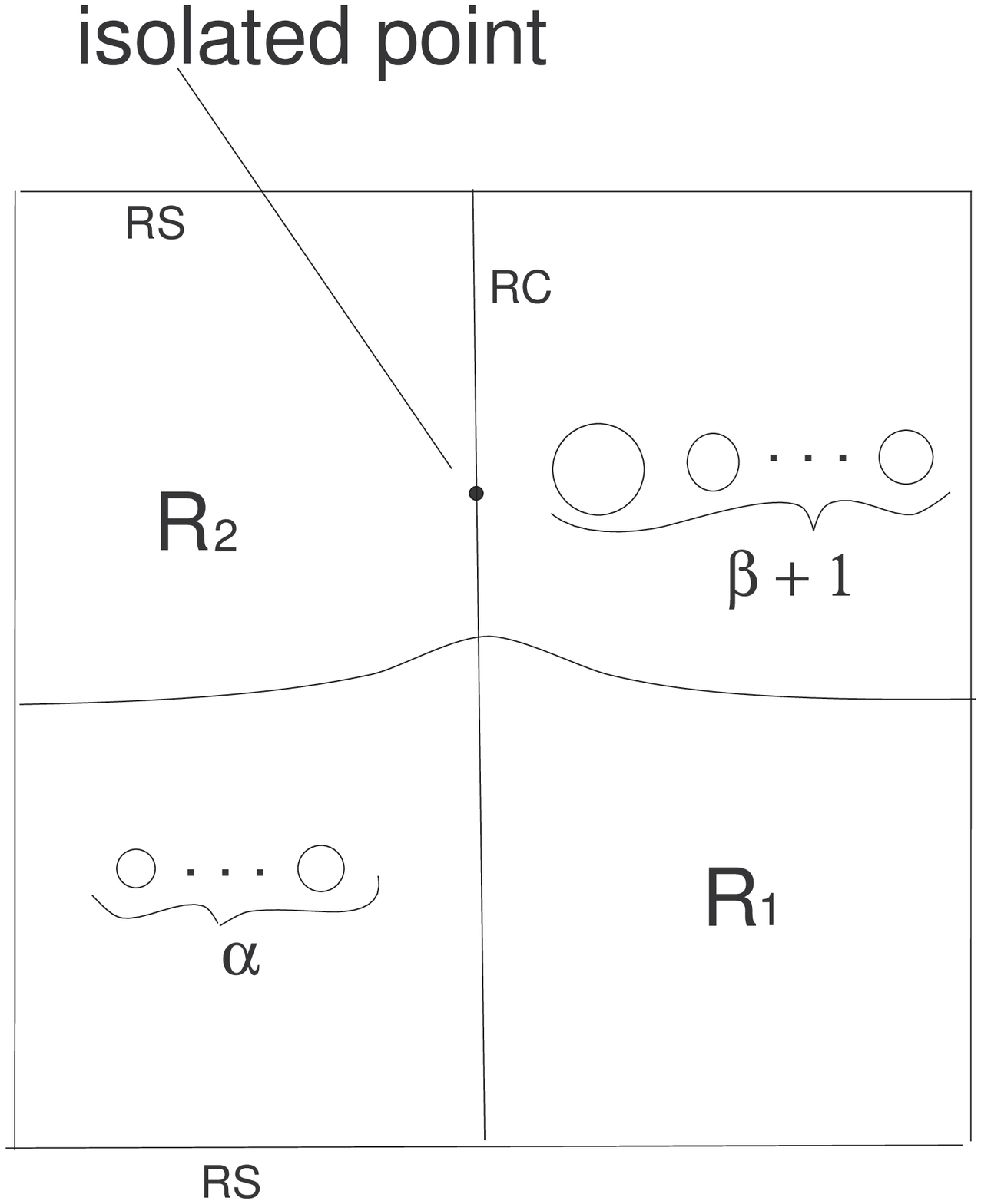}

\vspace{1cm}

\includegraphics[width=4cm]{node-ast-dege.eps}
\includegraphics[width=2.5cm]{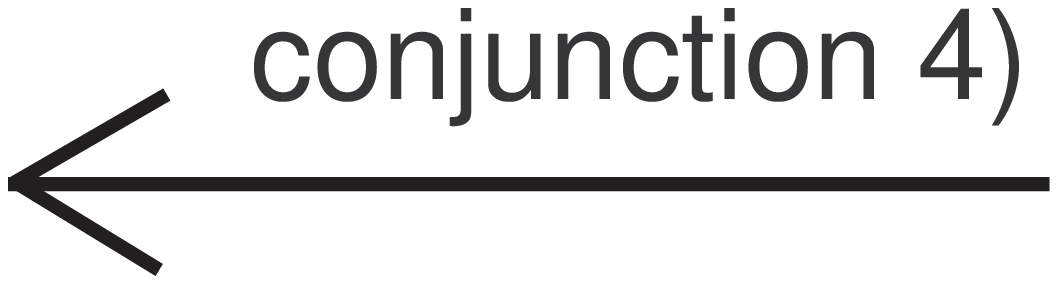}
\includegraphics[width=4cm]{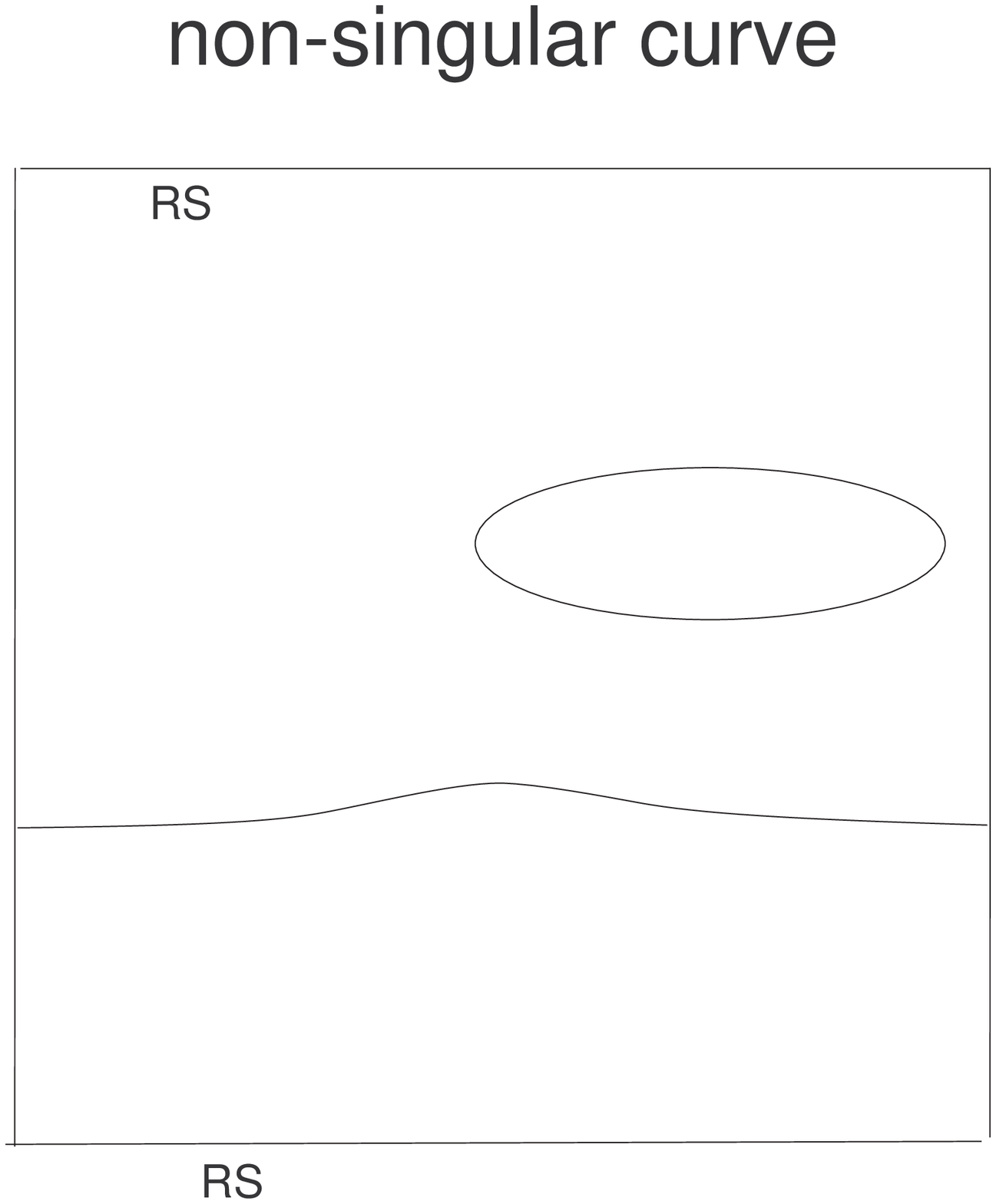}
\end{center}
\caption{Non-increasing simplest degenerations of nonsingular real anti-bicanonical curves on $\br \bff_4$.}
\label{Degenerations}
\end{figure}     

\clearpage


\setlength{\topmargin}{-0.6in}
\setlength{\textheight}{10in}

We list up candidates for possible non-increasing simplest degenerations 
for each of $63$ isometry classes of integral involutions of type $(S,\theta) \cong (\bu,- \id)$ 
with the invariant $(r(\psi),a(\psi),\delta_\psi)$. 
\begin{theorem} \label{degenerations-F4-re-arran}
We can enumerate up the following candidates: 
\begin{itemize}
\item For Conjunction 1)\ Conjunction 2)\ Contraction 3), see {\sc Table \ref{nonsing_F_4_a}}.
\item For Conjunction 1')\ Conjunction 2')\ Contraction 3'), see {\sc Table \ref{nonsing_F_4_b}}.
\item For Conjunction 4) and 4'), see {\sc Table \ref{nonsing_F_4_conj5and5'}}. $\Box$
\end{itemize}
\end{theorem}
\begin{table}[t]
\begin{center}
{\tiny 
\begin{tabular}{|r|r|c||r|r|r||c|c|c||r|}
\cline{1-9}
\multicolumn{3}{|l||}{isometry class} & \multicolumn{3}{c||}{} & conjunction 1)  & conjunction 2)  & contraction 3)  \\
\multicolumn{3}{|c||}{of type $(\bu,- \id)$} & \multicolumn{3}{c||}{} &               &                 &                 \\ \cline{7-9}
\multicolumn{3}{|c||}{(nonsingular curve)} & \multicolumn{3}{c||}{} & \ \ $\to$ Node (1) & \ \ $\to$ Node (2) & \ \ $\to$ Isolated point \\ \hline
 $r(\psi)$ & $a(\psi)$ & $\delta_\psi$             & $g$ & $k$ & $g-1$      &  $\alpha,\beta$ & $\alpha,\beta$ & $\alpha,\beta$    &   No.              \\
\hline
 1 & 1 & 1 &10 & 0 & 9 & 0,\ 8 & 0,\ 7 & 0,\ 8 &1\\ \hline
 2 & 0 & 0 &10 & 1 & 9 & 1,\ 8 & 1,\ 7 & 1,\ 8 &2\\ \hline
 2 & 2 & 0 & 9 & 0 & 8 & 0,\ 7 & 0,\ 6 & 0,\ 7 &3\\ \hline
 2 & 2 & 1 & 9 & 0 & 8 & 0,\ 7 & 0,\ 6 & 0,\ 7 &4\\ \hline
 3 & 1 & 1 & 9 & 1 & 8 & 1,\ 7 & 1,\ 6 & 1,\ 7 &5\\ \hline
 3 & 3 & 1 & 8 & 0 & 7 & 0,\ 6 & 0,\ 5 & 0,\ 6 &6\\ \hline
 4 & 2 & 1 & 8 & 1 & 7 & 1,\ 6 & 1,\ 5 & 1,\ 6 &7\\ \hline
 4 & 4 & 1 & 7 & 0 & 6 & 0,\ 5 & 0,\ 4 & 0,\ 5 &8\\ \hline
 5 & 3 & 1 & 7 & 1 & 6 & 1,\ 5 & 1,\ 4 & 1,\ 5 &9\\ \hline
 5 & 5 & 1 & 6 & 0 & 5 & 0,\ 4 & 0,\ 3 & 0,\ 4 &10\\ \hline
 6 & 2 & 0 & 7 & 2 & 6 & 2,\ 5 & 2,\ 4 & 2,\ 5 &11\\ \hline
 6 & 4 & 0 & 6 & 1 & 5 & 1,\ 4 & 1,\ 3 & 1,\ 4 &12\\ \hline
 6 & 4 & 1 & 6 & 1 & 5 & 1,\ 4 & 1,\ 3 & 1,\ 4 &13\\ \hline
 6 & 6 & 1 & 5 & 0 & 4 & 0,\ 3 & 0,\ 2 & 0,\ 3 &14\\ \hline
 7 & 3 & 1 & 6 & 2 & 5 & 2,\ 4 & 2,\ 3 & 2,\ 4 &15\\ \hline
 7 & 5 & 1 & 5 & 1 & 4 & 1,\ 3 & 1,\ 2 & 1,\ 3 &16\\ \hline
 7 & 7 & 1 & 4 & 0 & 3 & 0,\ 2 & 0,\ 1 & 0,\ 2 &17\\ \hline
 8 & 2 & 1 & 6 & 3 & 5 & 3,\ 4 & 3,\ 3 & 3,\ 4 &18\\ \hline
 8 & 4 & 1 & 5 & 2 & 4 & 2,\ 3 & 2,\ 2 & 2,\ 3 &19\\ \hline
 8 & 6 & 1 & 4 & 1 & 3 & 1,\ 2 & 1,\ 1 & 1,\ 2 &20\\ \hline
 8 & 8 & 1 & 3 & 0 & 2 & 0,\ 1 & 0,\ 0 & 0,\ 1 &21\\ \hline
 9 & 1 & 1 & 6 & 4 & 5 & 4,\ 4 & 4,\ 3 & 4,\ 4 &22\\ \hline
 9 & 3 & 1 & 5 & 3 & 4 & 3,\ 3 & 3,\ 2 & 3,\ 3 &23\\ \hline
 9 & 5 & 1 & 4 & 2 & 3 & 2,\ 2 & 2,\ 1 & 2,\ 2 &24\\ \hline
 9 & 7 & 1 & 3 & 1 & 2 & 1,\ 1 & 1,\ 0 & 1,\ 1 &25\\ \hline
 9 & 9 & 1 & 2 & 0 & 1 & 0,\ 0 & impossible & 0,\ 0 &26\\ \hline
10 & 0 & 0 & 6 & 5 & 5 & 5,\ 4 & 5,\ 3      & 5,\ 4 &27\\ \hline
10 & 2 & 0 & 5 & 4 & 4 & 4,\ 3 & 4,\ 2      & 4,\ 3 &28\\ \hline
10 & 2 & 1 & 5 & 4 & 4 & 4,\ 3 & 4,\ 2      & 4,\ 3 &29\\ \hline
10 & 4 & 0 & 4 & 3 & 3 & 3,\ 2 & 3,\ 1      & 3,\ 2 &30\\ \hline
10 & 4 & 1 & 4 & 3 & 3 & 3,\ 2 & 3,\ 1      & 3,\ 2 &31\\ \hline
10 & 6 & 0 & 3 & 2 & 2 & 2,\ 1 & 2,\ 0      & 2,\ 1 &32\\ \hline
10 & 6 & 1 & 3 & 2 & 2 & 2,\ 1 & 2,\ 0      & 2,\ 1 &33\\ \hline
10 & 8 & 1 & 2 & 1 & 1 & 1,\ 0 & impossible & 1,\ 0 &34\\ \hline
11 & 1 & 1 & 5 & 5 & 4 & 5,\ 3 & 5,\ 2      & 5,\ 3 &35\\ \hline
11 & 3 & 1 & 4 & 4 & 3 & 4,\ 2 & 4,\ 1      & 4,\ 2 &36\\ \hline
11 & 5 & 1 & 3 & 3 & 2 & 3,\ 1 & 3,\ 0      & 3,\ 1 &37\\ \hline
11 & 7 & 1 & 2 & 2 & 1 & 2,\ 0 & impossible & 2,\ 0 &38\\ \hline
12 & 2 & 1 & 4 & 5 & 3 & 5,\ 2 & 5,\ 1      & 5,\ 2 &39\\ \hline
12 & 4 & 1 & 3 & 4 & 2 & 4,\ 1 & 4,\ 0      & 4,\ 1 &40\\ \hline
12 & 6 & 1 & 2 & 3 & 1 & 3,\ 0 & impossible & 3,\ 0 &41\\ \hline
13 & 3 & 1 & 3 & 5 & 2 & 5,\ 1 & 5,\ 0      & 5,\ 1 &42\\ \hline
13 & 5 & 1 & 2 & 4 & 1 & 4,\ 0 & impossible & 4,\ 0 &43\\ \hline
14 & 2 & 0 & 3 & 6 & 2 & 6,\ 1 & 6,\ 0      & 6,\ 1 &44\\ \hline
14 & 4 & 0 & 2 & 5 & 1 & 5,\ 0 & impossible & 5,\ 0 &45\\ \hline
14 & 4 & 1 & 2 & 5 & 1 & 5,\ 0 & impossible & 5,\ 0 &46\\ \hline
15 & 3 & 1 & 2 & 6 & 1 & 6,\ 0 & impossible & 6,\ 0 &47\\ \hline
16 & 2 & 1 & 2 & 7 & 1 & 7,\ 0 & impossible & 7,\ 0 &48\\ \hline
17 & 1 & 1 & 2 & 8 & 1 & 8,\ 0 & impossible & 8,\ 0 &49\\ \hline
18 & 0 & 0 & 2 & 9 & 1 & 9,\ 0 & impossible & 9,\ 0 &50\\ \hline
\end{tabular}
} 
\end{center}
\caption{Conjunction 1), Conjunction 2) and Contraction 3)}
\label{nonsing_F_4_a}
\end{table}
\begin{table}[h]
\begin{center}
{\tiny 
\begin{tabular}{|r|r|c||r|r|r||c|c|c||r|}
\cline{1-9}
\multicolumn{3}{|l||}{isometry class} & \multicolumn{3}{c||}{} & conjunction 1')  & conjunction 2')  & contraction 3')  \\
\multicolumn{3}{|c||}{of type $(\bu,- \id)$} & \multicolumn{3}{c||}{} &               &                 &                 \\ \cline{7-9}
\multicolumn{3}{|c||}{(nonsingular curve)} & \multicolumn{3}{c||}{} & \ \ $\to$ Node (1) & \ \ $\to$ Node (2) & \ \ $\to$ Isolated point \\ \hline
 $r(\psi)$ & $a(\psi)$ & $\delta_\psi$             & $g$ & $k$ & $g-1$      &   $\alpha,\beta$ & $\alpha,\beta$ & $\alpha,\beta$   &   No.                   \\
\hline
19 & 1 & 1 & 1 & 9 & 0 & 0,\ 8 & 0,\ 7 & 0,\ 8 &1'\\ \hline
18 & 0 & 0 & 2 & 9 & 1 & 1,\ 8 & 1,\ 7 & 1,\ 8 &2'\\ \hline
18 & 2 & 0 & 1 & 8 & 0 & 0,\ 7 & 0,\ 6 & 0,\ 7 &3'\\ \hline
18 & 2 & 1 & 1 & 8 & 0 & 0,\ 7 & 0,\ 6 & 0,\ 7 &4'\\ \hline
17 & 1 & 1 & 2 & 8 & 1 & 1,\ 7 & 1,\ 6 & 1,\ 7 &5'\\ \hline
17 & 3 & 1 & 1 & 7 & 0 & 0,\ 6 & 0,\ 5 & 0,\ 6 &6'\\ \hline
16 & 2 & 1 & 2 & 7 & 1 & 1,\ 6 & 1,\ 5 & 1,\ 6 &7'\\ \hline
16 & 4 & 1 & 1 & 6 & 0 & 0,\ 5 & 0,\ 4 & 0,\ 5 &8'\\ \hline
15 & 3 & 1 & 2 & 6 & 1 & 1,\ 5 & 1,\ 4 & 1,\ 5 &9'\\ \hline
15 & 5 & 1 & 1 & 5 & 0 & 0,\ 4 & 0,\ 3 & 0,\ 4 &10'\\ \hline
14 & 2 & 0 & 3 & 6 & 2 & 2,\ 5 & 2,\ 4 & 2,\ 5 &11'\\ \hline
14 & 4 & 0 & 2 & 5 & 1 & 1,\ 4 & 1,\ 3 & 1,\ 4 &12'\\ \hline
14 & 4 & 1 & 2 & 5 & 1 & 1,\ 4 & 1,\ 3 & 1,\ 4 &13'\\ \hline
14 & 6 & 1 & 1 & 4 & 0 & 0,\ 3 & 0,\ 2 & 0,\ 3 &14'\\ \hline
13 & 3 & 1 & 3 & 5 & 2 & 2,\ 4 & 2,\ 3 & 2,\ 4 &15'\\ \hline
13 & 5 & 1 & 2 & 4 & 1 & 1,\ 3 & 1,\ 2 & 1,\ 2 &16'\\ \hline
13 & 7 & 1 & 1 & 3 & 0 & 0,\ 2 & 0,\ 1 & 0,\ 2 &17'\\ \hline
12 & 2 & 1 & 4 & 5 & 3 & 3,\ 4 & 3,\ 3 & 3,\ 4 &18'\\ \hline
12 & 4 & 1 & 3 & 4 & 2 & 2,\ 3 & 2,\ 2 & 2,\ 3 &19'\\ \hline
12 & 6 & 1 & 2 & 3 & 1 & 1,\ 2 & 1,\ 1 & 1,\ 2 &20'\\ \hline
12 & 8 & 1 & 1 & 2 & 0 & 0,\ 1 & 0,\ 0 & 0,\ 1 &21'\\ \hline
11 & 1 & 1 & 5 & 5 & 4 & 4,\ 4 & 4,\ 3 & 4,\ 4 &22'\\ \hline
11 & 3 & 1 & 4 & 4 & 3 & 3,\ 3 & 3,\ 2 & 3,\ 3 &23'\\ \hline
11 & 5 & 1 & 3 & 3 & 2 & 2,\ 2 & 2,\ 1 & 2,\ 2 &24'\\ \hline
11 & 7 & 1 & 2 & 2 & 1 & 1,\ 1 & 1,\ 0 & 1,\ 1 &25'\\ \hline
11 & 9 & 1 & 1 & 1 & 0 & 0,\ 0 & impossible & 0,\ 0 &26'\\ \hline
10 & 0 & 0 & 6 & 5 & 5 & 5,\ 4 & 5,\ 3      & 5,\ 4 &27'\\ \hline
10 & 2 & 0 & 5 & 4 & 4 & 4,\ 3 & 4,\ 2      & 4,\ 3 &28'\\ \hline
10 & 2 & 1 & 5 & 4 & 4 & 4,\ 3 & 4,\ 2      & 4,\ 3 &29'\\ \hline
10 & 4 & 0 & 4 & 3 & 3 & 3,\ 2 & 3,\ 1      & 3,\ 2 &30'\\ \hline
10 & 4 & 1 & 4 & 3 & 3 & 3,\ 2 & 3,\ 1      & 3,\ 2 &31'\\ \hline
10 & 6 & 0 & 3 & 2 & 2 & 2,\ 1 & 2,\ 0      & 2,\ 1 &32'\\ \hline
10 & 6 & 1 & 3 & 2 & 2 & 2,\ 1 & 2,\ 0      & 2,\ 1 &33'\\ \hline
10 & 8 & 1 & 2 & 1 & 1 & 1,\ 0 & impossible & 1,\ 0 &34'\\ \hline
 9 & 1 & 1 & 6 & 4 & 5 & 5,\ 3 & 5,\ 2      & 5,\ 3 &35'\\ \hline
 9 & 3 & 1 & 5 & 3 & 4 & 4,\ 2 & 4,\ 1      & 4,\ 2 &36'\\ \hline
 9 & 5 & 1 & 4 & 2 & 3 & 3,\ 1 & 3,\ 0      & 3,\ 1 &37'\\ \hline
 9 & 7 & 1 & 3 & 1 & 2 & 2,\ 0 & impossible & 2,\ 0 &38'\\ \hline
 8 & 2 & 1 & 6 & 3 & 5 & 5,\ 2 & 5,\ 1      & 5,\ 2 &39'\\ \hline
 8 & 4 & 1 & 5 & 2 & 4 & 4,\ 1 & 4,\ 0      & 4,\ 1 &40'\\ \hline
 8 & 6 & 1 & 4 & 1 & 3 & 3,\ 0 & impossible & 3,\ 0 &41'\\ \hline
 7 & 3 & 1 & 6 & 2 & 5 & 5,\ 1 & 5,\ 0      & 5,\ 1 &42'\\ \hline
 7 & 5 & 1 & 5 & 1 & 4 & 4,\ 0 & impossible & 4,\ 0 &43'\\ \hline
 6 & 2 & 0 & 7 & 2 & 6 & 6,\ 1 & 6,\ 0      & 6,\ 1 &44'\\ \hline
 6 & 4 & 0 & 6 & 1 & 5 & 5,\ 0 & impossible & 5,\ 0 &45'\\ \hline
 6 & 4 & 1 & 6 & 1 & 5 & 5,\ 0 & impossible & 5,\ 0 &46'\\ \hline
 5 & 3 & 1 & 7 & 1 & 6 & 6,\ 0 & impossible & 6,\ 0 &47'\\ \hline
 4 & 2 & 1 & 8 & 1 & 7 & 7,\ 0 & impossible & 7,\ 0 &48'\\ \hline
 3 & 1 & 1 & 9 & 1 & 8 & 8,\ 0 & impossible & 8,\ 0 &49'\\ \hline
 2 & 0 & 0 &10 & 1 & 9 & 9,\ 0 & impossible & 9,\ 0 &50'\\ \hline
\end{tabular}
} 
\end{center}
\caption{Conjunction 1'), Conjunction 2') and Contraction 3')}
\label{nonsing_F_4_b}
\end{table}

\begin{table}[!h]
{\tiny 
\begin{tabular}{|r|r|c||r|r|r||c|}
\hline
\multicolumn{3}{|l||}{isometry class} & \multicolumn{3}{c||}{} &   \\
\multicolumn{3}{|c||}{of type $(\bu,- \id)$} & \multicolumn{3}{c||}{} & conjunction 4) \\
\multicolumn{3}{|c||}{(nonsingular curve)} & \multicolumn{3}{c||}{} & \\
\cline{1-6}
 $r(\psi)$ & $a(\psi)$ & $\delta_\psi$ & $g$ & $k$ & $g-1$ & \multicolumn{1}{c|}{}   \\ \hline
 9 & 9 & 1 & 2 & 0 & 1 & Node (*) \\ \hline
\end{tabular}
\ \ \ \ 
\begin{tabular}{|r|r|c||r|r|r||c|}
\hline
\multicolumn{3}{|l||}{isometry class} & \multicolumn{3}{c||}{} &   \\
\multicolumn{3}{|c||}{of type $(\bu,- \id)$} & \multicolumn{3}{c||}{} & conjunction 4') \\
\multicolumn{3}{|c||}{(nonsingular curve)} & \multicolumn{3}{c||}{} & \\
\cline{1-6}
 $r(\psi)$ & $a(\psi)$ & $\delta_\psi$ & $g$ & $k$ & $g-1$ & \multicolumn{1}{c|}{}   \\ \hline
11 & 9 & 1 & 1 & 1 & 0 & Node (*) \\ \hline  
\end{tabular}
} 
\caption{Conjunctions 4) and 4')}
\label{nonsing_F_4_conj5and5'}
\end{table}

\clearpage

\setlength{\topmargin}{-0.4in}
\setlength{\textheight}{9.5in}

\begin{remark} \label{interesting-correspondence}
We observe an obvious interesting correspondence between 
the data of Theorem \ref{isotopy-F4-double} ({\sc Tables} \ref{3-1-1-delta_F_1} --- \ref{3-1-1-delta_F_0}) 
and 
those of Theorem \ref{degenerations-F4-re-arran} ({\sc Tables} \ref{nonsing_F_4_a} --- \ref{nonsing_F_4_conj5and5'}). 

If we prove the existence of all the non-increasing simplest degenerations listed in 
{\sc Tables} \ref{nonsing_F_4_a}, \ref{nonsing_F_4_b}, and \ref{nonsing_F_4_conj5and5'}, 
then 
we can say that 
every real isotopy type with an anti-holomorphic involution in Theorem \ref{isotopy-F4-double} 
except 
Node (1) with $(\alpha,\beta)=(1,0)$ and Isolated point with $(\alpha,\beta)=(1,0)$ 
in the isometry class $(9,9,0,H(\psi) \cong \bz/2\bz)$ 
can be obtain 
by certain non-increasing simplest degeneration of a real nonsingular curve in $|12c+3s|$ on $\bff_4$.
Recall Remark \ref{realizability1}. 
\end{remark}

To state the following lemma \ref{degene-from-nonsing}, 
we distinguish the two anti-holomorphic involutions $\varphi_{\pm}$ on $(X,\tau)$ of type $(S,\theta) \cong ((3,1,1),- \id)$ as follows. 
%
\begin{definition}[Definition of $\varphi_{\pm}$] \label{varphi+-}
We define the anti-holomorphic involutions $\varphi_{\pm}$ on a $2$-elementary K3 surface $(X,\tau)$ of type $(S,\theta) \cong ((3,1,1),- \id)$ 
such that 
$$\pi(X_{\varphi_{\pm}}(\br)) = A_\pm$$
respectively. Here recall Definition \ref{A_+A_-} of the two regions $A_\pm$. 
\end{definition}

Thus, if $H(\psi) = 0$ for a real $2$-elementary K3 surface $(X,\tau,\varphi)$ of type $(S,\theta) \cong ((3,1,1),- \id)$, 
then we set $A_- = \pi(X_\varphi(\br))$, and 
$$\varphi_- := \varphi \ \ \mbox{and} \ \ \varphi_+ := \wvarphi = \tau \circ \varphi.$$

Comparing the data in Theorem \ref{isotopy-F4-double} 
and Theorem \ref{degenerations-F4-re-arran}, 
we have:
\begin{lemma} \label{degene-from-nonsing}
Fix an {\bf isometry class} of integral involutions of $\bl_{K3}$ of type $(S,\theta) \cong (\bu,- \id)$ 
with $(r(\psi),a(\psi),$ $\delta_\psi)$ $\not=$ $(10,8,0),$ $(10,10,0)$, 
and 
take a corresponding real $2$-elementary K3 surface of type $(S,\theta) \cong (\bu,- \id)$ 
and 
a real {\bf nonsingular} curve $A = s + A_1$ in $|-2K_{\bff_4}|$ where 
$A_1$ is a real nonsingular curve in $|12c+3s|$ on $\br \bff_4$. 

Then we have the following:
\begin{itemize}
\item 
Take real curves $A^\prime_1$ with {\bf one nondegenerate double point} on $\bff_4$ 
from the degenerations of types 1)---3) of the real {\bf nonsingular} curve $A_1$. 
{\bf Choose} real $2$-elementary K3 surfaces 
$(X,\tau,\varphi_-)$ of type $(S,\theta) \cong ((3,1,1),- \id)$.
(See Definition \ref{varphi+-} and also Remark \ref{choose-F-4} below.)

Then, for all such marked real $2$-elementary K3 surfaces 
$((X,\tau,\varphi_-),\ \alpha)$ obtained from $A^\prime_1$, 
their associated integral involutions 
$\alpha \circ (\varphi_-)_* \circ \alpha^{-1}$ of $\bl_{K3}$ 
are {\bf isometric} with respect to $G = \{ \id \}$.

\item 
Take real curves $A^\prime_1$ with {\bf one nondegenerate double point} on $\bff_4$ 
from the degenerations of types 1')---3') of the real {\bf nonsingular} curve $A_1$. 
{\bf Choose} real $2$-elementary K3 surfaces 
$(X,\tau,\varphi_-)$ of type $(S,\theta) \cong ((3,1,1),- \id)$.
(See Definition \ref{varphi+-} and also Remark \ref{choose-F-4} below.)

Then, for all such marked real $2$-elementary K3 surfaces 
$((X,\tau,\varphi_-),\ \alpha)$ obtained from $A^\prime_1$, 
their associated integral involutions 
$\alpha \circ (\varphi_-)_* \circ \alpha^{-1}$ of $\bl_{K3}$ 
are {\bf isometric} 
with respect to $G = \{ \id \}$. $\Box$
\end{itemize}
\end{lemma}

Lemma \ref{degene-from-nonsing} is an analogy of Proposition 3.3 of \cite{Itenberg92}.

\bigskip

\section{Appendix:\ Period domains and further problems} \label{period domain and problems}

We use the terminology defined in Subsection \ref{real_2-elementary K3}, 
and 
review the formulations of period domains of 
marked real $2$-elementary K3 surfaces in \cite{NikulinSaito05}, \cite{Itenberg92}. 


\begin{lemma}\label{same-associated-invol}
Let $((X,\tau,\varphi),\ \alpha)$ and $((X^\prime,\tau^\prime,\varphi^\prime),\ \alpha^\prime)$ be 
two marked real $2$-elementary K3 surfaces of type $(S,\theta)$. 
Let $(\bl_{K3},\psi)$ and $(\bl_{K3},\psi^\prime)$ be their associated integral involutions respectively. 
Suppose that $f$ is an {\bf isometry} with respect to the group $G$ 
from $(\bl_{K3},\psi)$ to $(\bl_{K3},\psi^\prime)$. 
Then $((X^\prime,\tau^\prime,\varphi^\prime),\ f^{-1} \circ \alpha^\prime)$ 
is also a marked real $2$-elementary K3 surface of type $(S,\theta)$, 
and its associated integral involution is also $(\bl_{K3},\psi)$. 

Especially, if $f$ is an {\bf automorphism} of $(\bl_{K3},\psi)$ with respect to the group $G$, 
then 
$((X,\tau,\varphi),\ f \circ \alpha)$ 
is also a marked real $2$-elementary K3 surface of type $(S,\theta)$, 
and its associated integral involution is also $(\bl_{K3},\psi)$. 
(Recall Definition \ref{auto-psi}.) 
\end{lemma}
\begin{proof}
We have $\psi^\prime \circ f = f \circ \psi$, $f(S)=S$, and ${f|}_S \in G$. 
We have 
$\alpha^\prime \circ {\varphi^\prime}_* \circ {\alpha^\prime}^{-1} \circ f = f \circ \alpha \circ \varphi_* \circ \alpha^{-1}$, 
and hence, 
$f^{-1} \circ \alpha^\prime \circ {\varphi^\prime}_* \circ {\alpha^\prime}^{-1} \circ f = \alpha \circ \varphi_* \circ \alpha^{-1}$. 
We have 
$(f^{-1} \circ \alpha^\prime) \circ {\varphi^\prime}_* \circ (f^{-1} \circ \alpha^\prime)^{-1}
 = \alpha \circ \varphi_* \circ \alpha^{-1} = \psi$. 
Here 
$(f^{-1} \circ \alpha^\prime) : H_2(X^\prime, \bz) \to \bl_{K3}$ is an isometry with 
$(f^{-1} \circ \alpha^\prime) ({H_2}_+(X^\prime, \bz)) = f^{-1}(S) = S$ and 
$
(f^{-1} \circ \alpha^\prime) \circ {\varphi^\prime}_* 
= f^{-1} \circ (\alpha^\prime \circ {\varphi^\prime}_*) 
= f^{-1} \circ (\psi^\prime \circ \alpha^\prime) 
= \psi \circ f^{-1} \circ \alpha^\prime
$.
Since $f^{-1} \circ \alpha^\prime ({H_2}_+(X^\prime, \bz)) =S$ and ${\psi |}_S = \theta$, 
we have 
$$
(f^{-1} \circ \alpha^\prime) \circ {\varphi^\prime}_* 
= \theta \circ f^{-1} \circ \alpha^\prime 
\text{on} \ {H_2}_+(X^\prime, \bz).
$$
Hence, 
$$(f^{-1} \circ \alpha^\prime) : H_2(X^\prime, \bz) \to \bl_{K3}$$
is another marking of $(X^\prime,\tau^\prime,\varphi^\prime)$. 
If we take this new marking of $(X^\prime,\tau^\prime,\varphi^\prime)$, then 
its associated integral involution is $\psi$, 
which is the same as $((X,\tau,\varphi),\ \alpha)$. 
Moreover, 
since ${f|}_S \in G$, we have $f(\M)=\M$. 
Hence, if we set $\beta := (f^{-1} \circ \alpha^\prime)$, then 
$\beta_{\br}^{-1}(V^+(S))$ contains a hyperplane section of $X^\prime$ 
and the set $\beta^{-1}(\Delta(S)_+)$ contains only classes of effective curves of $X^\prime$. 
\end{proof}

\bigskip

Let us {\bf fix} an integral involution 
$$(\bl_{K3},\psi)$$
of type $(S,\theta)$ through this subsection. 

\medskip

We set 
$$
\Omega_\psi := 
\{ v \ (\neq 0) \in \bl_{K3}\otimes \bc \ |\ 
v \cdot v =0, \ v \cdot \overline{v} >0, \ v \cdot S =0, \ 
\psi_{\bc}(v)=\overline{v} \}/\br^{\times} .
$$

\medskip

Let 
$$((X,\tau,\varphi),\ \alpha)$$
be a marked real $2$-elementary K3 surface of type $(S,\theta)$ 
satisfying 
$$\alpha \circ \varphi_* \circ \alpha^{-1} = \psi ,$$
namely,
$\psi$ is {\bf the associated integral involution with $((X,\tau,\varphi,)\ \alpha)$}.
We denote by 
$$H \ \ (\subset H_2(X,\bc))$$
the Poincare dual of the complex $1$-dimensional space $H^{2,0}(X)$. 
Then we have the complex $1$-dimensional subspace 
$\alpha_{\bc}(H)$ in $\bl_{K3}\otimes \bc$. 
Then we have 
$$\alpha_{\bc}(H) \ \in \Omega_\psi .$$

\begin{definition}[Periods]
We say $\alpha_{\bc}(H)$ 
the {\bf period} of 
a marked real $2$-elementary K3 surface $((X,\tau,\varphi),\ \alpha)$ of type $(S,\theta)$ 
satisfying $\alpha \circ \varphi_* \circ \alpha^{-1} = \psi$.
\end{definition}

\medskip

By Lemma \ref{same-associated-invol}, 
all marked real $2$-elementary K3 surfaces 
whose associated integral involutions are isometric to $(\bl_{K3},\psi)$ with respect to $G$ 
can be found in $\Omega_\psi$ 
if we change their markings appropriately.

A point in $\Omega_\psi$ 
is not necessarily the period of 
some marked real $2$-elementary K3 surface of type $(S,\theta)$ satisfying $\alpha \circ \varphi_* \circ \alpha^{-1} = \psi$. 

\begin{definition}[Equivalence]
We say a point $[v] \ (\in \Omega_\psi)$ is {\bf equivalent} to 
a point $[v^\prime] \ (\in \Omega_\psi)$ 
if $[v^\prime] = f_{\bc}([v])$ for 
an automorphism $f$ of $(\bl_{K3},\psi)$ of type $(S,\theta)$ with respect to the group $G$. 
\end{definition}

\begin{lemma}\label{equivalence-lemma}
If a point $[v] \ (\in \Omega_\psi)$ is equivalent to $[v^\prime] \ (\in \Omega_\psi)$ and 
$[v]$ is the period of 
some marked real $2$-elementary K3 surface $((X,\tau,\varphi),\ \alpha)$ of type $(S,\theta)$ 
satisfying $\alpha \circ \varphi_* \circ \alpha^{-1} = \psi$, 
then 
$[v^\prime]$ is also the period of 
a marked real $2$-elementary K3 surface $((X,\tau,\varphi),\ \alpha^\prime)$ of type $(S,\theta)$
satisfying 
$(\alpha^\prime) \circ \varphi_* \circ (\alpha^\prime)^{-1} = \psi$ 
where $\alpha^\prime$ is {\bf some another marking} of $(X,\tau,\varphi)$. 
\end{lemma}

\begin{proof}
Since $[v]$ is equivalent to $[v^\prime]$, 
we have $[v^\prime] = f_{\bc}([v])$ for 
an automorphism $f$ of $(\bl_{K3},\psi)$ of type $(S,\theta)$ with respect to the group $G$. 
By Lemma \ref{same-associated-invol}, 
$((X,\tau,\varphi),\ f \circ \alpha)$
is also a marked real $2$-elementary K3 surface of type $(S,\theta)$
satisfying 
$(f \circ \alpha) \circ \varphi_* \circ (f \circ \alpha)^{-1} = \psi$. 
Moreover, the period of $((X,\tau,\varphi),\ f \circ \alpha)$ is
$$(f \circ \alpha)_{\bc}(H) 
= f_{\bc}(\alpha_{\bc}(H)) 
= f_{\bc}([v])
= [v^\prime]$$
where $H \ \ (\subset H_2(X,\bc))$ is the Poincare dual of $H^{2,0}(X)$. 
It is sufficient that we set $\alpha^\prime = f \circ \alpha)$. 
\end{proof}

\begin{remark}[\cite{NikulinSaito05}] \label{equivalence-remark}
By the global Torelli theorem, 
if two periods are equivalent, 
then corresponding marked real $2$-elementary K3 surfaces are analytic isomorphic (see Definition \ref{analytic-iso}). 
The converse is also true. 
\end{remark}

The domain $\Omega_\psi$ has two connected components which are interchanged by $-\psi$. 
We see $-\psi$ is an {\bf automorphism} of $(\bl_{K3},\psi)$ with respect to the group $G$. 
Hence, by Lemma \ref{equivalence-lemma} and Remark \ref{equivalence-remark}, 
it is sufficient that we investigate the quotient space 
$$\Omega_\psi /-\psi .$$
We set 
$$\bl_{\pm} := \{ x \in \bl_{K3}\ |\ \psi (x) = \pm x \}.$$
Note that the lattices $\bl_{\pm}$ depend on the integral involution $\psi$. 

\bigskip

For $[v] \in \Omega_\psi$ \ ($v \in \bl_{K3} \otimes \bc$), 
we have the decomposition 
$$v = v_+ + v_-,$$
where $v_{\pm} \in \bl_{\pm} \otimes \br$. 

\bigskip

We restrict ourselves the case when 
$S \subset \bl_-$, namely, $\theta = - \id$, and set 
$$\bl_{-,S} := \bl_- \cap S^{\perp}.$$
$\bl_{-,S}$ also depends on the integral involution $\psi$.

Since $v_- \in \bl_{-,S}\otimes \br$ and $v_+^2 = v_-^2 > 0$, we see that 
$\bl_+,\ \bl_{-,S}$ is a hyperbolic lattice.

\bigskip

Let
$$\La_+, \ \ \La_{-,S}$$
be the hyperbolic spaces obtained from 
$\bl_+ \otimes \br,\ \bl_{-,S} \otimes \br$ respectively. 
Then we have 
$$\Omega_\psi /-\psi \ = \ \La_+ \times \La_{-,S}\ \ \ \ (\mbox{a direct product}).$$

\bigskip


We now fix a primitive hyperbolic $2$-elementary sublattice $S$ with 
$$S \cong (3,1,1)$$
of the K3 lattice $\bl_{K3}$ and 
set $\theta := - \id$. 
Remark that $G = \{ \id \}$ for this case (Remark \ref{G=1}). 

We use the terminology and facts stated in Subsection \ref{RealK3-311}. 
For a real $2$-elementary K3 surface $(X,\tau,\varphi)$ of type $(S,\theta) \cong ((3,1,1),- \id)$, 
the fixed point curve $A$ is a disjoint union of $A_0$ and $A_1$. 
$X$ has an elliptic fibration with the exceptional section $A_0$. 
Its unique reducible fiber is $E+F$. 
There are two kinds of the curve $F$, which is irreducible or not. 
The classes $[A_0]$,\ $[E]$ and $[F]$ generate ${H_2}_+(X, \bz)$. 
We have an orthogonal decomposition
$$
\bz ([A_0],\ [E]+[F]) \oplus \bz ([F])
$$
of ${H_2}_+(X, \bz)$. 
The subgroups $\bz ([A_0],\ [E]+[F]),\ \bz ([F])$ are isometric to 
the hyperbolic plane and the lattice $\langle -2 \rangle$ respectively. 

\medskip

We fix an integral involution $$(\bl_{K3},\psi)$$ of type $(S,\theta) \cong ((3,1,1),- \id)$. 
Moreover, we fix a decomposition 
$$
S = \bu \oplus \langle -2 \rangle ,
$$
where $\bu$ is a sublattice of $S$ isometric to the hyperbolic plane. 
Let 
$$\Fa$$
be the generator of $\langle -2 \rangle$ with $\M \cdot \Fa \ge 0$. 

\medskip

We consider marked real $2$-elementary K3 surfaces 
$((X,\tau,\varphi), \alpha)$ 
of type $(S,\theta) \cong ((3,1,1),- \id)$ (see Definition \ref{marked_real_K3}) 
such that 
$$
\alpha(\bz ([A_0],\ [E]+[F]))=\bu ,\ \ \ \alpha([F])=\Fa ,
$$
and 
$$\alpha \circ \varphi_* \circ \alpha^{-1} = \psi .$$
Any real $2$-elementary K3 surface $(X,\tau,\varphi)$ of type $(S,\theta) \cong ((3,1,1),- \id)$ 
has such a marking $\alpha$. 

We want to know a criterion for 
the double point of a real anti-bicanonical curve $\br \BL(A)$ 
with one real double point on $\br A^\prime_1$ on $\br \bff_4$ 
to be {\bf nondegenerate}. 
At present we can prove the following lemma. 

\begin{lemma}\label{criterion}
For $[\omega] \in \Omega_\psi /-\psi$, 
$[\omega]$ is the period of 
a marked real $2$-elementary K3 surface of type $(S,\theta) \cong ((3,1,1),- \id)$ (see Definition \ref{marked_real_K3}) 
such that 
$
\alpha(\bz ([A_0],\ [E]+[F]))=\bu ,\ \ \ \alpha([F])=\Fa ,
$
and 
$\alpha \circ \varphi_* \circ \alpha^{-1} = \psi$ 
obtained from 
a real anti-bicanonical curve $\br \BL(A)$ 
with one real {\bf nondegenerate} double point on $\br A^\prime_1$ on $\br \bff_4$ 
{\bf if} 
there are no $\bv \ (\neq \pm \Fa)$ in $\bl_{K3}$ satisfying that 
$\bv \cdot \omega = 0,\ \bv \cdot \bu =0$, and $\bv^2 = -2$. 
\end{lemma}

\begin{proof}
For a real anti-bicanonical curve with one real degenerate double point on $\br \bff_4$, 
the corresponding marked real $2$-elementary K3 surface $(X,\tau,\varphi,\alpha)$ has 
a unique reducible fiber $E+F$ where 
$F^2=-2$, $F$ is the union of two nonsingular rational curves $F^\prime$ and $F^{\prime\prime}$. 
Here $F^\prime$ and $F^{\prime\prime}$ are conjugate by $\tau$ and $F^\prime \cdot F^{\prime\prime} = 1$. 
Hence, $(F^\prime)^2 = (F^{\prime\prime})^2 = -2$. (see Subsection \ref{RealK3-311}) 
Both $[F^\prime]$ and $[F^{\prime\prime}]$ are orthogonal to $[A_0]$ and $[E]+[F]$. 
We set $\bv := \alpha[F^\prime]$ or $\alpha[F^{\prime\prime}]$. 
Since $\bu = \alpha(\bz ([A_0],\ [E]+[F]))$, 
we have 
$\bv \ (\neq \pm \Fa)$, $\bv \cdot \omega = 0,\ \bv \cdot \bu =0$, and $\bv^2 = -2$ 
for the period $[\omega]$ of $(X,\tau,\varphi,\alpha)$. 
\end{proof}

\begin{problem}[cf. \cite{Itenberg92}, the top of p.281]
Is the converse of Lemma \ref{criterion} also true?
\end{problem}

If the converse of Lemma \ref{criterion} is true, then we can get 
the precise image ($\subset \La_+ \times \La_{-,S}$) of the period map 
on the set of all marked real $2$-elementary K3 surfaces of type $(S,\theta) \cong ((3,1,1),- \id)$ 
such that $\alpha(\bz ([A_0],\ [E]+[F]))=\bu$, $\alpha([F])=\Fa$, and $\alpha \circ \varphi_* \circ \alpha^{-1} = \psi$ 
obtained from real anti-bicanonical curves $\br \BL(A)$ with one real {\bf nondegenerate} double point on $\br A^\prime_1$ on $\br \bff_4$. 

Let $\bv$ be an element of $\bl_+$ with square $-2$. 
The reflection on $\La_+$ with respect to the real hyperplane $\bv^{\perp}$ is well-defined and 
it sends a point in $\Omega_\psi /-\psi$ to its equivalent point. 
Also, 
the reflection on $\La_{-,S}$ with respect to the real hyperplane $\bv^{\perp}$ 
where $\bv$ is an element of $\bl_{-,S}$ with square $-2$ is well-defined and 
it sends a point in $\Omega_\psi /-\psi$ to its equivalent point. 
Let 
$$\Omega_+ \ \ (\mbox{respectively,}\ \Omega_{-,S})$$
be the open fundamental domains with respect to the groups generated by 
the reflections with respect to the real hyperplanes $\bv^{\perp}$ satisfying 
$\bv^2=-2$ and $\bv \in \bl_+ \ \ (\mbox{respectively,}\ \bl_{-,S})$ and we consider the direct product 
$\Omega_+ \times \Omega_{-,S}$.

If the converse of Lemma \ref{criterion} is true, then 
the periods of marked real $2$-elementary K3 surfaces $(X,\tau,\varphi,\alpha)$ of type $((3,1,1),$ $- \id)$ obtained from 
some real anti-bicanonical curves with one real {\bf nondegenerate} double point on $\br \bff_4$ 
and satisfying $\alpha \circ \varphi_* \circ \alpha^{-1} = \psi$ 
are contained in $\Omega_+ \times \Omega_{-,S}$ up to equivalence. 
Moreover, we should remove more $\omega = \omega_+ + \omega_-$ orthogonal to 
some $\bv \in \bl_{K3}$ such that 
$\bv \ (\neq \pm \Fa)$,\ $\bv \cdot \bu =0$,\ $\bv^2 = -2$,\ 
$\bv \not\in \bl_+$ and $\bv \not\in \bl_{-,S}$. 
Hence, as the argument before Theorem 2.1 of Itenberg's paper \cite{Itenberg92}, 
we should remove 
some $(-6)$-orthogonal real hyperplanes (extra ``walls") from $\Omega_{-,S}$ or $\Omega_+$. 
This seems to be the reason why the connected components of the moduli space in the sense of \cite{NikulinSaito05} 
(equivalently, the isometry classes of integral involutions of the K3 lattice) 
cannot distinguish the topological types (node or isolated point) of the real double points (Remark \ref{realizability1}). 

\begin{problem}[cf. \cite{Itenberg92}, Theorem 2.1]
Formulate some period domain, say it were $\Omega^{\br \bff_4}_*$, 
whose connected components (up to equivalence) are in {\bf bijective correspondence} with 
the connected components of the moduli space (up to the action of the automorphism group of $\bff_4$ over $\br$) 
of real anti-bicanonical curves 
with one real nondegenerate double point on $\br \bff_4$ 
which yield marked real $2$-elementary K3 surfaces $((X,\tau,\varphi_-),\alpha)$ 
satisfying $\alpha \circ (\varphi_-)_* \circ \alpha^{-1} = \psi$. 
\end{problem}

\begin{remark}\label{choose-F-4}
For $(\bl_{K3},\psi)$, 
either $\bl_+$ or $\bl_{-,S}$ does not contain any element $\bv$ such that 
$\bv \equiv \Fa \ (\mbox{mod}\ 2\bl_{K3})$. 
For the anti-holomorphic involution ``$\varphi_-$" (recall Definition \ref{varphi+-}), 
$\bl_{-,S}$ contains an element $\bv$ such that $\bv \equiv \Fa \ (\mbox{mod}\ 2\bl_{K3})$. 
This phenomenon is similar to the argument in \cite{Itenberg92}. 
\end{remark}

\bigskip

We are also interested in the correspondences of 
the Coxeter graphs (see \cite{Itenberg92}, \cite{Itenberg94}) 
obtained from 
isometry classes of 
integral involutions of the K3 lattice $\bl_{K3}$ of type $(S,\theta) \cong ((3,1,1),- \id)$ 
and 
the non-increasing simplest degenerations of nonsingular curves. 
Problems concerning this topic are as follows. 

We fix 
an integral involution $(\bl_{K3},\psi)$ of type $(S,\theta) \cong ((3,1,1),- \id)$. 
Recall $\bl_{\pm} := \{ x \in \bl_{K3}\ |\ \psi (x) = \pm x \}$, 
$\bu$, 
and 
$\bl_{-,\bu} := \bl_- \cap \bu^{\perp}$. $\bl_{-,\bu}$ is also hyperbolic. 
Let 
$\La_{-,\bu}$ 
be the hyperbolic space obtained from $\bl_{-,\bu} \otimes \br$. 
The group generated by the reflections with respect to 
real hyperplanes $\bv^{\perp}$ such that $\bv \in \bl_{-,\bu}$ with $\bv^2=-2$ 
acts on the hyperbolic space $\La_{-,\bu}$. 
Let 
$\displaystyle \widetilde{\Omega_-}$ 
be one of its fundamental domains which has a face orthogonal to $F$. 
Let 
$C$ 
be the Coxeter graph of $\displaystyle \widetilde{\Omega_-}$.

\begin{definition}\label{K-def-F4}
\begin{itemize}
\item Let $C^\prime$ be the graph which is obtained by removing all thick or dotted edges from $C$. 
\item Consider the group of symmetries of $C^\prime$ obtained from 
some automorphism of $(\bl_{K3},\psi,\bu)$. 
Let $C''$ be the quotient graph of $C^\prime$ by the action of the group. 
\item Let $e$ be the vertex of $C$ corresponding to $F$ and 
let $e^\prime$ be the class (in $C''$) containing $e$. 
\item Let 
$K$
be the connected component of $C''$ containing $e^\prime$. 
\end{itemize}
\end{definition}

\begin{problem}[cf. \cite{Itenberg92}, Proposition 3.1]
Does the number (up to equivalence) of connected components of $\Omega^{\br \bff_4}_*$ 
coincide with 
that of vertices of the graph $K$ ?
\end{problem}

\noindent
{\bf Degenerations and the graph $P$.}\ \ \ 
We fix an isometry class of 
integral involutions of type $(S,\theta) \cong (\bu,- \id)$ with $(r(\psi),a(\psi),\delta_\psi)\not=(10,8,0),$ $(10,10,0)$. 
Then we get a real $2$-elementary K3 surface $(X,\tau,\varphi)$ of type $(S,\theta) \cong (\bu,- \id)$ and 
a real {\bf nonsingular} curve $A=s+A_1$ in $|-2K_{\bff_4}|$ where 
$A_1$ is a real nonsingular curve in $|12c+3s|$ on $\br \bff_4$. 

We define the graphs $P$ as follows (See Definition \ref{def-degenerations}). 

\begin{definition}
\begin{itemize}
\item The {\bf vertices} of $P$ are 
all the {\bf rigid isotopy classes} of 
real curves $A^\prime_1$ in $|12c+3s|$ with one nondegenerate double point on $\bff_4$ 
obtained from the degenerations of types 1)---3) of 
the real nonsingular curve $A_1$ with $\varphi$. 

\item Two vertices of $P$ are connected by an {\bf edge} if 
one rigid isotopy class is obtained from the conjunction of the ovals $\E_1$ and 
the non-contractible component of $A_1$ (Conjunction 1)), and 
the other rigid isotopy class is obtained from the contraction of the oval $\E_1$ (Contraction 3)). 

\item Two vertices of $P$ are connected by an {\bf edge} if 
one rigid isotopy class is obtained from the conjunction of the ovals $\E_1$ and $\E_2$ of 
$A_1$ (Conjunction 2)), and 
the other rigid isotopy class is obtained from the contraction of the oval $\E_1$ (Contraction 3)). 
\end{itemize}
\end{definition}


\begin{problem}[cf. \cite{Itenberg92}, Proposition 3.4]
Fix a real {\bf nonsingular} curve $A=s+A_1$ in $|-2K_{\bff_4}|$ where 
$A_1$ is a real nonsingular curve in $|12c+3s|$ on $\br \bff_4$ 
and  $\varphi$ as above. 

Let us get an {\bf arbitrary} 
real curve $A^\prime_1$ in $|12c+3s|$ with one nondegenerate double point on $\bff_4$ 
obtained from one of the degenerations of types 1)---3) of $A_1$ with $\varphi$, and 
construct the graph $K$ (see Lemma \ref{degene-from-nonsing} and Definition \ref{K-def-F4}) from $A^\prime_1$ and $\varphi_-$. 
Then, is the graph $K$ isomorphic to the graph $P$ ?
\end{problem}


\end{document}